\documentclass{amsart}
\usepackage{amssymb,amsmath}
\usepackage{graphicx}
\usepackage{subfigure}
\usepackage{upgreek}
\usepackage{tikz}
\usepackage{tikz-cd}
\usepackage{hyperref}
\usetikzlibrary{matrix}
\usepackage{cmap}
\usepackage[T1]{fontenc}

\usepackage[all]{xy}
\usepackage{optparams,eurosym}
\usepackage{amsthm}

\title[Mayer-Vietoris property for relative symplectic cohomology ]{Mayer-Vietoris property for relative symplectic cohomology }
\author{Umut Varolgunes}

\numberwithin{equation}{subsection}

\newcommand\paperbody%
        {}

\newtheorem{lemma}{Lemma}[subsection]
\newtheorem{proposition}[lemma]{Proposition}
\newtheorem{corollary}[lemma]{Corollary}
\newtheorem{theorem}[lemma]{Theorem}
\newtheorem{non-theorem}{Non-Theorem}

\newtheorem{claim}{Claim}

\newtheorem{definition}{Definition}
\newtheorem{remark}[lemma]{Remark}

\newtheorem{example}{Example}

\newcommand\abs[1]{|#1|}

\numberwithin{equation}{subsubsection}

\begin{document}
\begin{abstract}
In this paper, we construct a Hamiltonian Floer theory based invariant called relative symplectic cohomology, which assigns a module over the Novikov ring to compact subsets of closed symplectic manifolds. We show the existence of restriction maps, and prove some basic properties. Our main contribution is to identify natural geometric conditions in which relative symplectic cohomology of two subsets satisfies the Mayer-Vietoris property. These conditions involve certain integrability assumptions involving geometric objects called barriers - roughly, a one parameter family of rank 2 coisotropic submanifolds. The proof uses a deformation argument in which the topological energy zero (i.e. constant) Floer solutions are the main actors.
\end{abstract}
\maketitle
\section{Introduction}\label{c1}

\subsection{Motivation}

In \cite{S1}, using considerations coming from mirror symmetry, Seidel suggested the following. Let $M$ be a symplectic manifold which admits a Lagrangian torus fibration $M\to B$, possibly with certain nice singularities. Then, it might be possible to define Floer theoretic invariants of certain subsets of the base $B$ with sheaf like properties, such that the global sections of this sheaf are equal to classical Floer theoretic invariants of $M$. 

More specifically, Seidel constructs a candidate invariant for a convex neighborhood of a point in the base of $T^*T^n\to \mathbb{R}^n$ using first a standard ``wrapping'' procedure near the boundary, and then adding certain formal sums to the resulting Floer chain complex. We define a new invariant for compact subsets of a closed symplectic manifold, called relative symplectic cohomology, which extends Seidel's construction in the context of closed string theory. The theory could  be extended to open strings and also to open manifolds that are tame at infinity, but we do not explore these in this paper to stay focused on the new ideas we put forth.

We give a certain implementation of Seidel's suggestion as a Mayer-Vietoris sequence for relative symplectic cohomology. The context of our statements are more general than Lagrangian torus fibrations, but a shadow of integrability still remains in the picture. 

\subsection{Relative symplectic cohomology}\label{c1sprop}

Let $\Lambda:=\{\sum_{i\geq 0} a_iT^{\alpha_i}\mid a_i\in\mathbb{Q}, \alpha_i\in\mathbb{R}, \text{ where } a_i\to\infty \text{, as } i\to\infty\}$ denote the Novikov field. The Novikov ring $\Lambda_{\geq 0}\subset \Lambda$ is the ring consisting  of the formal power series with $\alpha_i\geq 0$ for every $i\geq 0$.

Let $M$ be a closed symplectic manifold. Relative symplectic cohomology $SH_M(K)$ is a $\mathbb{Z}_2$-graded $\Lambda_{\geq 0}$-module assigned to each compact $K\subset M$. $SH_M(K)$ is defined as the homology of a chain complex which depends on additional data, as it often happens in Floer theory.

Relative symplectic cohomology satisfies the following properties.
\begin{itemize}
\item (coordinate independence) Let $\phi:M\to M$ be a symplectomorphism, then there exists a canonical relabeling isomorphism $SH_M(K)\to SH_M(\phi(K))$.
\item (global sections) $SH_M(M)=H(M,\mathbb{Z})\otimes \Lambda_{> 0}$ as $\mathbb{Z}_2$-graded ${\Lambda_{\geq 0}}$-modules, where $\Lambda_{> 0}$ is the maximal ideal of $\Lambda_{\geq 0}$.
\item (empty set) $SH_M(\varnothing)=0$.
\item (restriction maps) For any $K'\subset K$, there are canonical graded module maps, called restriction maps: \begin{align}
SH_M(K)\to SH_M(K').\end{align}
Moreover, if $K''\subset K'\subset K$, the map $SH_M(K)\to SH_M(K'')$ is equal to the composition $SH_M(K)\to SH_M(K')\to SH_M(K'')$.
\end{itemize}
We construct $SH_M(K)$ and prove the properties above in this paper. For further properties (and their proofs), including:
\begin{itemize}
\item (Hamiltonian isotopy invariance of restriction maps) Let $\phi_t:M\to M$, $t\in [0,1]$, be a Hamiltonian isotopy such that $\phi_t(K)\subset K'$ for all $t$. We have a commutative diagram
\begin{align}
\xymatrix{ 
SH_M(K')\ar[r]\ar[dr]& SH_M(\phi_1(K))\ar[d]^{\phi_1^{-1}}\\ &SH_M(K).}
\end{align}
\item (displaceability condition) Let $K\subset M$ be displaceable by a Hamiltonian diffeomorphism, then $SH_M(K)\otimes_{\Lambda_{\geq 0}}\Lambda=0$;
\end{itemize}
as well as a lengthy motivational and historical discussion we refer the reader to author's thesis \cite{U}. 

Let us briefly discuss invariants similar to $SH_M(K)$ from the literature. In their seminal paper, Floer and Hofer constructed an invariant that they called symplectic homology for (bounded!) open subsets of $(\mathbb{R}^{2n},\omega_{st})$ \cite{FH}. This was generalized to aspherical manifolds with contact (or no) boundary in \cite{CFH}. In a more explicit construction, Viterbo defined an intrinsic invariant in the contact boundary case that only depends on the completion of the domain \cite{V}. More recently, Cieliebak-Oancea generalized Viterbo's construction to Liouville cobordisms \cite{CO}. 

Cieliebak et al. also commented that their constructions could be generalized to non-aspherical manifolds by the use of Novikov parameters in Section 5 of \cite{CFH}. It appears that the first time in the literature this was picked up again was in Groman \cite{G}. Groman's definition of reduced symplectic cohomology is very similar to ours, but it is not the same. The invariants of  \cite{Ve} and \cite{Mc} also follow similar patterns. The reader will find more detailed discussions of these references along with the appropriate comparisons in the aforementioned thesis.

%
%

\subsection{Mayer-Vietoris property}\label{c1smayer}

The main task of this paper is to analyze the question: does $SH_M(\cdot)$ satisfy the Mayer-Vietoris property, i.e. for $K_1,K_2$ compact subsets of $M$, is there an exact sequence
\begin{align}\label{c1emv}
\xymatrix{
SH_M(K_1\cup K_2)\ar[r]&SH_M(K_1)\oplus SH_M(K_2)\ar[dl]\\ SH_M(K_1\cap K_2)\ar[u]^{[1]}},
\end{align}where the degree preserving maps are the restriction maps (up to sign)?

A Mayer-Vietoris sequence for their version of symplectic homology, when $K_1$ and  $K_2$ are Liouville cobordisms inside a Liouville domain $M$ satisying a number of conditions (one of them being that their union and intersection is also a Liouville cobordism) was established by Cieliebak-Oancea in Theorem 7.17 of \cite{CO}. The most rudimentary version of our results Theorem \ref{c5tbasic} can be seen as a generalization of theirs. As far as we know this is the first investigation of a symplectic Mayer-Vietoris property where the boundaries of the domains under question intersect non-trivially. 

Mayer-Vietoris property does not hold in general. In Figure \ref{c1fexample}, we see examples of pairs of subsets inside the two sphere that do and do not satisfy Mayer-Vietoris property.

\begin{figure}
\includegraphics[width=\textwidth]{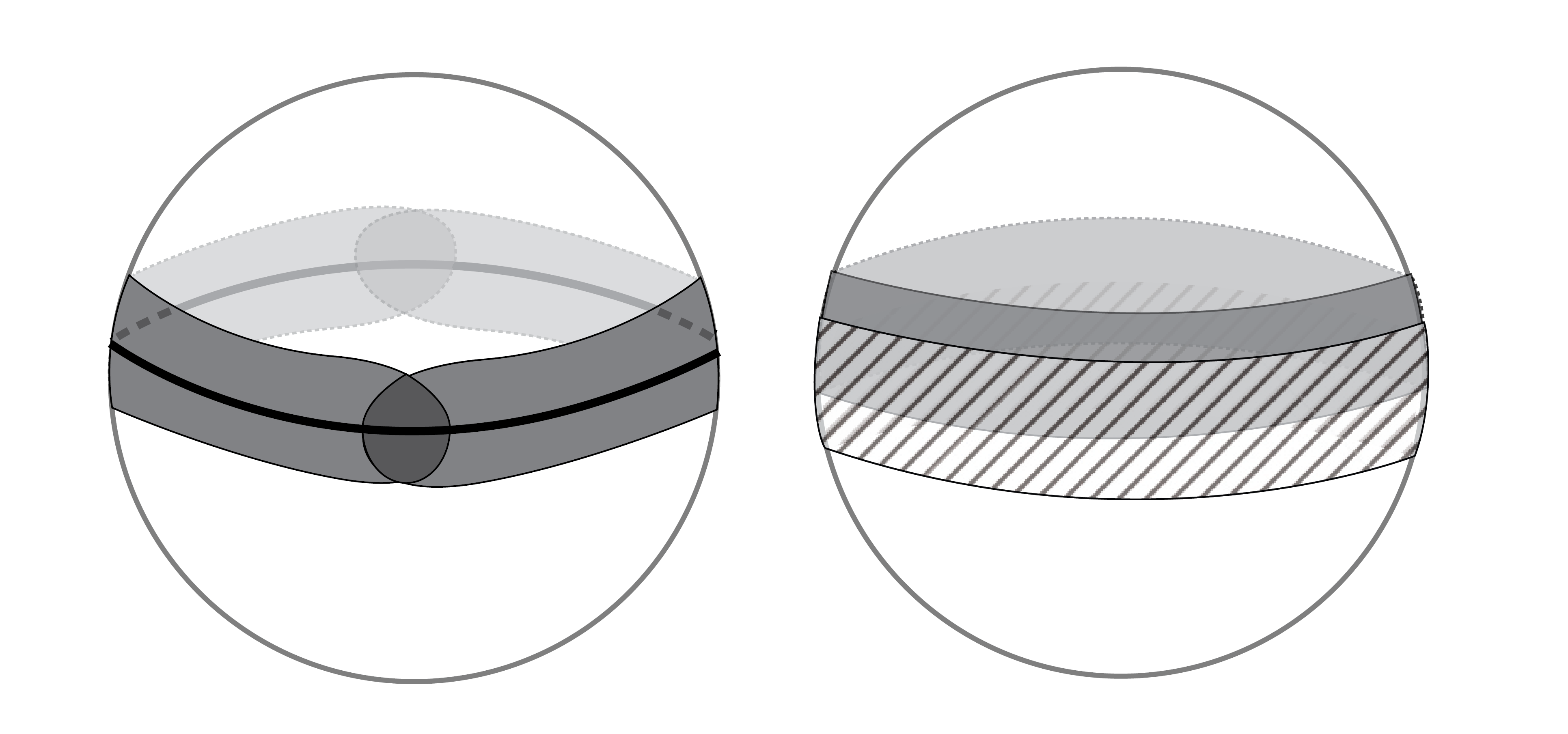}
\caption{On the left there are two subsets that cannot satisfy Mayer-Vietoris, and on the right are two that do. The thick circle on the left divides the sphere into equal areas.}
\label{c1fexample}
\end{figure}

One piece of good news is that we can measure the failure of the Mayer-Vietoris property to hold.  Recall that $SH_M(K)$ is the homology of a chain complex $SC_M(K)$, whose definition depends on additional data. Similarly, given compact subsets $K_1,\dots ,K_n$, we define a chain complex $SC_M(K_1,\dots,K_n)$ as follows. The underlying $\Lambda_{\geq 0}$-module of $SC_M(K_1,\ldots K_n)$ is $$\bigoplus_{I\subset [n]} SC_M\left(\bigcap_{i\in I} K_i\right).$$ Here $[n]=\{1,\ldots ,n\}$ and the empty $I$ means that we take the union of all $K_i$. Consider the $\mathbb{Z}$-grading given by the number of elements in $I$ corresponding to each of the summands. The differential of $SC_M(K_1,\ldots K_n)$ is of the form $$d=d_0+d_1+d_2+\ldots, $$  where $d_i$ increases the degrees by $i$.

The differential $d$ is defined roughly (in particular up to certain signs!) as follows. $d_0$ is the direct sum of the differentials of the summands $SC_M\left(\bigcap_{i\in I} K_i\right)$. $d_1$ is the direct sum of the chain level restriction maps $SC_M\left(\bigcap_{i\in I} K_i\right)\to SC_M\left(\bigcap_{i\in J} K_i\right)$, for every $I\subset J\subset [n]$ with $\abs{J}=\abs{I}+1$. Let us exhibit a way of visualizing this to explain the higher degree terms of $d$.

The summands of $\bigoplus_{I\subset [n]} SC_M\left(\bigcap_{i\in I} K_i\right)$ are placed at the vertices of an $n$-dimensional cube $\{(x_1,\ldots ,x_n)\mid x_j\in [0,1]\}\subset \mathbb{R}^n$, where the summand corresponding to $I\subset [n]$ is at the vertex $(a_1,\ldots a_n)$ of the cube with $a_i=1$ if and only if $i\in I$. We can then think of $d_0$ as a sum over the $0$-dimensional faces of the cube, and $d_1$ as a sum over the $1$-dimensional ones (i.e. edges). Similarly, $d_2$ is a sum over the $2$-dimensional faces. Namely for each $2$ dimensional face $F$, we have two different ways of going from the vertex of $F$ with the minimum number of $0$'s to the one with maximum number of $0$ using the edges of $F$. This corresponds to two different compositions of restrictions maps, and Hamiltonian Floer theory provides us with a chain homotopy between them. This homotopy map is the map we associate to $F$ in the sum defining $d_2$. For every, $i\geq 0$, $d_i$ is similarly a sum of higher homotopical coherence maps over the $i$-dimensional faces of the cube.



We will prove that the homology of $SC_M(K_1,\ldots K_n)$ only depends on $K_1,K_2,\ldots K_n$, therefore the following definition makes sense.

\begin{definition} $K_1,K_2,\ldots K_n$ satisfy \textbf{descent}, if $SC_M(K_1,\ldots K_n)$ is acyclic.
\end{definition} 

Satisfying descent implies the existence of a convergent spectral sequence:\begin{align} \label{c1espectral}
\bigoplus_{\varnothing\neq I\subset [n]} SH_M\left(\bigcap_{i\in I} K_i\right) \Rightarrow SH_M\left(\bigcup_{i=1}^n K_i\right),
\end{align}
which produces a Mayer-Vietoris sequence for $n=2$ as in Equation \ref{c1emv} above.

\begin{remark}
In fact, our results show that under the descent condition there exists a canonical such spectral sequence (up to equivalence of spectral sequences), and, in particular, a canonical connecting homomorphism in our Mayer-Vietoris sequences. This is a byproduct of the well-definedness of relative symplectic cohomology for multiple subsets. Let us also remark that, conversely, for the statements asserting merely the existence of a Mayer-Vietoris sequence (which we focus on for simplicity), the discussion of relative symplectic cohomology for two or more subsets is unnecessary, see Remark \ref{c5rmultiple}.
\end{remark}

\begin{definition}
Let $Z^{2n-2}$ be a closed manifold. We define a \textbf{barrier} to be an embedding $Z\times [-\epsilon,\epsilon]\to M^{2n}$, for some $\epsilon >0$, where $Z\times \{a\}\to M$ is a coisotropic for all $a\in [-\epsilon,\epsilon]$. We call the image of $Z\times\{0\}$ the \textbf{center} of the barrier, and the vector field obtained by pushing forward $\partial_{\epsilon}\in\Gamma(Z\times \{0\},T(Z\times (-\epsilon,\epsilon))\mid_{Z\times \{0\}}) $  to $M$ the \textbf{direction} of the barrier.
\end{definition}

We use the phrase \textbf{compact domain} to mean a compact submanifold with boundary of codimension $0$.

\begin{theorem}\label{c1tmv}
(Mayer-Vietoris sequence) Let $K_1, K_2\subset M$ be compact domains. Assume that $\partial K_1$ and $\partial K_2$ transversally intersect along a rank 2 coisotropic which, if non-empty, is the center of a barrier whose direction points out of $K_1$ and $K_2$. Then, $K_1$ and $K_2$ satisy descent. Therefore, we have an exact sequence:
\begin{align}
\xymatrix{
SH_M(K_1\cup K_2)\ar[r]&SH_M(K_1)\oplus SH_M(K_2)\ar[dl]\\ SH_M(K_1\cap K_2)\ar[u]^{[1]}},
\end{align}
where the degree preserving maps are the restriction maps (up to signs).
\end{theorem}

We made the assumption that $K_1, K_2\subset M$ are domains purely for the sake of keeping the statement simple. For the actual statement see Theorem \ref{c5tmv}. Note that in dimension $2$, the condition is equivalent to boundaries not intersecting, as a point in a surface can never be coisotropic (see Figure \ref{c1fexample}). In dimension $4$, it implies that the intersection is a disjoint union of Lagrangian tori, but unfortunately being outward pointing is an extra condition in this case, see Corollary \ref{c5ctorus}. 

%
%

\begin{definition} An \textbf{involutive map} is a smooth map $\pi: M\to B$ to a smooth manifold $B$, such that for any $f,g\in C^{\infty}(B)$, we have $\{f\circ\pi,g\circ\pi\}=0$
\end{definition}
\begin{remark}
The most studied examples of involutive maps are Lagrangian fibrations. These correspond to the case where the non-empty smooth fibers of $\pi$ has half the dimension of $M$ (which is the least they can be).
\end{remark}
\begin{theorem}\label{c1tinvolutive}
Let $\pi:M\to B$ be an involutive map, and $X_1,\ldots X_n$ be closed subsets of $B$. Then $\pi^{-1}(X_1),\ldots \pi^{-1}(X_n)$ satisfy descent.
\end{theorem}


We obtain Theorem \ref{c1tinvolutive} as a corollary of Theorem \ref{c5tmv2}. 

The following corollary of Theorem \ref{c1tinvolutive} (generally referred to as the Stem theorem) was first proven by Entov-Polterovich using a completely different set of tools \cite{EP}. 

\begin{theorem}\label{stem}
Any involutive map admits at least one fiber that is not displaceable by Hamiltonian isotopy.
\end{theorem}
\begin{proof}
We refer to the properties of $SH_M(K)$ by the names given to them in Section \ref{c1sprop}. Let $\bigcup C_i$ be any finite cover of the image of $M$ inside $B$ by compact subsets. Theorem \ref{c1tinvolutive}, and the global sections property (which in particular implies $SH_M(M)\otimes\Lambda\neq 0$) shows that $SH_M(\pi^{-1}(\bigcap_{J}C_i)\otimes\Lambda\neq 0$, for some non-empty $J\subset [n]$, by the spectral sequence \ref{c1espectral}. For the reader unfamiliar with spectral sequences, we note that this conclusion can also be reached by assuming the contrary, and using the Mayer-Vietoris sequence iteratively to reach a contradiction to $SH_M(M)\otimes\Lambda\neq 0$. 

Hence, by the displaceability condition, $\pi^{-1}(C_i)$ is not displaceable for some $i$. Now assuming that each fiber is displaceable easily leads to a contradiction, as it implies that a sufficiently small open neighborhood of the fiber is also displaceable, and hence provides a finite cover $\bigcup C_i$ by compact subsets such that each $\pi^{-1}(C_i)$ is displaceable, using compactness of the image of $\pi$.
\end{proof}

\begin{remark}
Even though the tools are different, the logic of our proof is similar to \cite{EP} as the experts will notice. We also refer the reader to \cite{EP} for a more detailed exposition of the corollary above including many interesting examples.
\end{remark}

\subsection{A remark on relative open string invariants}

Let $L\subset M$ be a closed aspherical Lagrangian (one can be a lot less restrictive, but we choose to be brief here). Replacing Hamiltonian Floer theory of closed orbits wth Lagrangian Floer theory of chords with endpoints on $L$, we immediately obtain a relative invariant $HF_L(K)$, for any compact subset $K$.  We leave the discussion of this invariant to an upcoming paper, but we would like to advertise one result:
\begin{theorem}
Any involutive map admits at least one fiber that is not displaceable from $L$ by Hamiltonian isotopy.
\end{theorem}
This open string version of the Stem theorem seems to be new. Its proof only notationally differs from the one of Theorem \ref{stem}.

\subsection{Outline of the paper}\label{c1soutline}

In Section \ref{c2}, we introduce the algebraic framework that will be used the later sections. \ref{c2shomotopical} is a lengthy subsection, where we discuss the homotopical algebra of certain diagrams composed of cube-shaped building blocks. We discuss the cone, telescope and composition operations in an entirely explicit fashion. In the sequel \ref{c2srays}, we prove some homology level statements regarding telescopes. In \ref{c2scompletion}, we recall the notions of completion and completeness for modules over the Novikov ring, and discuss their interaction with taking homology of chain complexes. We end with a short list of results which combine ingredients from the previous subsections in \ref{c2sacyclic}.

In Section \ref{c3}, first, we list our conventions for Hamiltonian Floer homology in \ref{c3sconventions}, and review Hamiltonian Floer theory in \ref{c3shamiltonian}. It takes us some effort to spell out our statements regarding ``contractibility" of the space of Hamiltonians, when a certain monotonicity property is imposed on the allowed families. In \ref{c3sconstruction}, we define relative symplectic cohomology, and show its basic properties as listed in the Introduction. In the last subsection (\ref{c3smultiple}), we introduce relative symplectic cohomology of multiple compact subsets.

Section \ref{c5} is where we discuss the Mayer-Vietoris/descent properties. We focus on the homology level statement for two subsets (i.e. the Mayer-Vietoris sequence) until the last subsection for better readability. In \ref{c5szero}, we reduce the problem to showing the existence of a sequence of (pairs of) Hamiltonians that can be used as cofinal sequences for our subsets, which satisfy a dynamical property. In \ref{c5sboundary}, we explain a controlled way of choosing acceleration data, and immediately show the Mayer-Vietoris property for two domains with non-intersecting boundary in \ref{c5snon}. Subsections \ref{c5sbarriers} and \ref{c5snond} introduce and motivate barriers and some relevant notions. In \ref{c5sproof} we prove our main theorem (Theorem \ref{c5tmv}). In \ref{c5sinstances}, we give examples of barriers. In the last subsection (\ref{c5sinvolutive}), after generalizing the main theorem slightly (Theorem \ref{c5tmv2}), we show the descent result for multiple subsets that are preimages of involutive maps. 

In Appendix \ref{appa}, we establish the easy translation from Pardon's simplicial diagrams to our cubical ones. Finally, in Appendix \ref{appb}, we show that the descent property for $n>2$ subsets follows from the descent property of a number of pairs of subsets.

\subsection{Acknowledgements} The first and foremost thanks go to my PhD advisor Paul Seidel, for suggesting the problem, and numerous enlightening discussions. I thank Francesco Lin, Mark McLean, and John Pardon for helpful conversations. I also thank the anonymous referee for carefully reading the paper, and suggesting lots of improvements to increase its readability. This work was partially supported by NSF grant 1500954 and the Simons Foundation (through a Simons Investigator award).

\section{Algebra preparations}\label{c2}

It is quite likely that all of the ideas, statements, or constructions that appear in this section are known to the experts (with the potential exception of explicit signs in the formulas of Section \ref{c2shomotopical}). That being said, we were not able to locate sources where the material in this section is presented in a fashion that we could reuse for our purposes in the later sections.

Throughout this section, we assume that our chain complexes are $\mathbb{Z}/2$-graded. However, whenever there is a $\mathbb{Z}/N$ or $\mathbb{Z}$-grading available, our statements can be modified to take into account those gradings without a problem. In fact, the reader might find it helpful to think of a $\mathbb{Z}/2$-graded chain complex as a $2$-periodic $\mathbb{Z}$-graded chain complex in order to translate statements from the homological algebra literature, which are usually stated in the $\mathbb{Z}$-graded context.

Let $A$ be set. An $m$-\textbf{tuple} of elements of $A$ is a sequence of length $m$ consisting of elements of $A$. We will denote an $m$-tuple either as a vector $(a_1,\ldots a_n)$ or as a word $a_1a_2\ldots a_n$. A \textbf{subtuple} of a tuple is simply a subsequence of the sequence that defines the tuple.

\subsection{Homotopical constructions}\label{c2shomotopical}
In this section, we explain a way of dealing with homotopy coherent diagrams that Floer theory naturally produces. We restrict ourselves to diagrams that are that are made out of $n$-cubes (Definition \ref{cubedef}) in very restricted ways, rather than introducing the machinery to deal with diagrams which are indexed by arbitrary cubical or simplicial sets. The reason is simply that the latter route did not seem to help much for our actual goals.

Constructions involving $n$-cubes were also used in the celebrated work \cite{KM} of Kronheimer-Mrowka to obtain their spectral sequences (see their Theorem 6.8 and the surrounding discussion.) Specifically, the contents of our Sections \ref{c2sspositive} and \ref{c2sscones} are related to their work.

\subsubsection{Cubes}\label{c2sscubes}

 Consider the standard unit cube: $Cube^n:=\{(x_1,\ldots ,x_n)\mid x_j\in [0,1]\}\subset \mathbb{R}^n$. We stress that the ordering of the coordinates will play an important role in what follows. For $0\leq k\leq n$, a $k$-\textbf{dimensional face} of $Cube^n$ is any subset of $Cube^n$ given by setting $n-k$ of the coordinates to either $0$ or $1$. Therefore, the faces of $Cube^n$ are in one-to-one correspondence with the set of $n$-tuples of elements of $\{0,1,-\}$: \begin{align}
\{(i_1,\ldots ,i_n)\mid i_k\in\{0,1,-\}\},
\end{align}where $-$ represents the coordinates that vary in the face. Let us denote this assignment by $F\mapsto \mu(F)$, for $F$ a face of $Cube^n$. 

Let us call a vertex (i.e. a $0$-dimensional face of $Cube^n$) contained in a face the \textbf{initial} vertex if it has the maximum number of zeros, and \textbf{terminal} if it has the maximum number of ones in its coordinates. We denote the initial vertex of a face $F$ by $\nu_{in(F)}$, and its terminal vertex by $\nu_{ter(F)}$.

Let us call two faces $F'$ and $F''$ \textbf{adjacent} if the terminal vertex of $F'$ equals the inital vertex of $F''$. We denote this relationship by $F'>F''$.  We say that two adjacent faces form a \textbf{boundary of a face} $F$ if $F$ is the smallest face that contains both $F'$ and $F''$. In this case, we define $v(F',F)$ to be the sub-tuple of $\mu(\nu_{ter(F')})-\mu(\nu_{in(F')})$ (as an $n$-tuple of $0$'s and $1$'s) corresponding to the $dim(F)$-entries given by the $-$'s in $\mu(F)$. 

For example, if $\mu(F')=0-10-$ and $\mu(F'')=-1101$, then they are adjacent at the vertex $(0,1,1,0,1)$, and they form a boundary of $\mu(F)=--10-$. Moreover, $\mu(\nu_{ter(F')})-\mu(\nu_{in(F')})=01001$, and finally, $v(F',F)=011$.

Let $A$ be a set. If $w$ and $v$ are tuples of elements of $A$, define $\#(w,v)$ to be the number of sub-tuples of $v$ that are equal to $w$. For example for $A=\{0,1\}$, $\#(11,101)=1$.

\begin{definition}\label{cubedef}Let $R$ be a commutative ring. We define an \textbf{$n$-cube (of chain complexes)} over $R$ in the following way. To each $0$-dimensional face $\nu$ (i.e. vertices) of $Cube^n$ we associate a $\mathbb{Z}/2$-graded $R$-module $C^{\nu}$, and for any $k$-dimensional face (including $k=0$) $F$ we give maps $f_F: C^{\nu_{in(F)}}\to C^{\nu_{ter(F)}}$ from its initial vertex to its terminal vertex, of degree $dim(F)+1$ modulo $2$. 

These maps are required to satisfy the following relations. For each face $F$ we have:
\begin{align}\label{c2ecube}
\sum_{F'>F'' \text{is a bdry of }F}(-1)^{*_{F',F}} f_{F''}f_{F'}=0,
\end{align}
where $*_{F',F}=\#(1,v(F',F))+\#(01,v(F',F))$. 
\end{definition}

Note that for any $k$-dimensional face $F$ of $Cube^n$, we canonically obtain a $k$-cube by only remembering the data of the $n$-cube corresponding to the faces contained inside $F$.

\begin{remark}
In the $\mathbb{Z}$-graded context the only difference in the definition would be that the chain complexes are $\mathbb{Z}$-graded and maps have degree $dim(F)+1$.
\end{remark}

In Figure \ref{c2fcube} we present a $3$-cube to illustrate the definition. At the corners there are chain complexes (differential is a degree $1$ map), at the edges chain maps (degree $0$), at the square faces homotopies between the two different ways of going between the initial and terminal vertices of that square (degree $1$), and lastly at the codimension 0 face we have one map $H$ (degree $0$) that satisfies:
\begin{align}\label{c2e3cube}
-g^{100}+g^{010}-g^{001}+g^{011}-g^{101}+g^{110}-dH+Hd=0,
\end{align}
where $g^{100}$ is the composition $C^{000}\to C^{100}\to C^{111}$ (the second map is the homotopy) etc.
\begin{figure}[!h]\label{c2fcube}
\centering
\begin{tikzpicture}
  \matrix (m) [matrix of math nodes, row sep=3em,
    column sep=3em]{
    & C^{000} & & C^{100}  \\
    C^{010}  & & C^{110}  & \\
    & C^{001}  & & C^{101}  \\
    C^{011}  & & C^{111}  & \\};
  \path[-stealth]
    (m-1-2) edge (m-1-4) edge (m-2-1)
            edge [densely dotted] (m-3-2) edge (m-2-3)
            edge [densely dotted] (m-4-1) 
            edge [densely dotted] (m-3-4)
            edge [densely dotted] (m-4-3)
    (m-1-4) edge (m-3-4) edge (m-2-3) edge (m-4-3)
    
    (m-2-1) edge [-,line width=6pt,draw=white] (m-2-3)
            edge (m-2-3) edge (m-4-1) edge (m-4-3)
    (m-3-2) edge [densely dotted] (m-3-4)
            edge [densely dotted] (m-4-1)
            edge [densely dotted] (m-4-3)
    (m-4-1) edge (m-4-3)
    (m-3-4) edge (m-4-3)
    (m-2-3) edge [-,line width=6pt,draw=white] (m-4-3)
            edge (m-4-3);
\end{tikzpicture}
\caption{A $3$-cube}
\end{figure}
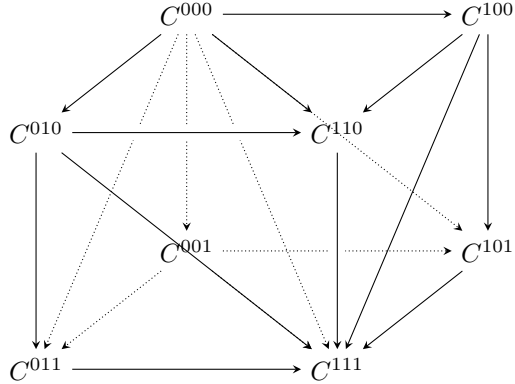

\subsubsection{Maps between $n$-cubes}\label{c2ssmaps} 
A \textbf{partially defined} $n$-cube is one where we have chain complexes at the vertices of $Cube^n$, and maps for some of the faces specified so that whenever it makes sense Equation \ref{c2ecube} is satisfied. An $n$-cube that agrees with a \textbf{partially defined} $n$-cube wherever they are both defined is called a \textbf{filling}.

We define a \textbf{map} between two $n$-cubes to be a filling of the partially defined $(n+1)$-cube where the $n$-dimensional faces $\{x_{n+1}=0\}$ and $\{x_{n+1}=1\}$ are the given $n$-cubes. Here we are using the data of the ordering of the coordinates of $Cube^{n+1}$. Given $n$-cubes $\mathcal{C}$ and $\mathcal{C}'$, a map between them will simply be represented as follows:
\begin{align}\mathcal{C}\to\mathcal{C}'\end{align}
We stress that this diagram is an $n+1$-cube, even though it is drawn as a $1$-cube, and also that $\mathcal{C}$, and $\mathcal{C}'$ represent the $n$-cubes which correspond to the faces $\{x_{n+1}=0\}$, and $\{x_{n+1}=1\}$, respectively. In this situation, we would say that the new coordinate is added as the last or $(n+1)$th coordinate. 

\begin{lemma}
Let $\mathcal{C}$ be an $n$-cube. Then, the following defines an $(n+1)$-cube: the $n$-dimensional faces $\{x_{n+1}=0\}$ and $\{x_{n+1}=1\}$ are both $\mathcal{C}$, these two $n$-cubes are connected to each other at the edges with identity maps, and all the homotopies are defined to be zero.
\end{lemma}

\begin{proof}
It suffices to check the equation corresponding to the top dimensional face of $Cube^{n+1}$. Let $f$ be the map in $\mathcal{C}$ associated to the top dimensional face of $Cube^{n}$. Then the equation that we want to verify is $$(-1)^af+(-1)^bf=0,$$ where $a$, and $b$ are the sum of the number of $1$ and $01$ subtuples in $00\ldots 01$, and $11\ldots 10$, respectively (both tuples have $n+1$ elements). This is easy to check.
\end{proof}

We call a map as in the Lemma the $id$ map of $\mathcal{C}$: $$\mathcal{C}\overset{id}{\to} \mathcal{C}.$$

\begin{remark}\label{idremark} Even if we added the new coordinate with the identity maps in a different place than the last in the ordering, we would get an $(n+1)$-cube (exercise for the reader), but we would not call this a map from $\mathcal{C}$ to $\mathcal{C}$.
\end{remark}
A \textbf{homotopy} of two maps of $n$-cubes is a filling of the partially defined $(n+2)$-cube where the faces $\{x_{n+1}=0\}$ and $\{x_{n+1}=1\}$ are the given maps of $n$-cubes; and  the faces $\{x_{n+2}=0\}$ and $\{x_{n+2}=1\}$ are the identity maps for the given $n$-cubes:
\begin{align}
\xymatrix{ 
\mathcal{C}\ar[r]\ar[d]^{id}\ar[dr]& \mathcal{C}'\ar[d]^{id}\\ \mathcal{C}\ar[r] &\mathcal{C}'.}
\end{align}

Here the diagonal arrow represents the filling that is in the definition. Let us call an $(n+2)$-cube of such form an $(n+2)$-\textbf{slit}. Note that if there is a homotopy from one map to the other, then there is also a homotopy in the other direction given by negating all the maps in the filling.

A \textbf{triangle} of maps of $n$-cubes is a filling of the partially defined $(n+2)$-cube given by a triple of $n$-cubes $\mathcal{C}$, $\mathcal{C}'$ and $\mathcal{C}''$; and maps $\mathcal{C}\overset{f}{\to}\mathcal{C}'$, $\mathcal{C}'\overset{f'}{\to}\mathcal{C}''$, and $\mathcal{C}\overset{g}{\to}\mathcal{C}''$ as below (again the diagonal arrow represents the filling):  
\begin{align}
\xymatrix{ 
\mathcal{C}\ar[r]^{id}\ar[d]_f\ar[dr]& \mathcal{C}\ar[d]^{g}\\ \mathcal{C}'\ar[r]_{f'} &\mathcal{C}'',}
\end{align} the faces $\{x_{n+1}=0\}$ and $\{x_{n+1}=1\}$ are $f$ and $g$; and $\{x_{n+2}=0\}$ and $\{x_{n+2}=1\}$ are $id$ and $f'$. Let us call an $(n+2)$-cube of such form an $(n+2)$-\textbf{triangle}.

We now give examples of these definitions in low dimensions. 

Let $\mathcal{C}:=C_0\overset{c}{\to} C_1$ and $\mathcal{C}':=C_0'\overset{c'}{\to} C_1'$ be $1$-cubes. A map $\mathcal{C}\to \mathcal{C}'$ is the data of $f_0,f_1, h$ below: \begin{align}\label{c2dcube}
\xymatrix{ 
C_0\ar[r]^{c}\ar[d]_{f_0}\ar[dr]^h& C_1\ar[d]^{f_1}\\ C'_0\ar[r]_{c'} &C'_1}
\end{align} such that $f_0$, $f_1$ are chain maps, and $c'f_0-f_1c+hd+dh=0$. In other words, we have a $2$-cube such that the coordinate of the vertex of $C_1$ is $(1,0)$.

Let $f'_0,f'_1, h'$ be another such map. Then a homotopy from the first triple to the second one (primed ones) would be given by $F_0$, a chain homotopy between $f_0$ and $f'_0$, similarly $F_1$, and also an $H$ that satisfies the equation that is associated to the maximal face of the $3$-cube:
\begin{align}
c'F_0-F_1c+h'-h-dH+Hd=0,
\end{align} as a special case of Equation \ref{c2e3cube}.

Finally consider the following homotopy commutative triangle:
\begin{align}
\xymatrix{ 
C_0\ar[r]^{f_0}\ar[dr]_g& C_1\ar[d]^{f_1}\\ &C_2,}
\end{align}
and $h:C_0\to C_2$ such that\begin{align}
g-f_1f_0+dh+hd=0.
\end{align}

We are thinking of this data as the following $2$-cube (with $C_1$ sitting at the vertex with coordinates $(1,0)$):
\begin{align}
\xymatrix{ 
C_0\ar[d]^{id}\ar[r]^{f_0}\ar[dr]^h& C_1\ar[d]^{f_1}\\ C_0\ar[r]^{g}&C_2.}
\end{align}

\subsubsection{$n$-cubes with positive signs}\label{c2sspositive}

We now introduce a slight variant of $n$-cubes, only to help us define the cone operation on $n$-cubes in the next section.

A \textbf{$n$-cube with positive signs} is defined exactly as an $n$-cube was defined with one exception: the signs in the Equation \ref{c2ecube} are all +1, in other words for every face $F$ of $Cube^n$, we now have: \begin{align}
\sum_{F'>F'' \text{is a bdry of }F} f_{F''}f_{F'}=0.
\end{align}

\begin{lemma}\label{c2lsigns}
Consider an $n$-cube defined by the $\mathbb{Z}/2$-graded modules $C^{\nu}$, and maps $f_F$ as in the Definition \ref{cubedef}. There exists a canonical way of multiplying each map $f_F$ by $+1$ or $-1$, i.e. $f_F\mapsto (-1)^{n(F)}f_F$, so that $C^{\nu}$ and $(-1)^{n(F)}f_F$ define an $n$-cube with positive signs.
\end{lemma}

\begin{proof}
We define $$n(F):=\#(0-,\mu(F))+\#(0,\mu(F)),$$ for every face $F$. 

Let $F'>F''$ be a boundary of $F$. We will show that the parity of $$n(F')+n(F'')+*_{F',F}$$ only depends on $F$ (and not on $F'$ or $F''$), where $*_{F',F}$ was defined in Definition \ref{cubedef}. This clearly proves the claim as the overall sign can be cancelled from the equation associated to $F$.

 Let $S\subset \{1,2,\ldots ,n\}$ be the set of entries of $\mu(F)$ that are equal to $-$. Then, there is a subset $S'\subset S$ such that $\mu(F')$ is obtained by changing the entries of $S$ corresponding to $S'$ to $0$, and $\mu(F'')$ is obtained by changing the entries corresponding to $S-S'$ to $1$.

It follows from the discussion in the previous paragraph that the parity of $\#(0-,\mu(F))+\#(0-,\mu(F'))+\#(0-,\mu(F''))$ is equal to $\#(01,v(F',F))$. It also follows that $\#(1,v(F',F))+\#(0,\mu(F'))+\#(0,\mu(F''))$ has the same parity as $dimF$.

This shows that the parity of $n(F')+n(F'')+*_{F',F}$ is equal to the parity of $\#(0-,\mu(F))+dimF$, finishing the proof.
\end{proof}

The operation described in the lemma, which inputs an $n$-cube and outputs an $n$-cube with positive signs, clearly admits an inverse given by the same formula.
 
\subsubsection{Cones of $n$-cubes}\label{c2sscones} Recall that the usual cone operation  takes a chain map (i.e. a $1$-cube) between two chain complexes, and outputs a single chain complex (a $0$ cube):\begin{align}
\Bigg((C,d)\xrightarrow{f} (C',d')\Bigg) \mapsto \Bigg[C[1]\oplus C',\begin{pmatrix} 
-d & 0 \\
f   & d'
\end{pmatrix}\Bigg].
\end{align}

This can be generalized to all cubes. First, given an $n$-cube with positive signs and one of the $n$ directions, we explain how to construct an $(n-1)$-cube with positive signs with the cone construction.  

If $w$ is an $(n-1)$-tuple of elements from a set $S$, $1\leq i\leq n$ an integer, and $a\in S$, we let $(w,i,a)$ be the $n$-tuple with $a$ inserted as the $i$th entry to $w$. For example if $S=\{0,1,2\}$, then $((0,2,2),2,1)=(0,1,2,2)$.

Let $(\{C^{\nu}\},\{f_F\})$ be an $n$-cube with positive signs, and $1\leq i\leq n$ an integer. The cone of $(\{C^{\nu}\},\{f_F\})$ in direction $i$ is defined by:\begin{align}
C^w=C^{(w,i,0)}[1]\oplus C^{(w,i,1)},
\end{align} for every vertex $w$ of $Cube^{n-1}$, and $f_F:C_{in(F)}\to C_{ter(F)}$ is given by the matrix,
\begin{align}
\begin{pmatrix} 
f_{(\mu(F),i,0)} & 0 \\
f_{(\mu(F),i,-)}   & f_{(\mu(F),i,1)}
\end{pmatrix},
\end{align} for every face $F$ of $Cube^{n-1}$. It is readily seen that this defines an $(n-1)$-cube with positive signs.

If $\mathcal{C}=(\{C^{\nu}\},\{f_F\})$ is an $n$-cube, and $1\leq i\leq n$ an integer. The \textbf{cone} of $\mathcal{C}$ in direction $i$: $$cone^i(\mathcal{C})$$ is defined by the composition of operations:\begin{align}
n-\text{cube} \to n-\text{cube with p. signs}\to (n-1)-\text{cube with p. signs}\to (n-1)-\text{cube}
\end{align}
Here the first arrow is the operation described in Lemma \ref{c2lsigns}, the second one is the cone operation on cubes with positive signs we just introduced, and the third one is the inverse to the operation in Lemma \ref{c2lsigns}. We stress that the signs in the formulas will be different for different directions. 

Unless otherwise stated, when we do a cone operation, we mean that it is applied to an $n$-cube. Occasionally, if we apply a cone operation in a certain direction $i$, we will say that direction $i$ is \textbf{contracted}.

\begin{lemma}\label{c2lcones}
\begin{enumerate}
\item Let $\mathcal{C}$ be an $n$-cube, and let $1\leq i<j \leq n$ be integers. First contracting the $i$th direction and then the $(j-1)$th direction results in a chain complex that is canonically identified with the chain complex obtained by contracting first the $j$th and then the $i$th direction. Note that after $i$ is contracted the $j$th direction becomes the $(j-1)$th direction for the $(n-1)$-cube $cone^i(\mathcal{C})$.
\item The cone operation in a direction $d$ other than the last one sends a map $\mathcal{C}\xrightarrow{} \mathcal{C}$ (considered as an $n$-cube) to a map $cone^d(\mathcal{C})\xrightarrow{} cone^d(\mathcal{C})$
\item The cone operation in a direction $d$ other than the last one sends the map $\mathcal{C}\xrightarrow{id} \mathcal{C}$ to $cone^d(\mathcal{C})\xrightarrow{id} cone^d(\mathcal{C})$
\item Cones in directions except the last two sends $(n+2)$-slits to $(n+1)$-slits, and $(n+2)$-triangles to $(n+1)$-triangles.
\end{enumerate}
\end{lemma}
\begin{proof}

\begin{enumerate}
\item First, note that this statement would be obvious for cone operation on the $n$-cubes with positive signs. The statement then follows from observing that if we want to compose two cone operations, we can instead apply the operation of Lemma \ref{c2lsigns}, then compose the two cone operations on cubes with positive signs, and apply the inverse of Lemma \ref{c2lsigns}. In other words, the two sign changes in the end of the first cone, and in the beginning of the second cone cancel each other.
\item It is enough to prove that for any $(n+1)$-cube $\mathcal{D}$, the sign change operation of Lemma \ref{c2lsigns} satisfies the following property:
\begin{itemize}
\item Let $F$ be a face of $Cube^{n+1}$ so that the last entry of $\mu(F)$ is zero, and let $\tilde{F}$ be the face with that last entry changed to $1$. Then $n(F)=n(\tilde{F})+1$.
\end{itemize}This follows immediately from the definition. Using it in the first and third steps (the ones where signs change) of the construction of the cone we finish the proof.
\item In addition to (2), we also need:
\begin{itemize}\item Let $F$ be an edge of $Cube^{n+1}$ parallel to the $(n+1)$-direction, i.e. $\mu(F)$ has $0$'s and $1$'s in its first $n$ entries and a $-$ in the last one. Then $n(F)$ is even.\end{itemize}
This is also easy to check.
\item Follows from (3).\end{enumerate}\end{proof}

We will refer to the fact that the cone operation turns an $n$-cube into an $(n-1)$-cube, in such way that the properties (2), (3) and (4) of Lemma \ref{c2lcones} holds, the \textbf{functoriality of the cone operation}.

Let us do an example. There are two cones of the $2$-cube in Diagram \ref{c2dcube} (called $\mathcal{C}$): one that contracts the direction parallel to $f$'s $cone^f(\mathcal{C})$, and the one that contracts $c$'s $cone^c(\mathcal{C})$. Let us write them down explicitly.

\begin{align}cone^f(C)=\Bigg[C_0[1]\oplus C_0', \begin{pmatrix} 
-d & 0 \\
-f_1   & d
\end{pmatrix}\Bigg] \xrightarrow{{\begin{pmatrix} 
-c & 0 \\
h   & c'
\end{pmatrix}}} \Bigg[C_1[1]\oplus C_1', \begin{pmatrix}
-d & 0 \\
f_2   & d
\end{pmatrix}\Bigg]
\end{align}

\begin{align}cone^c(C)=\Bigg[C_0[1]\oplus C_1, \begin{pmatrix} 
-d & 0 \\
c   & d
\end{pmatrix}\Bigg] \xrightarrow{{\begin{pmatrix} 
f_1 & 0 \\
h   & f_2
\end{pmatrix}}} \Bigg[C_0'[1]\oplus C_1', \begin{pmatrix}
-d & 0 \\
c   & d
\end{pmatrix}\Bigg]
\end{align}

In both cases, taking the cone in the remaining direction results in $C_0\oplus(C_1[1]\oplus C_0'[1])\oplus C_1'$ with differential: \begin{align}\begin{pmatrix} 
d & 0& 0 &0 \\
-c   & -d &0&0\\
f_1& 0& -d&0\\
h & c'& f & d
\end{pmatrix}.\end{align}

We now make a final point about the cone operation. As long as one remembers the direct sum decompositions of the chain complexes the cone operation does not actually lose any information. Let us explain this further.

Let $(\{C^{\nu}\},\{f_F\})$ be as in Definition \ref{cubedef}, but we do not require them to satisfy the Equations \ref{c2ecube}. Let us call this an $n$-\textbf{precube}. Using the the same formulas as above (the ones for $n$-cubes) we can define the cone of an $n$-precube in any direction $1\leq i\leq n$ as an $(n-1)$-precube.

Let us also say that an $n$-precube is in \textbf{coniform} if for every vertex $\nu$ of $Cube^n$, we are given a direct sum decomposition $C^{\nu}=C^{(\nu,i,0)}[1]\oplus C^{(\nu,i,1)}$ such that $f_F$ is lower triangular for every face $F$ of $Cube^n$.

The following lemma is true by construction.
\begin{lemma}\label{decone} Consider the operation $cone^i$ as a map from the set of $n$-precubes to the set of $(n-1)$-precubes in coniform, for $1\leq i\leq n$. Then, the following are true.
\begin{itemize}
\item $cone^i$ is a bijection.
\item An $n$-precube $\mathcal{C}$ is an $n$-cube if and only if $cone^i(\mathcal{C})$ is a an $(n-1)$-cube.
\end{itemize}
\qed
\end{lemma}

\subsubsection{Composing $n$-cubes}\label{c2sscomposing}
The composition of two chain maps is a chain map. We generalize this construction to higher dimensional cubes.

Let us start with $2$-cubes. Let the two squares below be commutative up to the given homotopies.
\begin{align}
\xymatrix{ 
C_0\ar[r]^{c_0}\ar[d]_{f_0}\ar[dr]_{g_0}& C_1\ar[d]_{f_1}\ar[r]^{c_1}\ar[dr]_{g_1} &C_2\ar[d]^{f_2}\\ D_0\ar[r]_{d_0} &D_1\ar[r]_{d_1}&D_2}
\end{align} In this case, we say that the two $2$-cubes are glued along $f_1$ a $1$-cube.

We can define the composite $2-cube$: 
\begin{align}
\xymatrix{ 
C_0\ar[r]^{c_1c_0}\ar[d]_{f_0}\ar[dr]^{G}&C_2\ar[d]^{f_2}\\ D_0\ar[r]_{d_1d_0}&D_2}
\end{align} where $G=g_1c_0+d_1g_0$.

More generally, let $\mathcal{C}$ and $\mathcal{C}'$ be two $n$-cubes. Assume that the $(n-1)$-cubes associated to the $\{x_k=1\}$ face of $\mathcal{C}$ and the $\{x_k=0\}$ face of $\mathcal{C}'$ are the same, for some $1\leq k\leq n$. Then, we say that $\mathcal{C}$ and $\mathcal{C}'$ can be \textbf{glued} in the $k$th direction. Note that if $\mathcal{C}$ and $\mathcal{C}'$ can be glued, this does not mean in any way that $\mathcal{C}'$ and $\mathcal{C}$ can be glued too, i.e. this is not a symmetric relation.

Let $\mathcal{C}\to\mathcal{C'}$ and $\mathcal{C'}\to\mathcal{C''}$ be two $n$-cubes, which are maps of $(n-1)$-cubes. This means that these two cubes can be glued in the $n$th direction. We will in the future represent this situation by the diagram: $$\mathcal{C}\to\mathcal{C'}\to\mathcal{C''},$$ and say that  $\mathcal{C}\to\mathcal{C'}$ and $\mathcal{C'}\to\mathcal{C''}$ can be \textbf{composed}. 

Now we will define the \textbf{composition} of $\mathcal{C}\to\mathcal{C'}$ and $\mathcal{C'}\to\mathcal{C''}$ as a map of $(n-1)$-cubes $\mathcal{C}\to\mathcal{C}''$, $n\geq 1$. 

We need to define the maps $f_F$ where $F$ is a face of $Cube^n$ such that $\mu(F)$ has a $-$ in its last entry. If $G$ is a face of $Cube^{n-1}$, then we let $G-$ be the face of $Cube^n$ such that $\mu(G-)=(\mu(G),n,-)$.  Let us denote the maps in $\mathcal{C}\to\mathcal{C'}$ and $\mathcal{C'}\to\mathcal{C''}$ by $f_F'$ and $f_F''$. Then, we define for $F=G-$,\begin{align}\label{eqcom} f_{F}:=\sum_{G'>G'' \text{is a bdry of }G} (-1)^{\#(01,\nu(G',G))}f_{G'-}'f_{G''-}''.\end{align}
\begin{claim}
This makes $\mathcal{C}\to\mathcal{C}''$ into an $n$-cube, which is a map of $(n-1)$-cubes.
\end{claim}
\begin{proof}
Let us first give another description of the composition. By taking the $(n-1)$ times iterated cone in the first $(n-1)$-directions, we get two chain maps glued along a chain complex $cone^{n-1}(\mathcal{C})\to cone^{n-1}(\mathcal{C'})\to cone^{n-1}(\mathcal{C''})$. The composition of these maps is also a chain map, hence we obtain a $1$-cube $$cone^{n-1}(\mathcal{C})\to cone^{n-1}(\mathcal{C''}).$$
Now using Lemma \ref{decone} iteratively, we produce from this data a map of $(n-1)$-cubes $\mathcal{C}\to\mathcal{C}''$ with the property that taking the $(n-1)$ times iterated cone in the first $(n-1)$-directions reproduces $cone^{n-1}(\mathcal{C})\to cone^{n-1}(\mathcal{C''})$. Here the only thing we are using is the fact that products of lower triangular block matrices are also of the same form.
Only thing left is to check that this description of composition agrees with the one given by Equation \ref{eqcom}. Up to the signs this is clear. It is easy to see that the exponent of the sign should be $$n(G-)+n(G'-)+n(G''-).$$

For any face $H$ of $Cube^{n-1}$, $$\#(0,H-)+\#(0-,H-)=\#(0-,H) \text{ (mod }2).$$ Noting that $\#(0-,\mu(G))+\#(0-,\mu(G'))+\#(0-,\mu(G''))$ is equal to $\#(01,v(G',G))$ for $G'>G$ a boundary of $G$ finishes the proof.
\end{proof}
%
%
%
%
%

The following lemma follows easily from the proof.

\begin{lemma}\label{c2lcompose}
\begin{itemize}
\item The composition operation is associative. Namely, if we have three cubes that can be glued $$\mathcal{C}\to\mathcal{C'}\to\mathcal{C''}\to \mathcal{C}''',$$ then the composition $\mathcal{C}\to \mathcal{C}'''$ is independent of the order in which we performed the compositions.
\item Composition of two $n$-cubes commutes with the cone operation done in any direction other than the last one.
\end{itemize}\qed
\end{lemma}

A simple case of composition is the following. Let $\mathcal{D}_{i}$, $i=1,2,3$, be three maps of $n$-cubes of the form $\mathcal{C}\to\mathcal{C'}$. Assume that there is a a homotopy from $\mathcal{D}_{1}$ to $\mathcal{D}_{2}$ and also from $\mathcal{D}_{2}$ to $\mathcal{D}_{3}$. Then the $(n+1)$-cubes defining these homotopies give rise to a homotopy from $\mathcal{D}_{1.}$ to $\mathcal{D}_{3}$ using the composition operation. 

The following lemma follows from the definitions. 

\begin{lemma} \label{triangletoslit} Consider an $n$-triangle \begin{align}
\xymatrix{ 
\mathcal{C}\ar[r]^{id}\ar[d]_f\ar[dr]& \mathcal{C}\ar[d]^{g}\\ \mathcal{C}'\ar[r]_{f'} &\mathcal{C}''}
\end{align}
We can define an $n$-slit (keeping the maps represented by the diagonal exactly the same):
 \begin{align}
\xymatrix{ 
\mathcal{C}\ar[r]^{id}\ar[d]_{f'\circ f}\ar[dr]& \mathcal{C}\ar[d]^{g}\\ \mathcal{C}''\ar[r]_{id} &\mathcal{C}'',}
\end{align} which is a homotopy of maps between $g$ and the composition of $f$ and $f'$.
\end{lemma}

\subsubsection{Rays}\label{c2ssrays}
We call an infinite sequence of $n$-cubes $\mathcal{D}_1,\mathcal{D}_2,\ldots$ an \textbf{$n$-ray} if for every $i\geq 1$, $\mathcal{D}_i$ and $\mathcal{D}_{i+1}$ can be glued in the $n$th direction. 

We will generally present an $n$-ray as \begin{align}
\mathcal{C}_1\to \mathcal{C}_2\to \mathcal{C}_3\to\ldots,
\end{align}where $\mathcal{C}_i$ are the $(n-1)$-cubes such that $\mathcal{D}_i$ is the map $\mathcal{C}_i\to \mathcal{C}_{i+1}$, for every $i\geq 1$. We call $\mathcal{C}_i$ the $i$th \textbf{slice}.

Below is a $1$-ray:
\begin{align}
C_1\to C_2\to C_3\to\ldots
\end{align}

And a $2$-ray:
\begin{align}\label{2-ray}
\xymatrix{ 
C_1\ar[r]^{c_1}\ar[d]_{f_1}\ar[dr]_{g_1}& C_2\ar[d]_{f_2}\ar[r]^{c_2}\ar[dr]_{g_2} &C_3\ar[r]\ar[d]^{f_3}\ar[dr]& \ldots\\ D_1\ar[r]_{d_1} &D_2\ar[r]_{d_2}&D_3\ar[r]&\ldots}
\end{align}

We define a \textbf{map} between two $n$-rays $\mathcal{C}_1\to \mathcal{C}_2\to \mathcal{C}_3\to\ldots$ and $\mathcal{C}_1'\to \mathcal{C}_2'\to \mathcal{C}_3'\to\ldots$ to be a collection $\mathcal{E}_i$, $i\geq 1$, of maps of $n$-cubes from $\mathcal{C}_i\to \mathcal{C}_{i+1}$ to $\mathcal{C}_i'\to \mathcal{C}_{i+1}'$ such that the $n$-cube $\mathcal{C}_{i+1}\to \mathcal{C}_{i+1}'$ induced from $\mathcal{E}_i$ and $\mathcal{E}_{i+1}$ agree, i.e. $\mathcal{E}_i$ and $\mathcal{E}_{i+1}$ can be glued in the $n$th direction. 

We warn the reader that a map between two $n$-rays is not an $(n+1)$-ray itself. For example the $2$-ray given in (\ref{2-ray}) satisfies the equations $$f_{i+1}c_i-d_if_i+dg_i+g_id=0.$$ If the same data were to be seen as a map between $1$-rays, then the equations would become $$-f_{i+1}c_i+d_if_i+dg_i+g_id=0.$$

A \textbf{homotopy} between two maps of $n$-rays given by $\mathcal{E}_i$ and $\mathcal{E}'_i$ is a collection of homotopies between $\mathcal{E}_i$ and $\mathcal{E}'_i$ ($(n+2)$-slits), which also can be glued in the $n$th direction. Figure \ref{c2fhomotopy} shows a homotopy between two maps between two $1$-rays. 
\begin{figure}[ht]\label{c2fhomotopy}
\centering
\begin{tikzpicture}[commutative diagrams/every diagram]
  \matrix (m) [matrix of math nodes, row sep=1.5em,
    column sep=1em]{
    & \ldots & & C_i & & C_{i+1} & & \ldots   \\
    \ldots & & C_i  & & C_{i+1}  & & \ldots & \\
    &\ldots & & D_i & & D_{i+1} & & \ldots \\
    \ldots & & D_i  & & D_{i+1} & &\ldots  & \\};
  \path[commutative diagrams/.cd, every arrow, every label]
    (m-1-4) edge (m-1-6) edge node[swap] {id} (m-2-3)
            edge [densely dotted] (m-3-4) 
            edge [densely dotted] (m-4-3) 
            edge [densely dotted] (m-3-6)
            edge [densely dotted] (m-4-5)
    (m-1-6) edge [densely dotted] (m-3-6) edge node[swap] {id}  (m-2-5) 
            edge [densely dotted] (m-4-5) 
    
    (m-2-3) edge [-,line width=6pt,draw=white] (m-2-5)
            edge (m-2-5) edge (m-4-3) edge (m-4-5)
    (m-3-4) edge [densely dotted] (m-3-6)
            edge [densely dotted] node {id} (m-4-3)
            
    (m-4-3) edge (m-4-5)
    (m-3-6) edge node {id}  (m-4-5)
    (m-2-5) edge [-,line width=6pt,draw=white] (m-4-5)
            edge (m-4-5)
    (m-1-2) edge (m-1-4)
    (m-2-1) edge (m-2-3) 
    (m-3-2) edge [densely dotted](m-3-4) 
    (m-4-1) edge (m-4-3) 
    (m-1-6) edge (m-1-8) 
    (m-2-5) edge [-,line width=6pt,draw=white] (m-2-7) edge (m-2-7)
    (m-3-6) edge (m-3-8) 
    (m-4-5) edge (m-4-7);             
\end{tikzpicture}
\caption{A homotopy between two maps between two $1$-rays. If there is no arrow on a face that means that the map on it is zero. Note that here we are seeing $3$-cubes with respect to the ordering where the infinite direction is the first and the direction containing the $id$ maps is the last.}
\end{figure}
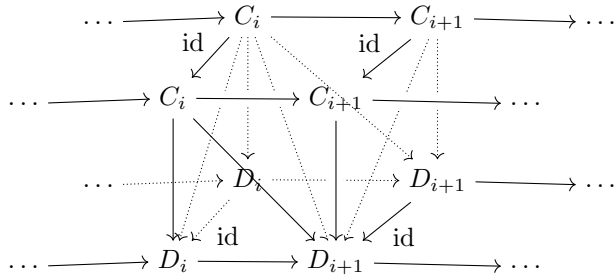

A \textbf{triangle} of maps between $n$-rays is defined in the same way, via $(n+2)$-triangles that can be glued in the $n$th direction.
 
\subsubsection{Cones and telescopes of $n$-rays}\label{c2ssconesand}

First of all, note that we can take the cone in any one of the $n-1$ finite directions of an $n$-ray and obtain an $(n-1)-$ ray. This is an immediate consequence of functoriality of the cones. 

Let $\mathcal{C}=C_1\overset{f_1}{\to} C_2\overset{f_2}{\to} C_3\overset{f_3}{\to}\ldots$ be a $1$-ray. The \textbf{telescope} $tel(\mathcal{C})$ of such a diagram \cite{AS} is defined to be the chain complex with the underlying $R$-module $\bigoplus_{i\in\mathbb{N}}C_i[1]\oplus C_i$ and the differential $\delta$ as depicted below:
\begin{align}
\xymatrix{
C_1\ar@{>}@(ul,ur)^{d }  &C_2\ar@{>}@(ul,ur)^{d} &C_3\ar@{>}@(ul,ur)^{d}\\
C_1[1]\ar@{>}@(dl,dr)_{-d} \ar[u]^{\text{id}}\ar[ur]^{f_1} &C_2[1]\ar@{>}@(dl,dr)_{-d} \ar[u]^{\text{id}}\ar[ur]^{f_2}&C_3[1]\ar@{>}@(dl,dr)_{-d} \ar[u]^{\text{id}}_{\ldots}}
\end{align}

More precisely, if $x\in C_i$, then $\delta x:=dx$, where $d$ represents the differential in $C_i$, and if $x\in C_i[1]$, then $\delta x:=\tilde{x}-dx+f_i(x)$, where again $d$ is the differential of $C_i$, $\tilde{x}$ is the copy of $x$ in $C_i$, and $f_i(x)$ is considered as an element of $C_{i+1}$, for every $i\geq 1$. 

We note the abuse of notation we did in the previous definition (and will continue to do so as we believe it does not cause confusion): the element $x$ and expressions on the right hand sides in these definitions are considered as elements of $\bigoplus_{i\in\mathbb{N}}C_i[1]\oplus C_i$ via the natural inclusions of summands. 

More generally, we define the \textbf{telescope} $tel(\mathcal{C})$ of an $n$-ray $\mathcal{C}$ as an $(n-1)$-cube in the following way. Let $\mathcal{C}$ be the $n$-ray $\mathcal{C}_1\to \mathcal{C}_2\to \mathcal{C}_3\to\ldots$.

Consider the cones $cone^n(\mathcal{D}_i)$ in the last direction of the $n$-cubes $\mathcal{D}_i:=\mathcal{C}_{i}\to\mathcal{C}_{i+1}$, for $i\geq 0$, where $\mathcal{C}_{0}$ is defined to have zero modules at all vertices. Denote the maps in $cone^n(\mathcal{D}_i)$ by $f^i_F$ for $F$ a face of $Cube^{n-1}$. 

Define the $R$-modules at any vertex $w$ of $Cube^{n-1}$ of $tel(\mathcal{C})$ as the direct sum $$\bigoplus_{i\in\mathbb{N}}(cone^n(\mathcal{D}_i))^w=\bigoplus_{i\in\mathbb{N}}\mathcal{C}^w_i[1]\oplus \mathcal{C}^w_i.$$

We then define the maps $\delta_F$ of $tel(\mathcal{C})$, for every face $F$ of $Cube^{n-1}$ as follows.  Let $w$ be the initial vertex of $F$. For $x\in  \mathcal{C}^w_i$, we simply define $\delta_Fx= f^i_Fx$. If $x\in  \mathcal{C}^w_i[1]$, and $F$ is zero dimensional, we define $\delta_Fx= f^i_Fx+\tilde{x}$, where $\tilde{x}$ is the copy of $x$ in $\mathcal{C}^w_i$. If $x\in  \mathcal{C}^w_i[1]$, and $F$ is positive dimensional, we define $\delta_Fx= f^i_Fx$.

Informally, the maps in the $(n-1)$-cube structure are depicted in: 
\begin{align}
\xymatrix{
\mathcal{C}_0   &\mathcal{C}_1&\mathcal{C}_2\\
\pm\mathcal{C}_0[1] \ar[u]^{\text{id}}\ar[ur] &\pm\mathcal{C}_1[1] \ar[u]^{\text{id}}\ar[ur]&\pm\mathcal{C}_2[1] \ar[u]^{\text{id}}_{\ldots}}
\end{align}
Note that the $\mathcal{C}_i$'s (and the shifted copies) have internal structure that makes them an $(n-1)$-cube that is taken into account here, and the $\pm$ in front means that some those maps are negated. The sign changes in the internal structure of $\pm\mathcal{C}_i[1]$ and the structure maps corresponding to diagonal and vertical arrows come from the cones of $\mathcal{D}_i=\mathcal{C}_{i}\to\mathcal{C}_{i+1}$ and $\mathcal{C}_{i}\xrightarrow{id}\mathcal{C}_i$ in the $n$th direction. The statement that $tel(\mathcal{C})$ is an $(n-1)$-cube boils down to functoriality of cones.

\begin{lemma}\label{c2ltelescope}
\begin{itemize}
\item We get a canonical $1$-ray from any $n$-ray by an $(n-1)$ times iterated cone. This commutes with the telescope operation.
\item Telescopes are functorial in the sense that (1) a map of $n$-rays canonically gives a map of the telescopes (which are $(n-1)$-cubes), (2) a homotopy between two maps gives a homotopy, (3) a triangle of maps gives a triangle of maps.
\end{itemize}\qed
\end{lemma}

\subsection{$1$-rays and quasi-isomorphisms}\label{c2srays}

The arguments in this section are slight modifications of \cite{AS}. 

\begin{lemma}\label{c2lcolimit}
Let $\mathcal{C}=C_1\to C_2\to C_3\to\ldots$ be a $1$-ray. 
Then, there is a canonical quasi-isomorphism \begin{align}tel(\mathcal{C})\to\lim_{\rightarrow}(C_1\to C_2\to\ldots)
\end{align}
\end{lemma}
\begin{proof}
Define $F^n(tel(\mathcal{C}))$ to be $(\bigoplus_{i\in [1,n-1]}C_i[1]\oplus C_i)\oplus C_n$. Notice that $tel(\mathcal{C})$ is the usual direct limit of $F^n(tel(\mathcal{C}))$. Moreover, there are canonical quasi-isomorphisms $F^n(\mathcal{C})\to C_n$ induced by the given maps $C_i\to C_n$, $i\in [1,n]$ and the zero maps $C_i[1]\to C_n$, $i\in [1,n-1]$, which makes the diagrams 
\begin{align}\label{c2estrictification}
\xymatrix{
F^n(tel(\mathcal{C})) \ar[d]\ar[r]  &F^{n+1}(tel(\mathcal{C}))\ar[d] \\ C_n \ar[r] &C_{n+1}},
\end{align}
commutative. The induced map $tel(\mathcal{C})\to\lim_{\rightarrow}C_i$ is also a quasi-isomorphism, since direct limits commute with homology. 
\end{proof}

In particular, $\lim_{\rightarrow}H(C_i)$ is canonically isomorphic to $H(tel(\mathcal{C}))$. This isomorphism is also functorial in the following way.

\begin{lemma}\label{c2lhomologymaps}
Let $\mathcal{C}=C_1\to C_2\to C_3\to\ldots$ and $\mathcal{D}=D_1\to D_2\to D_3\to\ldots$ be $1$-rays, and let $f:\mathcal{C}\to\mathcal{D}$ be a map of $1$-rays. Note that this induces a strictly commutative diagram  \begin{align}
\xymatrix{ 
H(C_1)\ar[r]\ar[d]& H(C_2)\ar[d]\ar[r] &H(C_3)\ar[r]\ar[d]& \ldots\\ H(D_1)\ar[r] &H(D_2)\ar[r]&H(D_3)\ar[r]&\ldots}, 
\end{align} and consequently a map $$\lim_{\rightarrow}H(C_i)\to \lim_{\rightarrow}H(D_i).$$ On the other hand $f$ also induces the map $$tel(\mathcal{C})\to tel(\mathcal{D}).$$ The diagram \begin{align}
\xymatrix{
H(tel(\mathcal{C})) \ar[d]\ar[r]  &H(tel(\mathcal{D}))\ar[d] \\ \lim_{\rightarrow}H(C_i) \ar[r] &\lim_{\rightarrow}H(D_i)}
\end{align}commutes.
\end{lemma}
\begin{proof}
The important point is that $f$ induces a strictly commutative diagram
 \begin{align}
\xymatrix{ 
F^1(\mathcal{C})\ar[r]\ar[d]& F^2(\mathcal{C})\ar[d]\ar[r] &F^3(\mathcal{C})\ar[r]\ar[d]& \ldots\\ F^1(\mathcal{D})\ar[r] &F^2(\mathcal{D})\ar[r]&F^3(\mathcal{D})\ar[r]&\ldots}, 
\end{align}
and $tel(\mathcal{C})\to tel(\mathcal{D})$ is simply the maps between the usual direct limits.

Also note that the diagrams
\begin{align}
\xymatrix{
H(F^n(tel(\mathcal{C}))) \ar[d]\ar[r]  &H(F^{n}(tel(\mathcal{D})))\ar[d] \\ H(C_n) \ar[r] &H(D_{n})})
\end{align} commute for every $n$, and that we have the commutative diagrams from Equation \ref{c2estrictification}.

We then obtain the desired statement from the morphism version of the fact that homotopy commutes with direct limits.
\end{proof}

Let $\mathcal{C}=C_1\to C_2\to C_3\to\ldots$ be a $1$-ray, and $i(0)<i(1)<i(2)<\ldots $ be an infinite strictly monotone sequence of positive integers. Note that by composing maps we get a unique map $C_n\to C_m$ for all $m\geq n$. Then we canonically obtain a $1$-ray $\mathcal{C}^i=C_{i(1)}\to C_{i(2)}\to\ldots$. Let us call this a \textbf{subray}. Let us call the canonical map of $1$-rays $\mathcal{C}\to\mathcal{C}^i$ a compression map:
 
 \begin{align}
\xymatrix{ 
C_1\ar[r]\ar[d]& C_2\ar[d]\ar[r] &C_3\ar[r]\ar[d]& \ldots\\ C_{i(1)}\ar[r] &C_{i(2)}\ar[r]&C_{i(3)}\ar[r]&\ldots}
\end{align}

\begin{lemma}\label{c2lcompression}
The compression map induces a quasi-isomorphism: $tel(\mathcal{C})\to tel(\mathcal{C}^i)$.
\end{lemma}

\begin{proof}
We have the following commutative diagram:
\begin{align}
\xymatrix{
tel(\mathcal{C}) \ar[d]\ar[r]  &tel(\mathcal{C}^i)\ar[d] \\ \lim_{\rightarrow}(C_1\to C_2\to\ldots) \ar[r] &\lim_{\rightarrow}(C_{i(1)}\to C_{i(2)}\to\ldots)}
\end{align}Note that the bottom horizontal map is a chain map which is an isomorphism of the underlying modules. This finishes the proof.
\end{proof}

\begin{definition}\label{c2dweak}
Let $\mathcal{C}=\mathcal{C}_1\to \mathcal{C}_2\to \mathcal{C}_3\to\ldots$ be an $n$-ray, and $i(0)<i(1)<i(2)<\ldots $ be any infinite strictly monotone sequence of positive integers. A ray $\mathcal{C}_{i(1)}\to \mathcal{C}_{i(2)}\to\ldots$ is called a \textbf{weak subray} if for every $k\geq 1$, $\mathcal{C}_{i(k)}\to \mathcal{C}_{i(k+1)}$ as a map of $(n-1)$-cubes is homotopic to the composition of $\mathcal{C}_{i(k)}\to \mathcal{C}_{i(k)+1}\to\ldots\to\mathcal{C}_{i(k+1)}$.

Similarly, a map of $n$-rays from $\mathcal{C}$ to a weak subray $\mathcal{C}_{i(1)}\to \mathcal{C}_{i(2)}\to\ldots$ is called a \textbf{weak compression} if for every $k\geq 1$, the map $\mathcal{C}_{k}\to \mathcal{C}_{i(k)}$ is homotopic to the composition $\mathcal{C}_{k}\to \mathcal{C}_{k+1}\to\ldots\to\mathcal{C}_{i(k)}$.
\end{definition}

\begin{proposition}\label{c2pweakcompression}
Let $\mathcal{C}=\mathcal{C}_1\to \mathcal{C}_2\to \mathcal{C}_3\to\ldots$ be an $n$-ray and \begin{align}
\xymatrix{ 
\mathcal{C}_1\ar[r]\ar[d]& \mathcal{C}_2\ar[d]\ar[r] &\mathcal{C}_3\ar[r]\ar[d]& \ldots \\ \mathcal{C}_{i(1)}\ar[r] &\mathcal{C}_{i(2)}\ar[r]&\mathcal{C}_{i(3)}\ar[r]&\ldots}
\end{align} be a weak compression. Then applying $tel\circ cone^{n-1}$ to this map of $n$-rays, we obtain a quasi-isomorphism.
\end{proposition}

\begin{proof}
Note that since composition commutes with cones as in Lemma \ref{c2lcompose}, we have the following weak compression of $1$-rays:
\begin{align}
\xymatrix{ 
cone^{n-1}(\mathcal{C}_1)\ar[r]\ar[d]& cone^{n-1}(\mathcal{C}_2)\ar[d]\ar[r] &cone^{n-1}(\mathcal{C}_3)\ar[r]\ar[d]& \ldots \\ cone^{n-1}(\mathcal{C}_{i(1)})\ar[r] &cone^{n-1}(\mathcal{C}_{i(2)})\ar[r]&cone^{n-1}(\mathcal{C}_{i(3)})\ar[r]&\ldots}
\end{align} Therefore, it suffices to prove the statement for $1$-rays. This follows from Lemmas \ref{c2lhomologymaps} and \ref{c2lcompression}.
\end{proof}

\subsection{Completion of modules and chain complexes over the Novikov ring}\label{c2scompletion}

Let us start by writing down our conventions for the Novikov field:
\begin{align}
\Lambda=\{\sum_{i\in\mathbb{N}} a_i T^{\alpha_i}\mid a_i\in\mathbb{Q}, \alpha_i\in\mathbb{R} \text{, and for any }R\in\mathbb{R},\\ \text{there are only finitely many }a_i\neq 0\text{ with }\alpha_i<R\}
\end{align}

There is a valuation map $val:\Lambda\to\mathbb{R}\cup\{+\infty\}$ given by $val(\sum_{i\in\mathbb{N}} a_i T^{\alpha_i})=min_i(\alpha_i\mid a_i\neq 0)$ for non zero elements, and $val(0)=+\infty$. We define $\Lambda_{\geq r}:=val^{-1}([r,\infty])$ and $\Lambda_{> r}:=val^{-1}((r,\infty])$. $\Lambda_{\geq 0}$ is called the Novikov ring. The valuation we described makes $\Lambda_{\geq 0}$ a complete valuation ring with real numbers as the value group (see Section 2.1 of \cite{Bosch_2014}). 

\begin{lemma}\label{c2lmodule}Let $A$ be a $\Lambda_{\geq 0}$-module. Then, $A$ is flat if and only if it is torsion free. 
\end{lemma}
\begin{proof}
This is true for any valuation ring (see Lemma 3.2 of \cite{Datta} for a proof). 
\end{proof}

\begin{corollary}\label{c2cacyclic}
Let $C$ be an acyclic chain complex over $\Lambda_{\geq 0}$ with a torsion free underlying module. Then, $C\otimes_{\Lambda_{\geq 0}}\Lambda_{\geq 0}/\Lambda_{\geq r}$, $r>0$, and $C\otimes_{\Lambda_{\geq 0}}\Lambda_{\geq 0}/\Lambda_{> r}$, for $r\geq 0$, are also acyclic.
\end{corollary}
\begin{proof}We have the short exact sequence:
\begin{align}
\xymatrix{
  0 \ar[r] & \Lambda_{\geq r} \ar[r] & \Lambda_{\geq 0} \ar[r] & \Lambda_{\geq 0}/\Lambda_{\geq r}\ar[r] & 0. 
}
\end{align}
We now tensor this equation with $C$, and consider the long exact sequence of the resulting short exact sequence of chain complexes (using that $C$ is flat). The desired result follows since $\Lambda_{\geq r}$ is a flat module, and consequently, $C\otimes_{\Lambda_{\geq 0}}\Lambda_{\geq r}$ is also acyclic. 

The same argument works if $\Lambda_{\geq r}$ is replaced with $\Lambda_{> r}$ as well.
\end{proof}

Completion is a functor $Mod(\Lambda_{\geq 0})\to Mod(\Lambda_{\geq 0})$ defined by \begin{align}
A\mapsto \widehat{A}:\lim_{\xleftarrow[r\geq 0]{}}A\otimes_{\Lambda_{\geq 0}}\Lambda_{\geq 0}/\Lambda_{\geq r},
\end{align} and by functoriality of inverse limits on the morphisms. There is a natural map of modules $A\to \widehat{A}$. $A$ is called \textbf{complete} is this map is an isomorphism.

One can construct the completion in the following way. Let us say that a sequence $(a_1,a_2,\ldots)$ of elements of $A$

\begin{itemize}
\item is a \textbf{Cauchy sequence}, if for every $r\geq 0$ there exists a positive integer $N$ such that for every $n,n'>N$, $a_n-a_{n'}\in T^rA$,
\item \textbf{converges} to $a\in A$, if for every $r\geq 0$ there exists a positive integer $N$ such that for every $n>N$, $a-a_n\in T^rA$.
\end{itemize}

Then, we have that $\widehat{A}$ is isomorphic to all Cauchy sequences in $M$ (with its natural $\Lambda_{\geq 0}$-module structure) modulo the ones that converge to $0$. Note that the completeness of $A$ is equivalent to all Cauchy sequences in $A$ to be convergent.

In case $A$ is free, this description becomes simpler. Choose a basis $\{v_i\}$, $i\in \mathcal{I}$. Then, $\widehat{A}$ is isomorphic to\begin{align}\{\sum_{i\in\mathcal{I}}\beta_iv_i\mid \beta_i\in\Lambda_{\geq 0},\text{ and for every }R\geq 0,\text{ there is only}\\ \text{ finitely many }i\in\mathcal{I}\text{ s.t. }val(\beta_i)< R\}.\end{align} The following lemma is immediate from this description.

\begin{lemma}\label{c2lfreecompletion}
Let $A$ be a free $\Lambda_{\geq 0}$-module. Then \begin{itemize}
\item $\widehat{A}$ is torsion free.
\item The map $A\otimes_{\Lambda_{\geq 0}}\Lambda_{\geq 0}/\Lambda_{\geq r}\to \widehat{A}\otimes_{\Lambda_{\geq 0}}\Lambda_{\geq 0}/\Lambda_{\geq r}$ is an isomorphism for all $r\geq 0$.
\end{itemize}
\end{lemma}

The completion functor automatically extends to a functor $Ch(\Lambda_{\geq 0})\to Ch(\Lambda_{\geq 0})$. Namely, if $(C,d)$ is a chain complex over $\Lambda_{\geq 0}$, then the completion $(\widehat{C},\widehat{d})$ is obtained by applying the completion functor to the underlying module, and also to the map $d:C\to C$. This is clearly a chain complex, and as usual, we mostly omit the differential from the notation.

\begin{remark}
In the $\mathbb{Z}$-graded context, there are two genuinely different ways of completing a graded module, and consequently, a chain complex: degree-wise completion and completing the underlying module forgetting the grading. For our $\mathbb{Z}/2$-graded chain complexes there is no difference between these two options as completion commutes with finite direct sums. 

Let us denote the periodical extension of a $\mathbb{Z}/2$-graded chain complex $C$ to a $\mathbb{Z}$-graded $2$-periodic chain complex by $C^{per}$ (see also the second paragraph of the introduction in the beginning of this section). Note that $(\widehat{C})^{per}$ is the degree-wise completion of $C^{per}$. Hence, what we are dealing with is a special case of grade-wise completions of $\mathbb{Z}$-graded complexes.

There is nothing special about this $2$-periodicity for any of our results, which all easily extend to $\mathbb{Z}$-graded chain complexes provided that we use degree-wise completion. \qed
\end{remark}

\begin{lemma}\label{c2lcomplete}
Let $C$ be a chain complex over $\Lambda_{\geq 0}$, and $r>0$. If the underlying module of $C$ is torsion-free and complete, then $C\otimes_{\Lambda_{\geq 0}}\Lambda_{\geq 0}/\Lambda_{\geq r}$ is acyclic only if $C$ is acyclic.
\end{lemma}

\begin{proof}
Let $\alpha\in C$, and $d\alpha=0$. We need to show that $\alpha$ is exact. Our assumption implies that there exists $a,b\in C$ such that $\alpha=db+T^ra$. 

We have that $d(T^ra)=T^rda=0$, which implies that $da=0$ by torsion-freeness. Now we repeat the previous step for $a$, and keep going. Because of our completeness assumption this defines a primitive of $\alpha$.
\end{proof}

\begin{corollary}\label{c2ccomplete}
\begin{enumerate}
\item Assume that $C$ is finitely generated free as a module, then if $C\otimes_{\Lambda_{\geq 0}}\Lambda_{\geq 0}/\Lambda_{> 0}$ is acyclic then so is $C$.
\item Assume that $C$ is free as a module, then $C$ acyclic implies $\widehat{C}$ acyclic.
\item Let $f:C\to C'$ be a chain map. Assume that the underlying modules of $C$ and $C'$ are free. Then $\hat{f}:\widehat{C}\to\widehat{C'}$ is  a quasi-isomorphism if $f$ is one.
\end{enumerate}
\end{corollary}
\begin{proof}
For (1), choose a basis for $C$ and write $d$ as a matrix. There exists a smallest positive number $r$ such that $T^r$ has a non-zero coefficient in a matrix entry. Then our assumption actually implies that $C\otimes_{\Lambda_{\geq 0}}\Lambda_{\geq 0}/\Lambda_{\geq r}$ is acyclic. We now apply Lemma \ref{c2lcomplete}.

For (2), we combine Corollary \ref{c2cacyclic}, Lemma \ref{c2lfreecompletion} and Lemma \ref{c2lcomplete} (noting that the completion of a module is complete). Finally, for (3), we use the fact a chain map is a quasi-isomorphism if its cone is acyclic.
\end{proof}

Even though taking homotopy colimits is better suited for general constructions, sometimes usual direct limits are better for computations. To this end we show that Lemma \ref{c2lcolimit} still holds after completions. A similar argument seems to have appeared in the Lemma 3.1 of \cite{Lunts}.

\begin{lemma}\label{c2lcomlim}
Let $\mathcal{C}=C_0\to C_1\to C_2\to\ldots$ be a $1$-ray. There is a canonical quasi-isomorphism \begin{align}\widehat{tel(\mathcal{C})}\to\widehat{\lim_{\rightarrow}(\mathcal{C})}.\end{align}
\end{lemma}

\begin{proof}
We have canonical quasi-isomorphisms \begin{align}f_r:tel(\mathcal{C})\otimes_{\Lambda_{\geq 0}} \Lambda_{\geq 0}/\Lambda_{\geq r}\to\lim_{\rightarrow}(\mathcal{C})\otimes_{\Lambda_{\geq 0}} \Lambda_{\geq 0}/\Lambda_{\geq r},\end{align}
that are compatible with each other, using Lemma \ref{c2lcolimit} and that tensor product commutes with telescopes and direct limit. We claim that the inverse limit over $r$ of these maps give the desired map.

We show that the inverse limit of $cone(f_r)$ is acyclic, which is clearly enough. Note  that the maps in this inverse system are all surjective. Therefore we have a Milnor short exact sequence (see Theorem 3.5.8 in \cite{W}), and the fact that $cone(f_r)$'s are acyclic implies the desired acyclicity.
\end{proof}

\subsection{Acyclic cubes and an exact sequence}\label{c2sacyclic}

Starting from an $n$-ray we can obtain a $(n-1)$-cube by applying telescope. We can then apply completion functor to the result. Hence, we obtain an assignment $\widehat{tel}: (n-rays)\to ((n-1)-cubes)$. This trivially extends to morphisms, and respects homotopies. It is functorial in the sense that it also preserves triangles. We can also apply the maximally iterated cone functor to obtain a chain complex. In fact we could have applied it before the other two operations and the result would not change: $cone^{n-1}\circ \widehat{tel}=\widehat{cone^{n-1}\circ tel}=\widehat{tel\circ cone^{n-1}}$. Note that completion is always applied after telescope.

Let us call an $n$-cube \textbf{acyclic} if its maximally iterated cone is an acyclic chain complex. Note that by Corollary \ref{c2ccomplete} Part (1), if the modules in this cube are finitely generated free, then this acyclicity is equivalent to acyclicity after tensoring with the residue field of $\Lambda_{\geq 0}$, i.e. $\Lambda_{\geq 0}/\Lambda_{> 0}$.

\begin{remark}\label{remarkT} Note that the residue field is naturally identified with $\mathbb{Q}$. Using this identification, if $C$ is a chain complex over $\Lambda_{\geq 0}$, then $C\otimes \Lambda_{\geq 0}/\Lambda_{> 0}$ is a chain complex over $\mathbb{Q}$. If the underlying module of $C$ is the free module $ (\Lambda_{\geq 0})^n$, then the differential can be represented by a matrix $M$ with entries in $ \Lambda_{\geq 0}$. Then, a concrete description of the $\mathbb{Q}$-chain complex $C\otimes \Lambda_{\geq 0}/\Lambda_{> 0}$ is the vector space $\mathbb{Q}^n$ with differential given by the matrix with rational number entries obtained by setting $T=0$ in the entries of the matrix $M$. \end{remark}

\begin{lemma}\label{c2lacycliccube}
Let $\mathcal{C}$ be a $n$-ray where the underlying modules are free. Assume that all the slices are acyclic $(n-1)$-cubes, then $tel(\mathcal{C})$ is acyclic, and hence $\widehat{tel}(\mathcal{C})$ is also acyclic.
\end{lemma}
\begin{proof}
The first follows because the maximally iterated cone commutes with the telescope functor, and Lemma \ref{c2lcolimit}. The second part is Lemma \ref{c2ccomplete} Part (2).
\end{proof}

\begin{lemma}\label{c2lmv}
An acyclic $2$-cube \begin{align}
\xymatrix{ 
C_{00}\ar[r]\ar[d]\ar[dr]& C_{10}\ar[d]\\ C_{01}\ar[r] &C_{11}}
\end{align} gives rise to an exact sequence, 
\begin{align}
\xymatrix{
H(C_{00})\ar[r]&H(C_{10})\oplus H(C_{01})\ar[dl]\\H(C_{11})\ar[u]^{[1]}}.
\end{align}
where the degree preserving arrows are induced from the ones in the $2$-cube.
\end{lemma}

\begin{proof}
The acyclicity implies that $C_{00}\to cone(C_{10}\oplus C_{01}\to C_{11})$ is a quasi-isomorphism. Then the long exact sequence of homology associated to the cone finishes the proof.
\end{proof}

\section{Definition and Basic properties}\label{c3}

In this section, we assume familiarity with Hamiltonian Floer theory in the form described in Pardon \cite{P}, Section 10. We also freely use notations and results of the previous section. 

\subsection{Conventions}\label{c3sconventions}

In this short subsection, we put together our conventions in setting up Hamiltonian Floer theory. 
\begin{enumerate}
\item $\omega(X_H,\cdot )=dH$.

\item $g:=\omega(\cdot,J\cdot )$ is a Riemannian metric, and hence $JX_H=\text{grad}_g H$.

\item Floer equation: $J\frac{\partial u}{\partial s}=\frac{\partial u}{\partial t}-X_H$.

\item Topological energy of (arbitrary) $u:S^1\times \mathbb{R}\to M$ for a given Hamiltonian $S^1\times \mathbb{R}\times M\to\mathbb{R}$: \begin{align}
\int u^*\omega+\int \partial_s&(H(s,t,u(s,t))dsdt\\=&\int\omega(\frac{\partial u}{\partial s},\frac{\partial u}{\partial t})dsdt+\int[(\partial_sH_s)+d_{u(s,t)}H_{s,t}(\frac{\partial u}{\partial s})]dsdt\\=&\int\omega(\frac{\partial u}{\partial s},\frac{\partial u}{\partial t}-X_H)dsdt+\int(\partial_sH_s)dsdt.
\end{align}
\item Homomorphisms defined by moduli problems always send the generator of the Floer complex at the negative punctures to the one at the positive puncture.
\item We consider all orbits, not just contractible ones.
\item We always work over $\Lambda_{\geq 0}$. The generators have no action but the solutions of Floer equations are weighted by their topological energy.
\item Non-degenerate $1$-periodic orbits of a $1$-periodic Hamiltonian vector field can be assigned a sign, which is the index of the fixed point of the time-$1$ map (i.e. the Lefschetz index). This sign defines the $\mathbb{Z}/2$-grading of our chain complexes.
\end{enumerate}

\begin{remark}
As already hinted at, we will be using virtual techniques for Hamiltonian Floer theory. Our heavy use of families of Floer equations (which are parametrized by arbitrarily large dimensional spaces) makes the purely genericity based approaches to deal with negative Chern number spheres insufficient. We could get around using virtual techniques when $c_1(M,\omega)|_{\pi_2(M)}$ and $[\omega]|_{\pi_2(M)}$ are non-negative multiples of each other, in other words, in the so called aspherical, monotone and Calabi-Yau cases. All of these cases are covered in the first three lectures of \cite{Sa} (and references cited therein) at a level that is sufficient for our purposes except that in the Calabi-Yau case, we would also need the results of McLean, Appendix B \cite{Mc} This is because in the standard reference \cite{HS} genericity of Hamiltonian data was used to achieve certain transversality results. We have a more limited flexibility in that choice for different reasons (but especially in Section \ref{c5}) and hence we would need to achieve transversality using the almost complex structure, which is done in the aforementioned reference.
\end{remark}

\begin{remark}\label{c3rgradingconventions}
Assume that $c_1(M,\omega)=0$, and we fix a homotopy class of a smooth trivialization of the canonical bundle (defined via some compatible almost complex structure). Then, all of our chain complexes can be equipped with a $\mathbb{Z}$-grading. All the statements that we prove can be extended to take into account this grading with no extra work.
\end{remark}

\begin{remark}
In \cite{P}, Pardon restricts himself to contractible orbits, but this is not a necessary restriction, and everything there generalizes to our setting with no extra effort. This includes the results regarding orientations of moduli spaces. The main reason for this is that what forms the backbone of the results in \cite{P} are properties of the orientation gluing operation that are proved in Floer-Hofer's coherent orientations paper \cite{FHcoh}. Floer-Hofer paper does not have any restrictions about the topology of the orbits in question. We note that a potentially different framework (not used in \cite{P}), where the orientation lines are constructed intrinsically using determinant lines of Floer operators on caps (as in \cite{AA}), at first glance seems use the contractibility of orbits. Yet, even this framework does not really need that assumption as explained in \cite{GPS}, page 77.
\end{remark}

\subsection{Hamiltonian Floer theory}\label{c3shamiltonian}

Let $M$ be a closed symplectic manifold. Take a one periodic time-dependent Hamiltonian $H:M\times S^1\to \mathbb{R}$ with non-degenerate one-periodic orbits, the set of which is denoted by $\mathcal{P}(H)$. Then, for every compatible almost complex structure $J$,  there exists choices of extra Pardon data $P$ (as in Definition 7.5.3 in \cite{P}), and coherent orientations (as in Appendix C of \cite{P}) so we can define a chain complex over $\Lambda_{\geq 0}$ as follows:

\begin{itemize}
\item As a $\mathbb{Z}_2$-graded module: \begin{align}CF(H,J,P)=\bigoplus_{\gamma\in\mathcal{P}(H)}\Lambda_{\geq 0}\cdot\gamma,\end{align}i.e. $CF(H,J,P)$ is freely generated over $\Lambda_{\geq 0}$ by the elements of $\mathcal{P}(H)$. 
\item We define the differential by the formula:\begin{align}
d\gamma=\sum_{\gamma', A\in\pi_2(\gamma,\gamma')}\#^{vir}\mathcal{M}(\gamma,\gamma',A,H,J,P) T^{\omega(A)+\int_{S^1}\gamma'^*Hdt-\int_{S^1}\gamma^*Hdt}\gamma',
\end{align}and extend it $\Lambda_{\geq 0}$-linearly. Here $\pi_2(\gamma,\gamma')$ denotes the set of homotopy classes of continuous maps $S^1\times I\to M$ such that $S^1\times\{0\}\to M$ and $S^1\times\{1\}\to M$ are the defining parametrizations of $\gamma$ and $\gamma'$, respectively. 

$\#^{vir}\mathcal{M}(\gamma,\gamma',A,H,J,P)\in\mathbb{Q}$ are virtual numbers defined as in Pardon. These are virtual counts of finite energy genus 0 nodal curves with one negative and one positive puncture, where both punctures are at the same component, mapping into $M$. The component with punctures is a cylinder and the restriction of the map to it $u:\mathbb{R}\times S^1\to M$ satisfies the Floer equation:\begin{align}\label{jhol}
J\frac{\partial u}{\partial s}=\frac{\partial u}{\partial t}-X_H,
\end{align}with the asymptotic conditions\begin{align}
u(t,s)\to
\begin{cases}
\gamma(t),\ s\to -\infty\\
\gamma'(t),\ s\to\infty.
\end{cases}
\end{align} 
The other components of the curve are $J$-holomorphic spheres. Using asymptotic convergence and resolving the nodes, we canonically obtain a class in $\pi_2(\gamma,\gamma')$. We require this class to be equal to $A$. 

Moreover, integrating $\omega$ over all the components of this nodal curve, and summing them, we obtain a number that only depends on the cohomology class of $\omega$ and $A$: $\omega(A)$. This number can be defined as follows. $\omega$ defines a class in $H^2(M, im(\gamma)\cup im(\gamma'))$, as $im(\gamma)\cup im(\gamma')$ is a one dimensional manifold. On the other hand $A$ gives us a class in $H_2(M, im(\gamma)\cup im(\gamma'))$ using the orientation on $S^1\times I$ induced by the complex structure that we used to write down the Floer equation. The canonical pairing of these two classes is $\omega(A)$.

Hence, the exponent of $T$ in the formula, $\omega(A)+\int_{S^1}\gamma'^*Hdt-\int_{S^1}\gamma^*Hdt$, is the topological energy (as in the item (4) of Section \ref{c3sconventions}) of $u$ plus the integral of $\omega$ along the sphere components. It follows from the well-known computation presented in that same item (4) that each of these terms, and hence $\omega(A)+\int_{S^1}\gamma'^*Hdt-\int_{S^1}\gamma^*Hdt$, is always non-negative whenever $\#^{vir}\mathcal{M}(\gamma,\gamma',A,H,J,P)\neq 0$.

For a more careful description of the moduli spaces involved see Definition 10.2.2 for $n=0$ in \cite{P}.

This makes $d$ a degree one $\Lambda_{\geq 0}$-module map that squares to zero.
\end{itemize} 

Continuing to follow Pardon, we outline what Hamiltonian Floer theory gives us for higher dimensional families of Hamiltonians. It will be more convenient to use cubes, so we give the theory in that framework, instead of the simplices as Pardon does.

Let $Cube^n=[0,1]^n\subset \mathbb{R}^n$. Let us consider the Morse function \begin{align}\label{c3emorse}f(x_1,\ldots ,x_n)=\sum_{i=1}^n\rho{( x_i)},\end{align}where $\rho:[0,1]\to [-1,1]$ is a smooth function satisfying: \begin{itemize}
\item $\rho(x)=1-x^2$ near $x=0$
\item $\rho(x)=-1+x^2$ near $x=1$
\item $\rho$ is strictly decreasing 
\end{itemize}

 Critical points of $f$ are precisely the vertices of the cube, and its gradient vector field (with respect to the flat metric) is tangent to all the strata of the cube.
 

By an \textbf{$n$-cube of Hamiltonians}, we mean a smooth\footnote{meaning $Cube^n\times M\times S^1\to \mathbb{R}$ is smooth} map $H:Cube^n\to C^{\infty}(M\times S^1,\mathbb{R})$, which is constant on an open neighborhood of each of the vertices, and also so that the Hamiltonians at the vertices are non-degenerate. 

\begin{remark}\label{c3rsmooth}
In Section \ref{c3sscontractibility}, we will relax the smoothness requirement of the map $H:Cube^n\to C^{\infty}(M\times S^1,\mathbb{R})$ a little bit. This is related to the fact that what is relevant is not the smoothness of $H$ but the smoothness of the family of the continuation map equations that it defines.
\end{remark}

We also choose a $Cube^n$-family of compatible almost complex structures $J$, which are also constant near the vertices, Pardon data $P$, and coherent orientations. Now, for each face $F$ of the cube we can consider virtual counts $\#^{vir}\mathcal{M}(\gamma,\gamma',A,H,J,P,F)$ of Floer trajectories for any $\gamma$ and $\gamma'$, which are periodic orbits of the Hamiltonians at the initial and terminal (resp.) vertices of $F$, intuitively counting rigid, possibly broken and bubbled, solutions of Equation \ref{jhol} with $(s,t)$-dependent $H$ and $J$ prescribed by the gradient flow lines of $f$ (see Figure \ref{c3fsalamon} for a picture, and Definition 10.2.2 in \cite{P} for a precise definition). We will again weight these counts by powers of the Novikov parameter $T$ using their topological energy. Note that the exponents here (meaning for $n\geq 1$) do not have to be non-negative, unless a monotonicity condition is assumed on the $n$-cube of Hamiltonians, see Definition \ref{c3dmonotone}.

\begin{figure}
\fbox{\begin{minipage}{0.5\textwidth}
\includegraphics[scale=0.3]{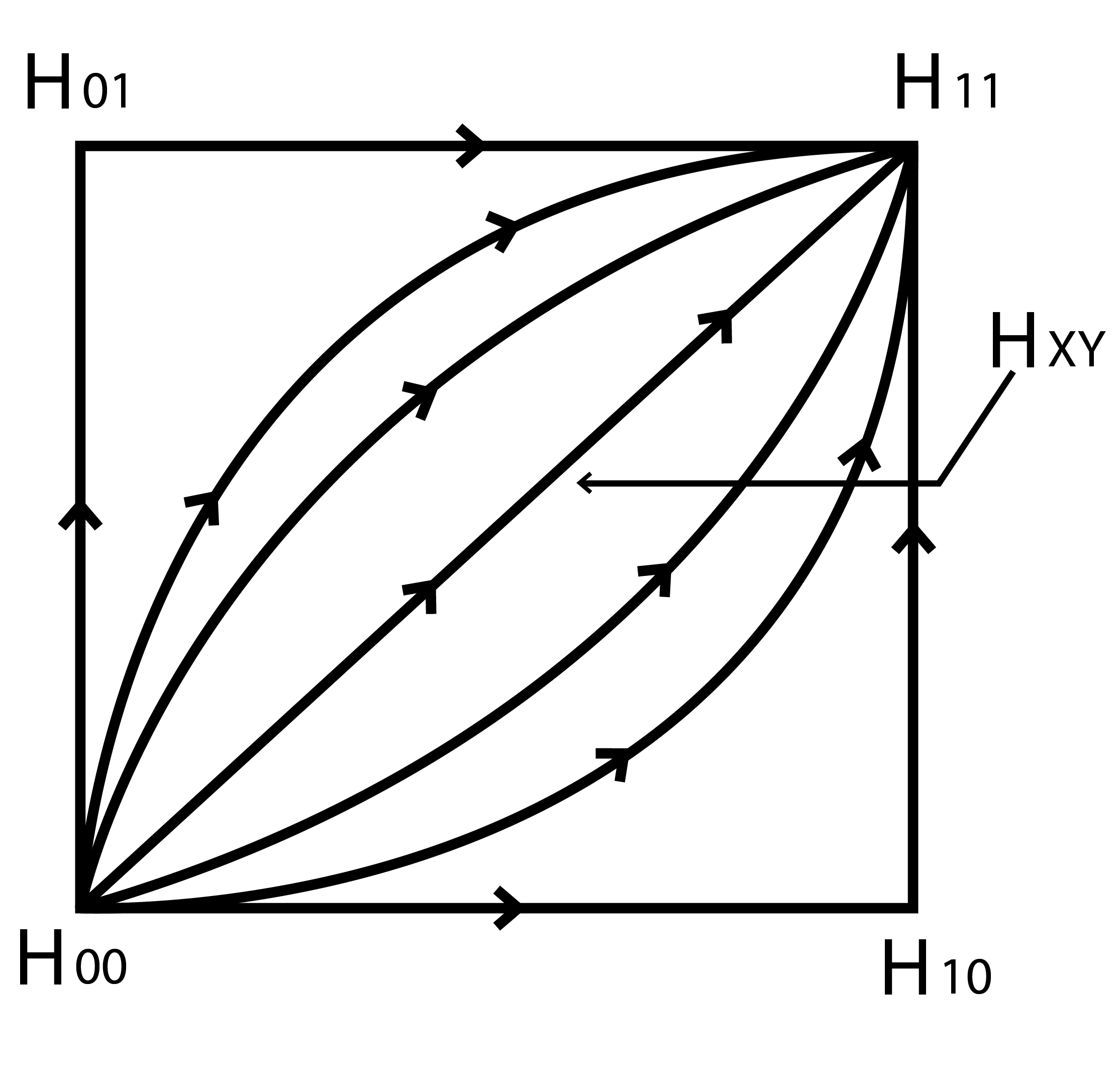}
\end{minipage}
\begin{minipage}{0.5\textwidth}

A square family of Hamiltonians as depicted on the left gives rise to a $2$-cube of chain complexes as below. Note that the homotopy is (intuitively) defined by counting the rigid solutions in the one parameter family of continuation map equations.
\begin{align}
\xymatrix{ 
CH(H_{00})\ar[r]\ar[d]\ar[dr]& CH(H_{10})\ar[d]\\ CH(H_{01})\ar[r] &CH(H_{11})}.
\end{align}
\end{minipage}}
\end{figure}
We want to make three remarks about these virtual counts:
\begin{lemma}\label{c3lVFC}
\begin{enumerate}
\item If the compactified moduli space of stable Floer trajectories $\bar{\mathcal{M}}(\gamma,\gamma',A,H,J,F)$ (as in Definition 10.2.3 iv. of \cite{P}) is empty for some homotopy class $A$, then the virtual count $\#^{vir}\mathcal{M}(\gamma,\gamma',A,H,J,P,F)$ is zero, for any $P$ and coherent orientations. 
\item If a compactified moduli space of stable Floer trajectories consists of one point and that point is regular, then the virtual number associated to it is non-zero. This is a consequence of Lemma 5.2.6 of \cite{P}.
\item If the virtual dimension of a moduli space is not equal to zero, then the virtual counts necessarily give zero.\qed
\end{enumerate}
\end{lemma}

In particular,  if $\#^{vir}\mathcal{M}(\gamma,\gamma',A,H,J,P,F)\neq 0$, then, by (1) of Lemma \ref{c3lVFC} and the computation shown in the item (4) of Section \ref{c3sconventions},  \begin{align}\label{c3etope}
\omega(A)+\int_{S^1}\gamma'^*H'-\int_{S^1}\gamma^*H\geq \int\abs{\frac{\partial u}{\partial s}}^2dsdt+\int(\partial_sH_s)dsdt,
\end{align}
where there exists a broken flow line of $f$ in $Cube_n$ with intermediate vertices $v_1,\ldots ,v_i$ (possibly equal to each other, $v_{in(F)}$ or $v_{ter(F)}$), and $$u:\underbrace{\mathbb{R}\times S^1\sqcup\ldots \sqcup\mathbb{R}\times S^1}_{i+1\text{ many}}\to M$$ is a of solution of the Floer equations (as dictated by the broken flow line) going from $\gamma$ to $\gamma_1$, $\gamma_1$ to $\gamma_2$, $\ldots$, $\gamma_i$ to $\gamma'$, for some $\gamma_i$, a one-periodic orbit of the Hamiltonian at $v_i$. Note that we only have an  inequality because we are not considering the energy of the sphere bubbles on the right hand side. We call this the \textbf{energy inequality}. We have already alluded to a special case of this inequality once in the discussion of the differential, where the second term on the right is zero.

The upshot for us is that these (weighted) counts fit together to give an $n$-cube (over $\Lambda$, just so that we can make this point without introducing monotonicity): the chain complexes at the vertices are the Hamiltonian Floer cochain complexes (only for this paragraph: tensored with $\Lambda$); at the edges we have what is known as continuation maps; and higher dimensional faces give a hierarchy of homotopies as in the definition of an $n$-cube. Instead of showing this from scratch, we deduce it from Pardon's results for simplex families in Appendix A.

\begin{figure}[!h]
\centering
\includegraphics[scale=0.2]{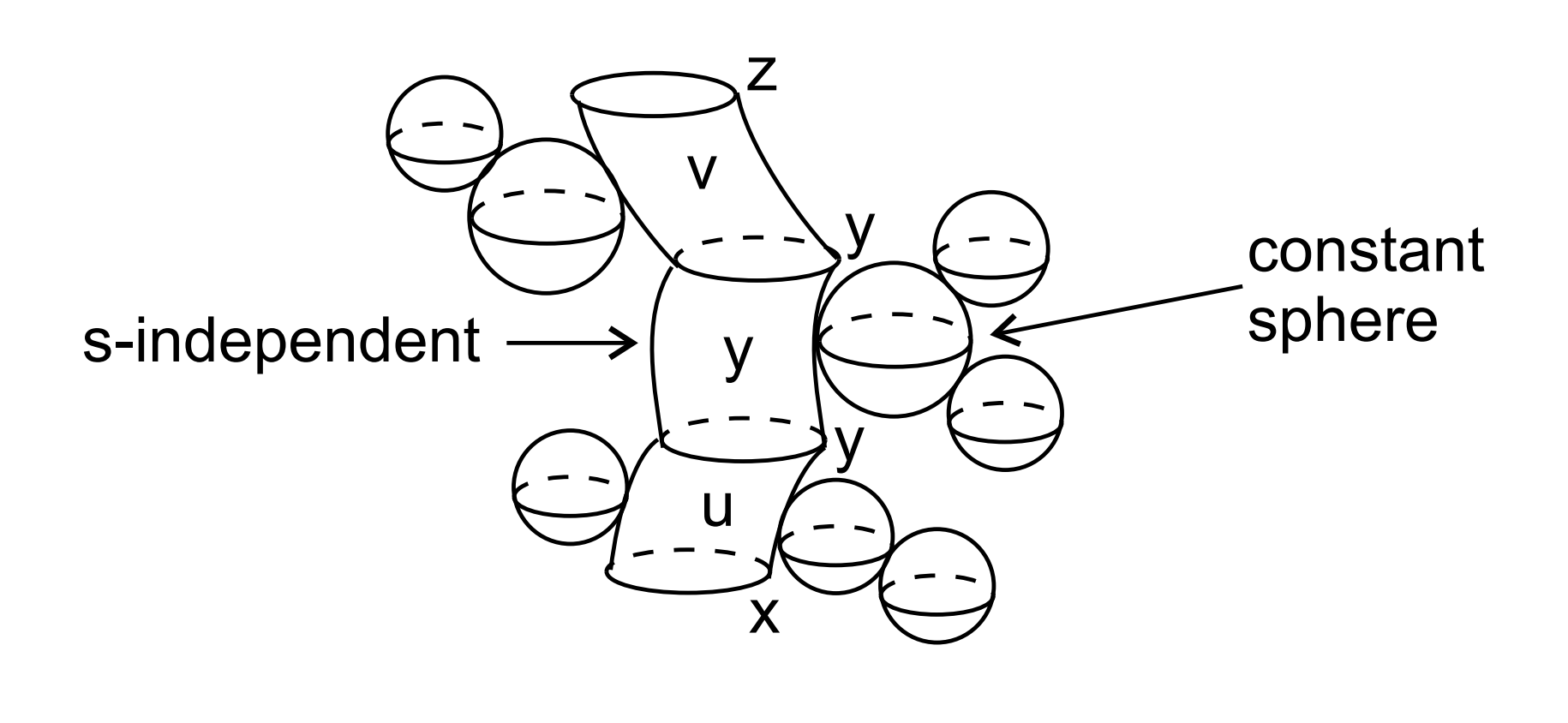}
\caption{Picture of a stable Floer trajectory (taken from Salamon \cite{Sa})}
\label{c3fsalamon}
\end{figure}

\begin{remark}\label{remarkextrachoices}
Whenever we pass from a family of Hamiltonians to a diagram of chain complexes we have to make choices of almost complex structures, Pardon data, and coherent orientations. Our final statements do not depend on these choices. In proofs and constructions all we need is their existence. We can handle these choices in two different ways (1) make a universal choice once and for all, or (2) make the choices inductively whenever you need one as in Pardon \cite{P1}. We generally suppress these choices and omit them from the labeling of the diagrams. 
\end{remark}

\subsubsection{Monotone families}\label{c3ssmonotone}

\begin{definition}\label{c3dmonotone} We call an $n$-cube family of Hamiltonians \textbf{monotone} if the Hamiltonians are non-decreasing along all of the flow lines of $f$ (as defined in (\ref{c3emorse})). By the energy inequality (\ref{c3etope}), a monotone $n$-cube of Hamiltonians gives rise to an $n$-cube defined over $\Lambda_{\geq 0}$. 
\end{definition}

We will also use two other shapes $Triangle^n$ and $Slit^n$ which are subsets of $Cube^n$. These are used to define $n$-triangles and $n$-slits of chain complexes (always over $\Lambda_{\geq 0}$). 

\begin{figure}
\includegraphics[scale=0.3]{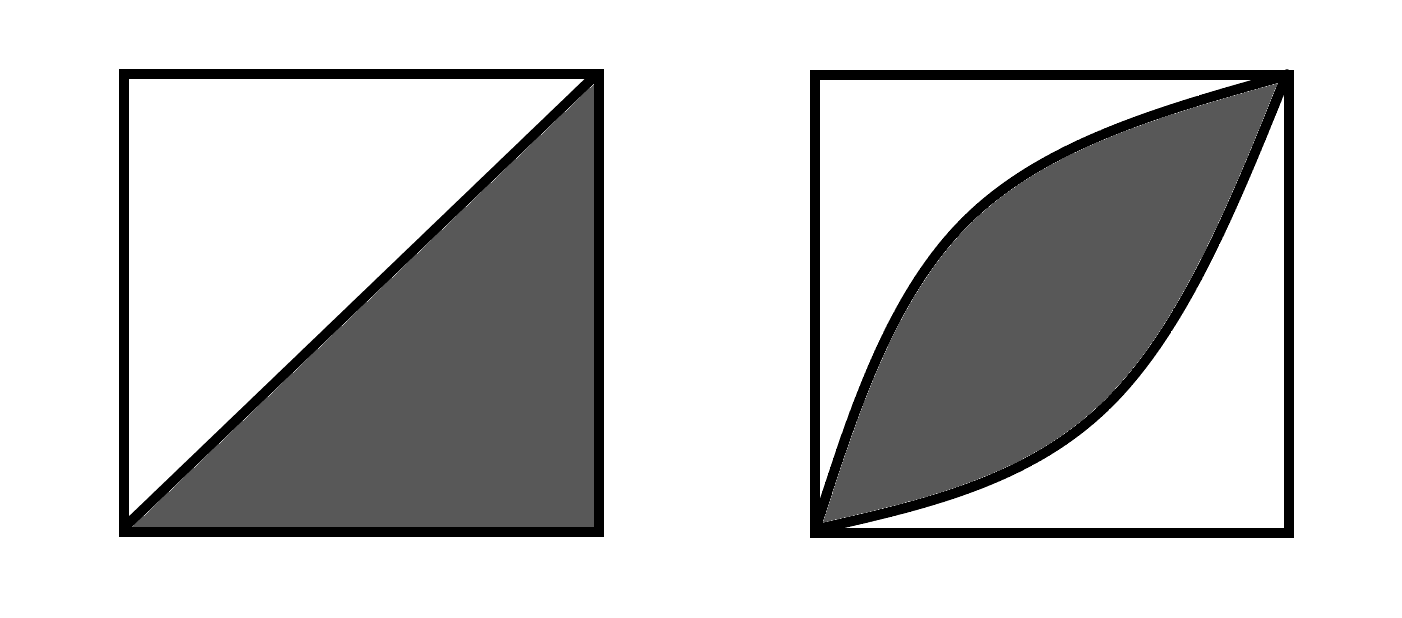}
\caption{On the left $Triangle^2$ and on the right $Slit^2$ is pictured.}
\label{slit}
\end{figure}

We define $Triangle^2:=\{x_1\geq x_2\}\subset Cube^2$ and $Slit^2$ to be the closed region that lies between the flow lines of $f$ (as in Equation \ref{c3emorse}) that pass through the points $(1/3,2/3)$ and $(2/3,1/3)$, see Figure \ref{slit}. Then, for $n\geq 0$ we define $Triangle^n:=Cube^{n-2}\times Triangle^2$ and $Slit^n:=Cube^{n-2}\times Slit^2$, which are both seen as subsets of $Cube^n=Cube^{n-2}\times Cube^{2}$. The gradient flow of $f$ is tangent to $Triangle^n$ and $Slit^n$ as well. The notion of monotonicity is defined in the same way as we did for the cubes. Families of Hamiltonians parametrized by these shapes give rise to special $n$-cubes as in Section \ref{c2ssmaps}: 
\begin{itemize}
\item $Slit^n$ gives two $(n-2)$-cubes, two maps between them, and a homotopy between the two maps, i.e. an $n$-slit.
\item $Triangle^n$ gives three $(n-2)$-cubes, three maps between them as dictated by the connections in the triangle, and a filling of the remainder of the diagram, i.e. an $n$-triangle.
\end{itemize}

We note that in our framework these statements do not follow from the fact that $n$-cubes of Hamiltonians give rise to $n$-cubes of chain complexes. See Appendix \ref{appa}.

\begin{remark} There is an alternative way of obtaining slightly generalized $n$-slits and $n$-triangles which would eliminate the need for a special treatment (but has its own drawbacks). Let us define a cubical $Slit^n$-family of Hamiltonians to be an $n$-cube of Hamiltonians such that the Hamiltonian is independent of the $x_n$-coordinate along the faces $\{x_{n-1}=0\}$ and $\{x_{n-1}=1\}$. Similarly, a cubical $Triangle^n$-family of Hamiltonians to be an $n$-cube of Hamiltonians such that the Hamiltonian is independent of the $x_n$-coordinate along the face $\{x_{n-1}=0\}$. We claim that $Slit^n$-families give rise to a slightly generalized version of $n$-slits, and $Triangle^n$-families to $n$-triangles.

Slightly generalizing the situation, let $H$ be an $n$-cube of Hamiltonians that is independent of the coordinate $x_n$. Let us also choose an $n$-cube family of almost complex structures which is independent of the $x_n$ coordinate. Then given Pardon data and coherent orientations for the $(n-1)$-cube with $H$ and $J$, one should be able to make a choice of Pardon data and coherent orientations so that Hamiltonian Floer theory gives us an $n$-cube which is a self map of an $(n-1)$-cube that is homotopic to the identity map. We do not attempt a proof as we will not use this approach, and indeed, it is not entirely obvious how to do this in Pardon's framework.  Assuming this could be done we would then define a notion of generalized $n$-slits and $n$-triangles to allow for the extra flexibility. We note that if we could rely on standard transversality techniques, we could make this map $id$ on the nose (as in \cite{GPS}, Equation 4.54 in page 78).

Let us fix a continuous map from $Cube^2$ to $Slit^2$, which contracts the sides with constant $x_1$ coordinate to the vertices of $Slit^2$, and is a diffeomorphism in the complement of those sides. Using this map we can find a homeomorphism between cubical $Slit^n$-families of Hamiltonians and Hamiltonians parametrized by $Slit^n$. Consequently, we could relate the two constructions. We can make similar statements for $Triangle^n$. We do not use this method and treat $Slit^n$ and $Triangle^n$ as shapes of their own and prove that they lead to $n$-slits and $n$-triangles separately. 
\end{remark}

\subsubsection{Contractibility}\label{c3sscontractibility}

The goal of this section is to create a framework in which we can systematically prove statements similar to the following claim:

\begin{claim}\label{fillablecube}
Let $H:\partial Cube^n\times M\times S^1\to \mathbb{R}$ be a continuous function so that it is a monotone $(n-1)$-cube family of Hamiltonians on each $(n-1)$-dimensional face of $Cube^{n}$. Then, we can extend $H$ to a monotone $n$-cube family of Hamiltonians.
\end{claim}

Without the monotonicity condition this would be fairly simple (see Lemma 16.8 of \cite{Lee} for example). We generalize the framework first. We refer to Melrose (\cite{Mel}, Chapters 1 and 5) for our conventions and basic results regarding manifolds with corners.

\begin{definition}\label{def}
Let $Y$ be a compact manifold with corners (Definition 1.8.5 in \cite{Mel}). We call a Morse function $h$ and a Riemannian metric $g$ on $Y$ \textbf{admissible} if they satisfy the following conditions.
\begin{enumerate}
\item For any boundary face $Z$ (definition at the end of page I.13 of \cite{Mel}) of $Y$ the directional derivative of $h$ in the normal directions to $Z$ is zero. This implies that the negative gradient vector field of $h$ is tangent to all the boundary faces, and in particular that $h$ has critical points at the $0$-dimensional boundary faces of $Y$.
\item For every boundary face $F$ of $Y$, $h|_F$ has exactly one local minimum and one local maximum. Let us denote the maximum point of $h$ by $a\in Y$, and the minimum point by $b\in Y$.
\item $h$ has no other critical points than the $0$-dimensional boundary faces.
\item The pair $(h,g)$ is locally trivial at the critical points. Namely, if $p$ is a critical point, which implies that it is a $0$-dimensional boundary face, then there exists local coordinates centered at $p$, $(x_1,\ldots ,x_n): U\to \mathbb{R}_{\geq 0}^n$, and a $1\leq k\leq n$, such that $$h(x_1,\ldots ,x_n)=h(p)+x_1^2+\ldots +x_k^2-x_{k+1}^2-\ldots -x_n^2,$$ $$g=dx_1^2+\ldots +dx_n^2.$$ Note that this implies that the intersection of unstable, and stable sets of $p$ with $U$ are given by the (local) faces $$\{x_1=\ldots =x_k=0\},$$ and $$\{x_{k+1}=\ldots =x_n=0\},$$ respectively.
We call the unique face of $Y$ that contains the local boundary face $\{x_1=\ldots =x_k=0\}$ the \textbf{unstable face} of $p$. The \textbf{stable face} of $p$ is defined similarly.
\item Let $p$ be a critical point. Then, $a$ is contained in its stable face and $b$ is contained in its unstable face.
\item Morse-Smale condition is satisfied. Namely, if $p\neq q$ are any two critical points, then the unstable face of $p$ and the stable face of $q$ intersect transversely as $p$-submanifolds. See Definition 1.7.4 of \cite{Mel} for what is a $p$-submanifold, and equation (1.9.3) on page I.15 for what it means for two of them to intersect transversely. Note that boundary faces are $p$-submanifolds.
\end{enumerate}
\end{definition}

Note that for $Y$ equal to one of $Cube^n$, $Triangle^n$, or $Slit^n$, the Morse function $f$ as in Equation \ref{c3emorse} and the flat metric are admissible. 

\begin{remark}
Considering Theorem 3.5 of \cite{Qin}, for any locally trivial Morse-Smale pair in a smooth manifold, the restriction of the Morse pair to $\widehat{W(p,q)}$, using the notation of \cite{Qin},  gives an example of an admissible Morse pair on the manifold with corners $\widehat{W(p,q)}$. By periodically extending $f$ to the entire $\mathbb{R}^n$, $Cube^n$ is seen to be of this form. More examples are obtained by restricting to invariant subsets (under the gradient flow) of $\widehat{W(p,q)}$. $Triangle^n$ and $Slit^n$ are of this form.
\end{remark}

Let us note the following lemma even though we will not use it.

\begin{lemma}
Let $Y$ be a compact manifold with corners. Assume that we have a Morse function $h$ and a Riemannian metric $g$ on $Y$ which are admissible.

Let $p$ and $q$ be critical points of $h$, and let $Y(p,q)$ be the subset of $Y$ that is the intersection of the unstable face of $p$ and the stable face of $q$. Then, $Y(p,q)$ is a boundary face, and hence is a manifold with corners itself. Moreover, the restriction of $h$ and $g$ to $Y(p,q)$ is admissible.
\end{lemma}
\begin{proof}
Being the intersection of transversely intersecting boundary faces, $Y(p,q)$ is a disjoint union of boundary faces. But, it is easily seen that $Y(p,q)$ is connected, so the first claim follows. We omit the straightforward proof of the second claim.
\end{proof}


Let us first introduce some notation regarding blow-ups (see Chapter 5 of \cite{Mel}). Let $W$ be a manifold with corners and $Z$ be a $p$-submanifold. Then, we can define a canonical blow-up manifold with corners $[W;Z]$, which comes with a canonical blow-down map $\beta:[W;Z]\to W$. Note that $\pi$ is a diffeomorphism onto its image away from $\beta^{-1}(Z)$, and $\beta^{-1}(Z)$ is the inward pointing part of the spherical normal bundle, called front face of the blow-up in Melrose. 

Now, let $Z'$ be another $p$-submanifold of $W$, and assume that it is transverse (in the $p$-submanifold sense) to $Z$. We define the proper transform $\beta^*Z'$ of $Z'$ to be the closure of $\beta^{-1}(Z'-Z)$ inside $[W;Z]$, which is also a $p$-submanifold. Note that a disjoint union of boundary faces of a manifold with corners is a $p$-submanifold.

A key property of blow-up is commutativity. This means that, continuing the same notation above, $$[[W;Z];\beta^*Z']=[[W;Z'];\beta^*Z].$$ Note that we denoted both blow-down maps by $\beta$ abusing notation - we will continue doing this. Let us denote such an iterated blow-up by $[W;Z,Z']$. We have that the diagram of blow-down maps
\begin{align}
\xymatrix{ 
[W;Z,Z']\ar[r]\ar[d]& [W;Z]\ar[d]\\ [W;Z']\ar[r] &W}
\end{align}
commutes. 

Finally, we note the following elementary lemma.

\begin{lemma}\label{c3lfaces} Under the blow-up of a boundary face, the proper transforms of boundary faces are boundary faces, and transversal boundary faces remain transversal.
\end{lemma}

If $Y$, $h$, and $g$ are as above, for any two points $p$ and $q$ in $Y$, let us denote the set of unparametrized, possibly broken, negative gradient flow lines from $p$ to $q$ by $M(p,q)$. The following proposition is a generalized and strengthened version of the construction of the map $\phi$ on page $9$ of \cite{CJS}. The proof was inspired by the construction of manifolds $P_i$ from Observation $8$ on page $21$ of \cite{Bur}.

\begin{proposition}\label{c3prectify} Let $Y$ be a compact manifold with corners. Assume that we have a Morse function $h$ and a Riemannian metric on $Y$ which are admissible. Denote the maximum, and minimum of $h$ by $a$ and $b$. Then, there exists a canonical manifold with corners $X$, and a canonical surjective continuous map $$s:X\times [h(b),h(a)]_r\to Y,$$ called the \textbf{rectification} such that:\begin{itemize} 
\item $s$ is smooth at $x\in X\times [h(a),h(b)]$, if $s(x)$ is not a critical point of $h$.
\item $s(x,h(a))=a$ and $s(x,h(b))=b$ for every $x\in X$
\item For every $x\in X$, $h\circ s|_{\{x\}\times [h(a),h(b)]}$ is strictly increasing in the $r$ coordinate, and, most importantly, $s_x:=s|_{\{x\}\times [h(b),h(a)]}$ satisfies the equation:
 $$\frac{\partial s_x}{\partial r}(r)=\frac{V(s(r))}{\abs{V(s(r))}^2}$$ for all points $r\in [h(a), h(b)]$ that are not mapped to critical points of $h$, where $V$ is the gradient vector field of $h$. 
\item By the previous item, for every $x\in X$, there exists a broken negative gradient flow line $\gamma_x$ from $a$ to $b$ of $h$ such that the image of $s_x$ is the closure of the image of $\gamma_x$. This defines a map $X\to M(a,b)$ and we require it to be a bijection.
\item For every critical point $p$ of $h$, which is not $a$ or $b$, there exists a boundary hypersurface $F_p$ of $X$ such that $s^{-1}(p)=F_p\times \{h(p)\}$.
\end{itemize}
\end{proposition}
\begin{proof}
Let us denote the flow of the smooth vector field $-\frac{V(s(r))}{\abs{V(s(r))}^2}$, defined in the complement of the critical points, by $\psi^t$ (see \cite{Audin}, proof of Theorem 2.1.7, or pages 20-21 of \cite{Bur} for a discussion of this flow). If $y\in Y$, and $t_0\in \mathbb{R}_{\geq 0}$ is such that $\psi^t(y)$ exists (in other words, never hits a critical point) for all $0\leq t\leq t_0$, then $$h(\psi^t(y))=h(y)-t.$$For $w\in [h(b),h(a)]$, we define $M(w):=h^{-1}(w)$, and $c(w)$ to be the set of critical points with value $w$. Moreover, we denote the stable and unstable faces of a critical point $p$ by $SF(p)$ and $UF(p)$.

Let us first prove that if $c(w)$ is empty than $M(w)$ is a $p$-submanifold of $Y$. Let $h(a)=c_0>c_1>\ldots>c_k>c_{k+1}=h(b)$ be all of the critical values of $h$. 

We start by showing that for $c_0>w>c_1$ the statement is true. For $w$ really close to $a$ we are inside the local model at $a$, and the level set is the intersection of a sphere centered at the origin with $\mathbb{R}_{\geq 0}^n$ inside $\mathbb{R}_{\geq 0}^n$. This is clearly a $p$-submanifold. For the larger values of $w$ in the same interval, one then transports the $p$-submanifold charts via $\psi^t$. 

A similar strategy works for $c_i>w>c_{i+1}$, for all $1\leq i\leq k$, by an inductive argument. Consider a $w$ that is very close to $c_i$ in this interval. It is easy to get $p$-submanifold charts for points that are contained in $M(w)\cap UF(p)$, for some $p\in c(c_i)$, as we are inside one of the local models (see Appendix C for a more detailed analysis). For all the other points of $M(w)$ one can bring the charts from $M(w')$, where $c_{i-1}>w>c_{i}$, using $\psi$. This finishes the proof for $w$ that is very close to $c_i$, and then we can again move charts to other values via $\psi^t$. 

Moreover, for $c_{i-1}>w'>c_i>w>c_{i+1}$, the local model can be used to show that there exists a manifold with corners $M_i$, which is the blow-up of $M(w)$ at the union of its disjoint boundary faces $\bigcup_{p\in c(c_i)}(M(w')\cap UF(p))$ and also at the same time the blow-up of $M(w')$ at the union of its disjoint boundary faces $\bigcup_{p\in c(c_i)}(M(w)\cap SF(p))$ (see Appendix C for details).

Now, let us choose $c_{0}>w(0)>c_1>w(1)>c_{2}>w_2>\ldots>c_{k-1}>w_{k-1}>c_k>w_k>c_{k+1}$. We will construct the diagram below where each entry is a manifold with corners, and each arrow is a blow-down map (which is associated to the blow-up of a union of disjoint boundary faces). In the end $E_{k,k}$ will be our $X$. We refer to this diagram as the pyramid.

\begin{align}\label{c3epyramid}
\xymatrix@!0@R=10mm@C=10mm{ 
&     & &      & &&E_{k,k}\ar[dl]\ar[dr] & &\\ 
&     & &      & &E_{k-1,k-1}\ar[dl]\ar[dr]& &E_{k-1,k} \ar[dl]\ar[dr]&\\ 
&     & &      &  & &&&\\ 
&     & &      & \ldots & &\ldots&&\ldots\\ 
  &&     &E_{3,3} \ar[dl]\ar[dr] &        &  \ar[dl] & \ldots && &E_{3,k}\ar[dr] \ar[dl] \\ 
&     &E_{2,2} \ar[dl]\ar[dr]&      &E_{2,3} \ar[dl]\ar[dr] &  &  \ar[dl]& \ldots&& &E_{2,k}\ar[dr] \ar[dl]\\ 
           &M_1\ar[dl]\ar[dr] &     &  M_2\ar[dl]\ar[dr]    &&M_3\ar[dl]\ar[dr]& & \ar[dl]&\ldots&  & &M_k\ar[dr]  \ar[dl]\\ 
 M(w_0) &     & M(w_1) &      &M(w_2)  &  & M(w_3)& & & \ldots&& &M(w_k)\\ 
}
\end{align}

Let $U_i\subset M(w_i)$ be the disjoint union $\bigcup_{p\in c(c_i)}(M(w_i)\cap UF(p))$ of faces of $M(w_i)$, for every $i\in\{0,\ldots, k\}$, and let $S_i\subset M(w_{i-1})$ be the disjoint union $\bigcup_{p\in c(c_i)}(M(w_{i-1})\cap SF(p))$ of faces of $M(w_{i-1})$ for every $i\in\{1,\ldots, k+1\}.$ Note that $U_0=M(w_0)$ and $S_{k+1}=M(w_k)$, and all others contained in the boundary. Moreover, by the Morse-Smale property, $U_i$ and $S_{i+1}$ are transverse to each other. This is summarized below:
\begin{align}
\xymatrix@!0@R=10mm@C=10mm{ 
 M(w_0) &     & M(w_1) &      &M(w_2)  &  & M(w_3)& & & \ldots&& &M(w_k)\\ 
 U_0 \pitchfork S_1&     & U_1\pitchfork S_2 &      &U_2\pitchfork S_3  &  & U_3\pitchfork S_4& & & \ldots&& &U_{k}\pitchfork S_{k+1}
}
\end{align}

We restate the definition of $M_i$'s using our new notations:
$$M_i:=[M(w_{i-1}); S_i]=[M(w_{i}); U_i],$$ for every $i\in\{1,\ldots, k\}.$ This explains the bottom two rows of our pyramid. Also, note that for $1\leq i\leq k$, $M_i$ contains the proper transforms of $U_{i-1}$ and $S_{i+1}$. These are also disjoint unions of (positive codimension) boundary faces (being proper transforms of boundary faces transverse to the blow-up boundary faces). Moreover, by the Morse-Smale property, they are transverse to each other.

Now let us define, for $2\geq i\geq k$, $$E_{2,i}:=[M_{i};\beta^*U_{i-1}]=[M_{i-1};\beta^*S_i].$$ For, $2<i<k$, $E_{2,i}$ contains the proper transforms of $U_{i-2}$ and $S_{i+1}$. These are also disjoint unions of boundary faces, which are transverse to each other, by the Morse-Smale property. We also define $E_{0,i}:=M(w_0)$, for $i\in\{0,\ldots, k\}$, and $E_{1,i}:=M_i$, for $i\in\{1,\ldots, k\}$ to have uniformity of notation. We construct the entire pyramid similarly, by induction.

We define, for every $j\geq 1$ and $j\leq i\leq k$, $$E_{j,i}:=[E_{j-1,i};\beta^*U_{i-j+1}]=[E_{j-1,i-1};\beta^*S_i].$$ Using Morse-Smale property, we can show that for $j\leq i\leq k$, the proper transforms of $U_{i-j}$ and $S_{i+1}$ are disjoint unions of boundary faces of $E_{j,i}$, which are transverse to each other. The last step defines $E_{k,k}:=[E_{k-1,k};\beta^*U_{1}]=[E_{k-1,k-1};\beta^*S_k]$, and we finish the procedure. We define $X:=E_{k,k}$.

Finally, we define the rectification map $s:X\times [h(b),h(a)]\to Y.$ First note that for each $w\in [h(b),h(a)]$ that is not a critical value, $X$ has a canonical map to $M(w)$. This is because for $c_i>w>c_{i+1}$, using the pyramid, we have a canonical iterated blow-down map $X\to M(w_i)$. Using the flow $\psi^t$ (or its inverse) this gives us the desired maps for every regular value $w$. Also note that, we have canonical continous maps: 
\begin{align}
\xymatrix@!0@R=10mm@C=10mm{ 
 M(w_0)\ar[dr] &     & M(w_1)\ar[dl]\ar[dr] &      &M(w_2) \ar[dl]\ar[dr] &  & M(w_3) \ar[dl]\ar[dr]& &  \ar[dl]& \ldots&& &M(w_k)\ar[dl] \\ 
&M(c_1)&     &  M(c_2)    &&M(c_3)& &&\ldots&  & &M(c_k),
}
\end{align} which result in a commutative diagram if added to the pyramid. These maps are defined by following $\psi^t$ (or its inverse) except that for the right pointing diagonal arrows, the intersections with the corresponding stable faces,  and for the left pointing diagonal arrows) he intersections with the unstable ones are collapsed into their critical points. We can put all these maps together and define $s:X\times [h(b),h(a)]\to Y.$ It is straightforward to check that all the properties listed are satisfied.
\end{proof}

\begin{remark}
What happens to a level set as it crosses a critical point (as in the proof above) is analogous to an Atiyah flop. The bottom two rows of the pyramid in Equation \ref{c3epyramid} can be seen as a sequence of flops from a simplex to itself. See Figure \ref{c3fflop}. This is analyzed in details in Appendix C. The idea behind the construction of the manifold with corners structure on the set of broken flow lines is to not do the blow-down step, and keep blowing-up as we cross critical points.
\end{remark}
\begin{figure}[!h]\label{c3fflop}
\centering
\includegraphics[scale=0.5]{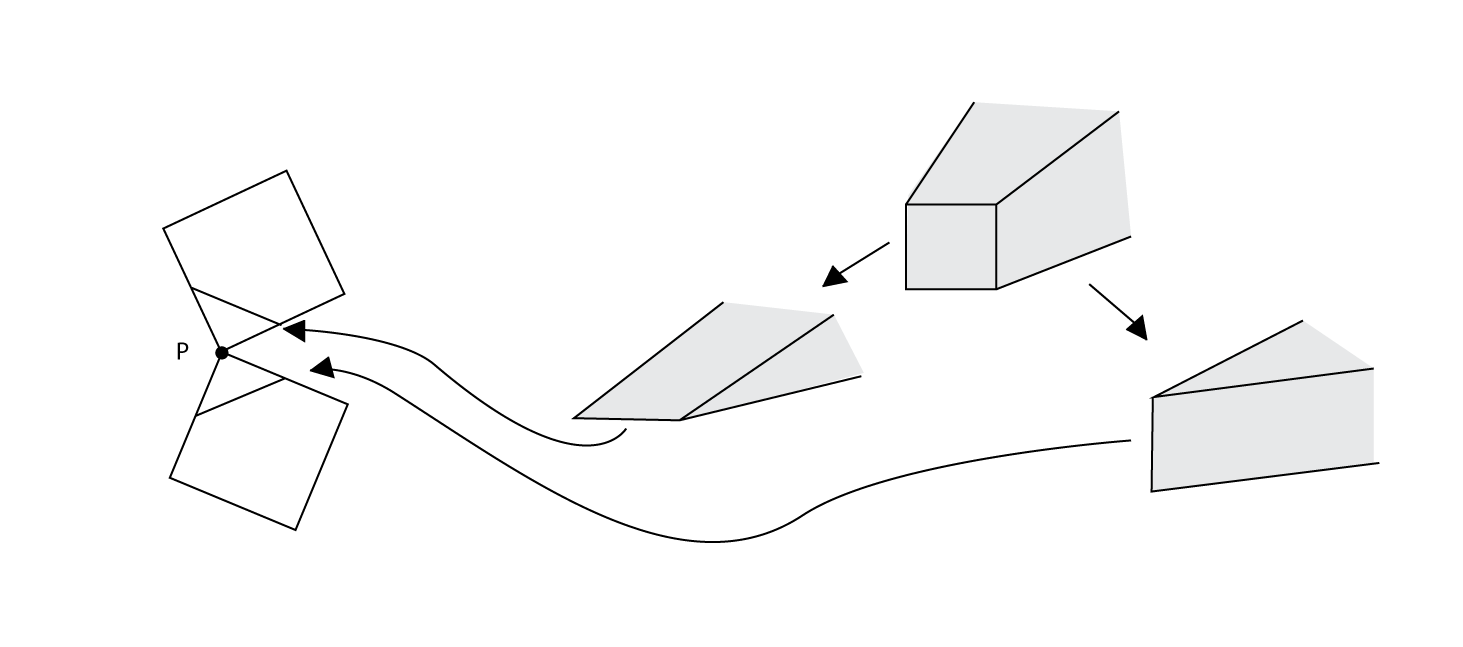}
\caption{Here $p$ is a critical value with $2$ dimensional stable and unstable faces, which are pictured on the left. On the right of the picture, we see the local modification that the level sets undergo after crossing $p$.}
\end{figure}

\begin{example}
Let us give an example of the procedure described in the proof. Consider $Y=Cube^4$ with $f$ as the Morse function as before. Then, $X$ is constructed by starting with a three dimensional simplex and doing the following blow-ups. We first blow-up all corners of the simplex, and then we blow-up the proper transforms of the original edges of the simplex. What we obtain is a standard permutohedron as expected (compare with \cite{Sar}, e.g. the last paragraph in the proof of Lemma 4.3). In Figure \ref{c3fblowup}, we presented what the procedure does to obtain $X$ for $Y$ being the $4$-dimensional simplex $\{(x_1,\ldots,x_4)\mid 0\leq x_1\leq\ldots\leq x_4\leq 1\}\subset [0,1]^4$ with restriction of $f$ as the Morse function and flat metric. The result is as expected a three dimensional cube (see Section 10.1 in \cite{P}).
\end{example}
\begin{figure}[!h]\label{c3fblowup}
\centering
\includegraphics[scale=0.5]{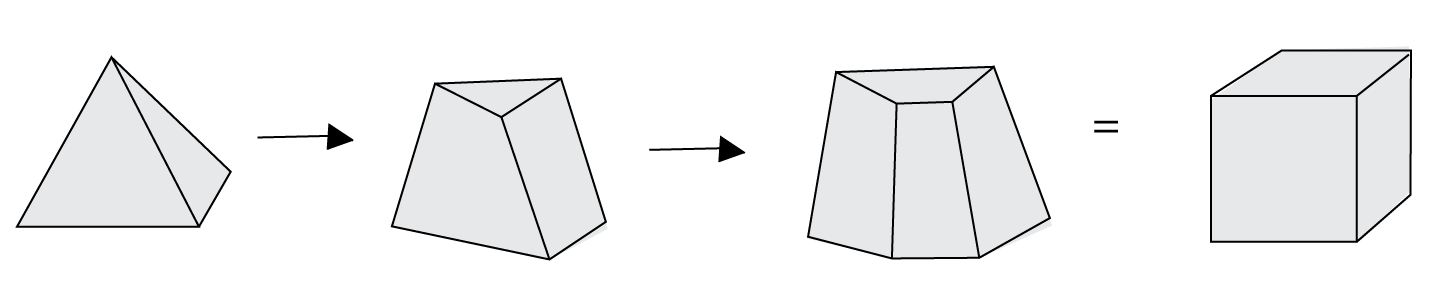}
\caption{}
\end{figure}


\begin{definition}\label{c3d2}Let $X$ be a manifold with corners. We define an \textbf{$X$ family of homotopy of Hamiltonians with stations} between $H_0: M\times S^1\to \mathbb{R}$ and $H_1:M\times S^1\to \mathbb{R}$ as a smooth map $$H:X\times [0,1]\times M\times S^1\to \mathbb{R}$$ such that $\{x\}\times\{0\}\times M\times S^1\to\mathbb{R}$ is $H_0$ for all $x\in X$, and $\{x\}\times\{1\}\times M\times S^1\to\mathbb{R}$ is $H_1$ for all $x\in X$. Moreover, we are given a subset $S\subset X\times [0,1]$ (the stations) satisfying the conditions: \begin{itemize}
\item There exists numbers $0<s_1<\ldots <s_k<1$ and faces $F_1,\ldots F_k$ of $X$ such that $S=\bigcup_{i=1}^k F_i\times\{s_i\}$.
\item There exists a neighborhood $U$ of $S\cup X\times\{0\}\cup X\times\{1\}$ in $X\times [0,1]$ such that for every $x\in M$ and $t\in S^1$, $H\mid_{U\times\{x\}\times\{t\}}$ is locally constant.\end{itemize}

We say such family is \textbf{monotone} if it is increasing in the $[0,1]$-direction.
\end{definition}

\begin{remark} Let $H$ be a monotone $X$ family of homotopy of Hamiltonians with stations between $H_0$ and $H_1$. Let us denote the coordinate in the $[0,1]$-direction by $r$. Choose a smooth function $\rho:X\times [0,1]\to \mathbb{R}$ such that: \begin{itemize}
\item $\rho\geq 0$,
\item $\rho$ vanishes precisely along $S$,
\item all the integral curves of the vector field $\rho\partial_r$ are defined for all times $(-\infty,\infty)$.
\end{itemize}

A generalized version of Pardon's construction can then be used to produce a diagram of chain maps and a hierarchy of homotopy maps indexed by the faces of $X$ from $CF(H_0)$ to $CF(H_1)$, over $\Lambda_{\geq 0}$. 

Note that if $X\times [0,1]\to Y$ is the rectification of a Morse pair $(h,g)$ on $Y$ (as in Proposition \ref{c3prectify}), then we obtain a canonical $\rho$ as above given by $|grad_g(h)_{s(x,t)}|^2)$.
\end{remark}

Now, we modify the definition of what it means for an $n$-cube of Hamiltonians to be smooth in this more general context. From now on, whenever we talk about about Hamiltonians parametrized by $Cube^n$, $Triangle_n$, or $Slit_n$, we are using this definition.

\begin{definition} Let $Y$ be a manifold with corners and assume that we have a Morse function $h$ and a Riemannian metric on $Y$, which are admissible with the rectification $s:X\times [h(b),h(a)]\to Y$. Then a (monotone) $(X,h,g)$-family of Hamiltonians is a map $H:Y\to C^{\infty}(M\times S^1,\mathbb{R})$, which is constant on an open neighborhood of each of the vertices with non-degenerate Hamiltonians at the vertices, and the induced map $X\times [0,1]\times M\times S^1\to \mathbb{R}$ given by pre-composing with $s$ and the orientation reversing affine identification $\phi:[0,1]\to[h(b),h(a)]$ is smooth (and monotone). Here the stations for $X\times [0,1]\times M\times S^1\to \mathbb{R}$ are precisely at $F_p\times\{\phi^{-1}(h(p))\}$ for $p$ a critical point of $h$.
\end{definition}
\begin{remark}
This is indeed a more flexible notion of smoothness (see Remark \ref{c3rsmooth}). To see this note our $s$ has a behaviour similar to polar coordinates at the preimages of the $1$ or more dimensional faces of $Y$. This can be observed in Figure \ref{filling}, after the family is filled.
\end{remark}

The general contractibility result we have is as follows. 

\begin{lemma}\label{c3lcontractible}
Let $H:\partial X\times [0,1]\times M\times S^1\to \mathbb{R}$ continuous map which is a monotone  family of homotopy of Hamiltonians with stations on each boundary hypersurface of $X$ such that the union of all of the stations $S\subset\partial X\times[0,1]$ is of the form $S=\bigcup_{i=1}^k F_i\times\{s_i\}$, for some numbers $0<s_1<\ldots <s_k<1$ and boundary faces $F_1,\ldots F_k$ of $X$, and if a point is in the station set of a boundary hypersurface, then it is in the station set of all boundary hypersurfaces containing it. Then, $H$ can be extended to a monotone $X$ family of homotopy of Hamiltonians with stations along $S$.
\end{lemma}
\begin{proof}
This is an application of Whitney extension theorem \cite{Wh}, more accurately of the construction that is involved in proving it. We refer to Sections VI.2.1-3 in \cite{St} for the construction (i.e. Equation (8) in \cite{St}) and its properties. 

Let us embed $X$ into a Euclidean space $\mathbb{R}^N$. For every $r\in I$, we choose the same partitions of unity as used in Equation (8) of \cite{St} for $\mathbb{R}^N-\partial X$, and also the same closest points on the boundary for each cube used in the partitions of unity. Then, we extend Hamiltonians seperately for each $r\in I$ to $\mathbb{R}^N$, and restrict to $X$. It is easy to check that the total extension (which is by construction smooth on each $\{r\}\times X$) is smooth on $I\times X$.

The only property to check is constancy near the stations. This is automatically satisfied in our construction using the item (3) in Theorem 1 of Section VI.2.1 in \cite{St}, since we already have constancy near stations in the boundary.
\end{proof}

\subsubsection{Floer theoretic $n$-cubes}

Assume that we have Hamiltonians defined on some union of the faces of $Cube^n$ (called $H$), such that it is defined on all vertices (and is non-degenerate at all of them), it is locally constant in a neighborhood of the vertices, and if it is defined on two points $x$ and $y$ in $Cube^n$ and there exists a possibly broken negative gradient flow line of $f$ from $x$ to $y$, then $H|_y\geq H|_x$. We call such a family monotone as well. Let us call a partial $n$-cube defined by such data \textbf{Floer theoretic}. Note that this definition includes the case of completely defined $n$-cubes. A consequence of contractibility is the following.

\begin{proposition}\label{Floertheoreticprop}
Any partially defined Floer theoretic $n$-cube has a filling to an $n$-cube which is also Floer theoretic. Moreover, the statement extends to partially defined $n$-cubes that are obtained from partial data on $Slit^n$ or $Triangle^n$, which provides fillings that are $n$-slits or $n$-triangles.
\end{proposition}
\begin{proof}
We want to extend the partially defined family of Hamiltonians to a monotone $n$-cube family of Hamiltonians. We start extending on the $1$-dimensional faces of $Cube^{n}$ that are undefined. Because of our monotonicity assumption this can be done easily in such a way that near the vertices the extension is constant. For example, we can do this by fixing a partitions of unity on $[0,1]$: $\rho_i:[0,1]\to [0,1]$, for $i=0,1$, and $\rho_0+\rho_1=1$ such that $\rho_0$ is $1$ near $0$, $0$ near $1$, and non-increasing. 

We go on to $2$-dimensional faces. We take any one of them and rectify that face. Then, we use Lemma \ref{c3lcontractible} to obtain our extension. After we do this for all $2$-dimensional faces, we go to $3$-dimensional faces and so on (see Figure \ref{filling}). The only point to remark is that the smoothness requirement for the functions that are defined on the boundary in Lemma \ref{c3lcontractible} are always satisfied throughout this procedure. This boils down to the fact that if $Z$ and $Z'$ are manifolds with boundary $Z\times Z'\to Z$ is a smooth function. 

This finishes the proof of the first statement, and the second one follows by exactly the same argument.
\end{proof}

\begin{figure}[!h]
\centering
\includegraphics[scale=0.7]{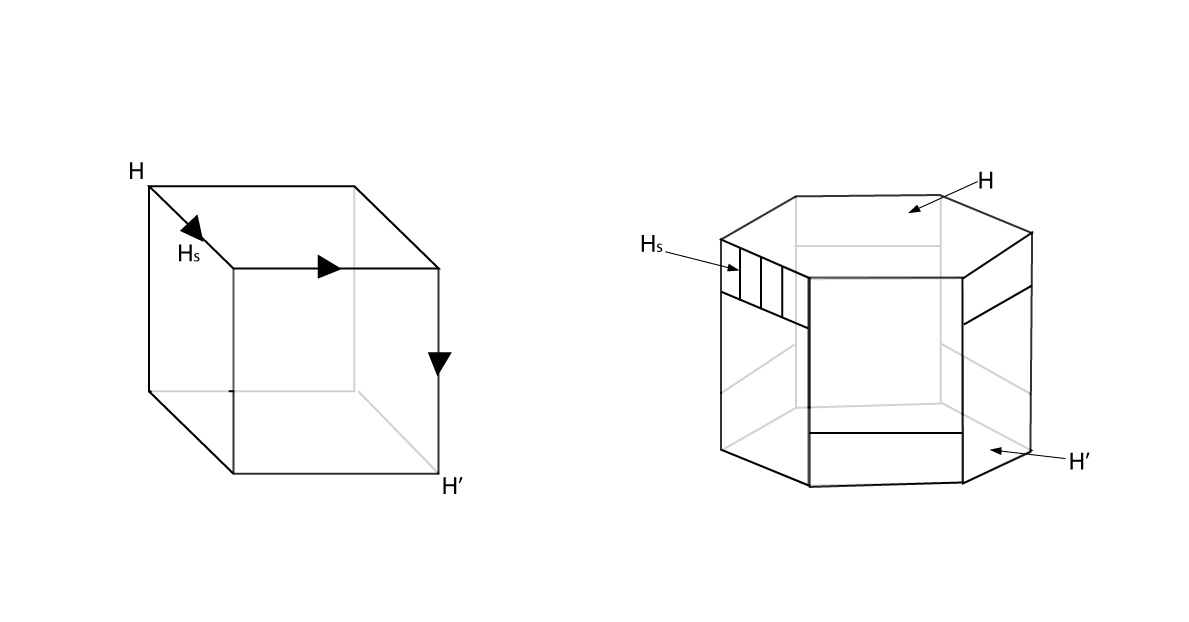}
\caption{On the left we have a family of Hamiltonians defined on the boundary of $Cube^3$ and on the right the rectification of $Cube^3$. The arrows on the cube depict the direction of the Morse flow. The stations on the rectification are the 6 horizontal  lines on the two dimensional faces. We have labeled only part of the Hamiltonians on the cube and showed where they are on the rectified picture. The monotone filling of the cube with Hamiltonians is found by filling the hexagons at every height by a uniform formula using Whitney's technique.}
\label{filling}
\end{figure}

Let us give a corollary of this proposition. Whenever we refer to Proposition \ref{Floertheoreticprop}, we mean that an argument similar to the one used to prove the statement below is used. 

\begin{corollary}\label{Floertheoreticcor}
\begin{itemize}
\item Let $\mathcal{C}$ and $\mathcal{C}'$ be two Floer theoretic $n$-cubes, defined by two monotone $n$-cubes of Hamiltonians, $\mathcal{H}, \mathcal{H}': Cube^n\to C^{\infty}(M\times S^1,\mathbb{R})$ such that $\mathcal{H}\leq \mathcal{H}'$ at every vertex of $Cube^n$. Then, there exists a Floer theoretic $(n+1)$-cube, which is a map from $\mathcal{C}$ to $\mathcal{C}'$. Moreover, any two such maps are homotopic as maps of $n$-cubes.
\item Let $\mathcal{D}$, $\mathcal{D'}$ be Floer theoretic $(n+1)$-cubes defined using monotone $(n+1)$-cubes of Hamiltonians $\mathcal{H}, \mathcal{H}': Cube^{n+1}\to C^{\infty}(M\times S^1,\mathbb{R})$ such that $\mathcal{H}|_{x_{n+1}=1}=\mathcal{H}'|_{x_{n+1}=0}$. Then, $\mathcal{D}$, $\mathcal{D'}$ can be glued in the $(n+1)$st direction. Let us consider them as maps of $n$-cubes which form a diagram: $$\mathcal{C}\to\mathcal{C}'\to\mathcal{C}''.$$ We can then compose these and obtain a map of $n$-cubes $\mathcal{C}\to\mathcal{C}''$. 
Any Floer theoretic $(n+1)$-cube $\mathcal{C}\to\mathcal{C}''$ as in the previous bullet point (which exists) is homotopic to the composition $\mathcal{C}\to\mathcal{C}''$ just described.
\item Let $\mathcal{D}_1\ldots, \mathcal{D}_k$ be Floer theoretic $(n+1)$-cubes defined using monotone $(n+1)$-cubes of Hamiltonians $\mathcal{H}_1\ldots, \mathcal{H}_k: Cube^{n+1}\to C^{\infty}(M\times S^1,\mathbb{R})$ such that $\mathcal{H}_i|_{x_{n+1}=1}=\mathcal{H}_{i+1}|_{x_{n+1}=0}$, for every $1\leq i\leq k-1$. Then, $\mathcal{D}_1\ldots, \mathcal{D}_k$ can be glued in the $(n+1)$st direction in a chain. Let us consider them as maps of $n$-cubes which form a diagram: $$\mathcal{C}_1\to\mathcal{C}_2\to\ldots\to \mathcal{C}_k\to\mathcal{C}_{k+1}.$$ We can then compose these and obtain a map of $n$-cubes $\mathcal{C}_1\to\mathcal{C}_{k+1}$. 
Any Floer theoretic $(n+1)$-cube $\mathcal{C}_1\to\mathcal{C}_{k+1}$ as in the first bullet point (which exists) is homotopic to the composition $\mathcal{C}\to\mathcal{C}''$ just described.
\end{itemize}
\end{corollary}

\begin{proof}
The first statement of the first bullet point follows from the first sentence of Proposition \ref{Floertheoreticprop}. The second one follows from the $Slit^n$ part of Proposition \ref{Floertheoreticprop}. 

First of all, the existence statement in the last two bullet points follow from the first bullet point. The rest of second bullet point follows from the $Triangle^n$ part of Proposition \ref{Floertheoreticprop} along with Lemma \ref{triangletoslit}. The third bullet point follows from iterating the second one. 
\end{proof}

\subsection{Construction of the invariant}\label{c3sconstruction}

\subsubsection{Cofinality}\label{c3sscofinality}
Let $X$ be a closed smooth manifold, and $A\subset X$ be a compact subset. We define $C^{\infty}_{A\subset X}:=\{H\in C^{\infty}(X,\mathbb{R}) \mid H\mid_A<0\}$. Note that $C^{\infty}_{A\subset X}$ is a directed set, with the relation $H\geq H'$ if $H(x)\geq H'(x)$ for all $x\in X$.

\begin{lemma}\label{c3lcofinal}
Let $H_1\leq H_2\leq\ldots$ be elements of $C^{\infty}_{A\subset X}$. They form a cofinal family if and only if $H_i(x)\to 0$, for $x\in A$, and $H_i(x)\to \infty$, for $x\in X-A$, as $i\to \infty$.
\end{lemma}

\begin{proof}
The only if direction is trivial, we prove the if direction. Take any $f\in C^{\infty}_{A\subset X}$, we need to show that there exists an $i>0$ such that $f\leq H_i$.

By compactness (and Dini's theorem), there is a $j>0$ such that $f< H_j$ on $A$. But, then there has to be a neighborhood $U$ of $A$ such that $f< H_j$ on $U$. 

Again, by compactness, there is a $j'>0$ such that $f< H_{j'}$ on $X-U$. Choosing, $i=max(j,j')$ finishes the proof.
\end{proof}

\subsubsection{Definition and basic properties}\label{c3ssdefinition}

Let $M$ be a closed symplectic manifold, $K\subset M$ be a compact subset. We call the following data an \textbf{acceleration data} for $K$:\begin{itemize}
\item $H_1\leq H_2\leq\ldots$ a cofinal family in $C^{\infty}_{K\times S^1\subset M\times S^1}$, where $H_i$ are non-degenerate for all $i\geq 1$.
\item Monotone $1$-cube of Hamiltonians $\{H_s\}_{s\in [i,i+1]}$, for all $i$.
\end{itemize}

Note that acceleration data gives one $\mathbb{R}_{\geq 1}$ family of Hamiltonians, which we will denote by $H_s$. From an acceleration data, we obtain a $1$-ray of chain complexes over $\Lambda_{\geq 0}$: $\mathcal{C}(H_s):= CF(H_1)\to CF(H_2)\to\ldots $.

We define $SC_M(K,H_s):=\widehat{tel}(\mathcal{C}(H_s))$. 

If $H_s$ and $H'_s$ are two acceleration data for $K$ such that $H_n\geq H'_n$ for all $n\in\mathbb{N}$, we can produce a map of $1$-rays $\mathcal{C}(H'_s)\to \mathcal{C}(H_s)$ by filling in the $2$-cubes (using Proposition \ref{Floertheoreticprop}). \begin{align}\label{canonical-1-ray-maps}
\xymatrix{ 
CF(H_1')\ar[r]\ar[d]\ar[dr]& CF(H_2')\ar[d]\ar[r]\ar[dr] &CF(H_3')\ar[r]\ar[d]\ar[dr]& \ldots\\ CF(H_1)\ar[r] &CF(H_2)\ar[r]&CF(H_3)\ar[r]&\ldots}
\end{align}

Note that here we are using Proposition \ref{Floertheoreticprop} in a slightly stronger way than Corollary \ref{Floertheoreticcor} does. This is because of the gluing condition for the $2$-cubes defining a map of $1$-rays. Proposition \ref{Floertheoreticprop} is written so that it covers this situation and we do not make such comments from now on about how we use Proposition \ref{Floertheoreticprop}.

This map is unique up to homotopy of maps of $1$-rays by filling in the $3$-slits using Proposition \ref{Floertheoreticprop}. Therefore, by Lemma \ref{c2ltelescope}, and functoriality (and additivity) of the completion functor on chain complexes, we get a canonical map:
\begin{align}\label{c3ecompare}
H(SC_M(K,H'_s))\to H(SC_M(K,H_s)).
\end{align}

Moreover, if we have $H_n\geq H'_n\geq H''_n$, the canonical triangle is commutative, this time by filling in the $3$-triangles, again by Lemma \ref{c2ltelescope}.

%

\begin{proposition}\label{c3piso}
The comparison maps (as defined in (\ref{c3ecompare})) $H(SC_M(K,H'_s))\to H(SC_M(K,H_s))$ are isomorphisms.
\end{proposition}
\begin{proof}
We can find infinite strictly monotone sequences  $n(i)$ and $m(i)$ of positive integers such that $H_i'<H_i<H_{n(i)}'<H_{m(i)}$. We extend $H_{n(i)}'$ and $H_{m(i)}$ to acceleration data $H^{n'}_s$ and $H^m_s$. We then get three $2$-rays glued to each other by constructing maps of $1$-rays as in Equation \ref{canonical-1-ray-maps}.\begin{align}
\xymatrix{ 
CF(H_1')\ar[r]\ar[d]\ar[dr]& CF(H_2')\ar[d]\ar[r]\ar[dr] &CF(H_3')\ar[r]\ar[d]\ar[dr]& \ldots\\ CF(H_1)\ar[r]\ar[d]\ar[dr]& CF(H_2)\ar[d]\ar[r]\ar[dr] &CF(H_3)\ar[r]\ar[d]\ar[dr]& \ldots\\CF(H_{n(1)}')\ar[r]\ar[d]\ar[dr]& CF(H_{n(2)}')\ar[d]\ar[r]\ar[dr] &CF(H_{n(3)}')\ar[r]\ar[d]\ar[dr]& \ldots\\CF(H_{m(1)})\ar[r] &CF(H_{m(2)})\ar[r]&CF(H_{m(3)})\ar[r]&\ldots}
\end{align}
Now we apply $H(\widehat{tel}(\cdot))$ to this diagram. By Proposition \ref{c2pweakcompression} and the third bullet point of Corollary \ref{Floertheoreticcor}, the composition of the first and second (from the top) maps $$H(\widehat{tel}(\mathcal{C}(H_s')))\to H(\widehat{tel}(\mathcal{C}(H_s)))\to H(\widehat{tel}(\mathcal{C}(H^{n'}_s)))$$ is an isomorphism. The same is true for the second and third maps for the same reason. This finishes the proof.
\end{proof}

\begin{proposition}\label{c3pwell}
\begin{enumerate}
\item Let $H_s$ and $H'_s$ be two different acceleration data, then $H(SC_M(K,H_s))$ is isomorphic to $H(SC_M(K,H'_s))$ canonically. Therefore we denote the invariant by $SH_M(K)$, by a harmless abuse of notation.
\item For $K\subset K'$, there are canonical restriction maps $SH_M(K')\to SH_M(K)$, which satisfy the presheaf property.
\item Let $\phi:M\to M$ be a symplectomorphism. Then, there exists a canonical isomorphism between $SH_M(K)$ and $SH_M(\phi(K))$.
\end{enumerate}
\end{proposition}
\begin{proof}
To construct the maps in (1), we take acceleration data $H_s''$ (a roof) such that $H_i''>H_i$ and $H_i''>H_i'$ for every $i\geq 1$. Then we obtain canonical isomorphisms $H(SC_M(K,H_s))\to H(SC_M(K,H_s''))$ and $H(SC_M(K,H_s'))\to H(SC_M(K,H_s''))$. Inverting the latter we obtain an isomorphism $H(SC_M(K,H_s))\to H(SC_M(K,H'_s))$. The canonicity of this map is proved by constructing two such isomorphisms, and then taking another acceleration data that is a roof for the two roofs involved and using the discussion around Equation \ref{c3ecompare}. 

The maps in (2), are defined exactly as the maps in Equation \ref{c3ecompare} were defined. Their properties are proved exactly in the same way. 

The map in (3) is defined by relabeling all choices by the symplectomorphism $\phi$.
\end{proof}

\subsubsection{Computing $SH_M(M)$ and $SH_M(\varnothing)$}\label{c3sscomputing}

For $K=M$, take a $C^2$-small non-degenerate $H: M\to \mathbb{R}$ with no non-constant time-$1$ orbits, which is negative everywhere (see Lemma \ref{c5lflat}). We define $H_s=s^{-1}H$, for $s\geq 1$, as the acceleration data. 

Let $CM(H)$ be the Morse complex of $H$ with $\mathbb{Z}$-coefficients. On the other hand we denote by $CM(H,\Lambda_{\geq 0})$ the complex freely generated over $\Lambda_{\geq 0}$ by the critical points, but with the terms in the differential weighed by $T^{H(p_+)-H(p_-)}$.

By the arguments leading to Pardon \cite{P} Theorem 10.7.1, and that we can actually achieve transversality for the relevant moduli spaces of gradient flow lines for $H$, we see that the extra choices can be made so that the associated $1$-ray for this acceleration data looks like
\begin{align}
\ldots CM(H_n,\Lambda_{\geq 0})\to CM(H_{n+1},\Lambda_{\geq 0})\ldots ,
\end{align}
where a generator $p$ in the $n$th level is sent to $T^{\frac{-H(p)}{n(n+1)}}p$ by the continuation map, using $\frac{1}{n(n+1)}=\frac{1}{n}-\frac{1}{n+1}$. 

Recalling Lemma \ref{c2lcomlim}, note that we can use completed direct limit rather than completed telescope to compute the relative symplectic cohomology groups from this $1$-ray. It is easy to see that the direct limit of this diagram of chain complexes is $CM(H)\otimes_{\mathbb{Z}}\Lambda_{> 0}$ with maps \begin{align}
CM(H_n,\Lambda_{\geq 0})\to CM(H)\otimes_{\mathbb{Z}}\Lambda_{> 0},
\end{align}sending $p$ to $T^{\frac{-H(p)}{n}}p$. Completion does nothing to $CM(H)\otimes_{\mathbb{Z}}\Lambda_{> 0}$. Using that $\Lambda_{> 0}$ is flat over $\mathbb{Z}$, we get the result that was stated in the Introduction: \begin{align}
SH_M(M)=H(M,\mathbb{Z})\otimes_{\mathbb{Z}}\Lambda_{> 0}.
\end{align}

For $K=\varnothing$, we start with any non-degenerate $H:M\times S^1\to \mathbb{R}$, and define $H_s:=H+s$. The $1$-ray $\mathcal{C}$ for this acceleration data is of the form $C\to \ldots C\to C\to\ldots$ for some chain complex $C:=CF(H)$, where all the maps $C\to C$ are also the same. Let us give this map the name $F:C\to C$. 

We will apply the completed direct limit to $\mathcal{C}$ and observe that the resulting chain complex vanishes on the nose, which implies the result (again using Lemma \ref{c2lcomlim}).

By the energy inequality (as in Equation \ref{c3etope}), the image of $F$ is contained in the submodule $\Lambda_{\geq 1}C$ of $C$. Hence the image of $F^{\circ N}$, i.e. the $N$th iterate of $F$, is contained in $\Lambda_{\geq N}C$.

Therefore, $\lim_{\rightarrow}(\mathcal{C})\otimes_{\Lambda_{\geq 0}} \Lambda_{\geq 0}/\Lambda_{\geq r}=\lim_{\rightarrow}(\mathcal{C}\otimes_{\Lambda_{\geq 0}} \Lambda_{\geq 0}/\Lambda_{\geq r})=0$, for every $r>0$, using the standard model for computing filtered directed limits. This finishes the proof, as the inverse limit of any diagram of $0$-modules is the $0$-module.

\subsection{Multiple subsets}\label{c3smultiple}

The reader who is only interested in the existence of a Mayer-Vietoris sequence as in Theorem \ref{c1tmv} can safely skip this section.

Let $K_1,\ldots ,K_n$ be compact subsets of $M$. For every $I\subset [n]$, choose a cofinal sequence $H_k^{C_I}$ for $C_I:=\bigcap_{i\in I} K_i$, such that $H_k^{C_{I'}}\geq H_k^{C_{I}}$ whenever $I\subset I'$.  Here by $C_{\varnothing}$ we mean the union of $K_i$'s. 

For each $k$, we can find a monotone $(n+1)$-cube family of Hamiltonians extending $H_k^{C_I}$ and $H_{k+1}^{C_I}$ as in Proposition \ref{Floertheoreticprop}, which agree on the common faces. Let us call this entire $n+1$ dimensional familty of Hamiltonians $\mathcal{H}$. Such an $\mathcal{H}$ is called \textbf{acceleration data} for $K_1,\ldots ,K_n$. This extends the definition given in the previous section for $n=1$.

Using Hamiltonian Floer theory, an acceleration data gives us an $(n+1)$-ray. The ordering of the coordinates of the slices is given by the ordering of the subsets and the infinite direction is the last one as is needed in an $(n+1)$-ray. Here is a diagram for how this looks like for $n=2$:

\begin{tikzpicture}
  \matrix (m) [matrix of math nodes, row sep=1.5em,
    column sep=1em]{
    & \ldots & & CF(H_k^{K_1\cup K_2}) & & CF(H_{k+1}^{K_1\cup K_2}) & & \ldots   \\
    \ldots & & CF(H_k^{K_1})  & & CF(H_{k+1}^{K_1})  & & \ldots & \\
    &\ldots & & CF(H_k^{K_2})  & & CF(H_{k+1}^{K_2}) & & \ldots \\
    \ldots & & CF(H_k^{K_1\cap K_2})  & & CF(H_{k+1}^{K_1\cap K_2}) & &\ldots  & \\};
  \path[-stealth]
    (m-1-4) edge (m-1-6) edge (m-2-3)
            edge [densely dotted] (m-3-4) edge (m-2-5)
            edge [densely dotted] (m-4-3) 
            edge [densely dotted] (m-3-6)
            edge [densely dotted] (m-4-5)
    (m-1-6) edge [densely dotted] (m-3-6) edge (m-2-5) edge [densely dotted] (m-4-5)
    
    (m-2-3) edge [-,line width=6pt,draw=white] (m-2-5)
            edge (m-2-5) edge (m-4-3) edge (m-4-5)
    (m-3-4) edge [densely dotted] (m-3-6)
            edge [densely dotted] (m-4-3)
            edge [densely dotted] (m-4-5)
    (m-4-3) edge (m-4-5)
    (m-3-6) edge (m-4-5)
    (m-2-5) edge [-,line width=6pt,draw=white] (m-4-5)
            edge (m-4-5)
    (m-1-2) edge (m-1-4)
    (m-2-1) edge (m-2-3) 
    (m-3-2) edge [densely dotted](m-3-4) 
    (m-4-1) edge (m-4-3) 
    (m-1-6) edge (m-1-8) 
    (m-2-5) edge [-,line width=6pt,draw=white] (m-2-7) edge (m-2-7)
    (m-3-6) edge (m-3-8) 
    (m-4-5) edge (m-4-7);             
\end{tikzpicture} 

Applying $\widehat{tel\circ cone^n}$ to this $(n+1)$-ray, we construct a chain complex $SC_M(K_1,\ldots K_n,\mathcal{H})$. Note that $SC_M(K_1,\ldots K_n,\mathcal{H})$ depends on the ordering of the subsets. Now, we want to show that for any two acceleration data $\mathcal{H}$ and $\mathcal{H}'$ for $K_1,\ldots, K_n$, there exists a canonical isomorphism $$H(SC_M(K_1,\ldots K_n,\mathcal{H}'))\to H(SC_M(K_1,\ldots K_n,\mathcal{H})).$$ This goes through exactly the same steps as in the $n=1$ case.

\begin{enumerate}
\item Construct maps $H(SC_M(K_1,\ldots K_n,\mathcal{H}'))\to H(SC_M(K_1,\ldots K_n,\mathcal{H}))$ whenever $(H_i')^A\leq H_i^A$, for every $i\geq 1$ and $I\subset [n]$, where $A=\bigcap_{i\in I} K_i$, and show that they are well-defined, both using Proposition \ref{Floertheoreticprop}, and Lemma \ref{c2ltelescope}.
\item Show that the maps from the previous item are isomorphisms by the same strategy as in Proposition \ref{c3piso}. Namely, choose infinite strictly monotone sequences  $n(i)$ and $m(i)$ of positive integers such that $(H_i')^A<H_i^A<(H_{n(i)}')^A<H_{m(i)}^A$, for every $i\geq 1$ and $I\subset [n]$, where $A=\bigcap_{i\in I} K_i$. Then, using Proposition \ref{Floertheoreticprop}, we construct 3 composable maps of $(n+1)$-rays and use Proposition \ref{c2pweakcompression} and the second bullet point of Corollary \ref{Floertheoreticcor} to get the desired statement.
\item To construct the comparison maps for any $\mathcal{H}'$ and $\mathcal{H}$, find an $\mathcal{H}''$, which is larger than both $\mathcal{H}'$ and $\mathcal{H}$. Using the previous item to invert the map $H(SC_M(K_1,\ldots K_n,\mathcal{H}'))\to H(SC_M(K_1,\ldots K_n,\mathcal{H}''))$, we obtain the desired map by composition. To show that they are well defined, again we use a ``roof of the roofs argument".
\end{enumerate}


Hence, we define $$SH_M(K_1,\ldots ,K_n):=H(SC_M(K_1,\ldots K_n,\mathcal{H})),$$ for some acceleration data $\mathcal{H}$. In particular, it makes sense to talk about $SH_M(K_1,\ldots ,K_n)$ being trivial or not.

\section{Mayer Vietoris property}\label{c5}
Let us start with a recap. Let $X,Y$ be two compact subsets of a closed smooth manifold $M$. 
We can choose acceleration data $H_s^A$, for $A=X\cap Y,X,Y, X\cup Y$, so that $H_n^A\geq H_n^B$, whenever $A\subset B$. We can then construct a $3$-ray with $\mathcal{C}(H_s^A)$ at the four infinite edges (using Proposition \ref{Floertheoreticprop}):

\begin{tikzpicture}
  \matrix (m) [matrix of math nodes, row sep=1.5em,
    column sep=1em]{
    & \ldots & & CH(H_n^{X\cup Y}) & & CH(H_{n+1}^{X\cup Y}) & & \ldots   \\
    \ldots & & CH(H_n^{X})  & & CH(H_{n+1}^{X})  & & \ldots & \\
    &\ldots & & CH(H_n^{Y})  & & CH(H_{n+1}^{Y}) & & \ldots \\
    \ldots & & CH(H_n^{X\cap Y})  & & CH(H_{n+1}^{X\cap Y}) & &\ldots  & \\};
  \path[-stealth]
    (m-1-4) edge (m-1-6) edge (m-2-3)
            edge [densely dotted] (m-3-4) edge (m-2-5)
            edge [densely dotted] (m-4-3) 
            edge [densely dotted] (m-3-6)
            edge [densely dotted] (m-4-5)
    (m-1-6) edge [densely dotted] (m-3-6) edge (m-2-5) edge [densely dotted] (m-4-5)
    
    (m-2-3) edge [-,line width=6pt,draw=white] (m-2-5)
            edge (m-2-5) edge (m-4-3) edge (m-4-5)
    (m-3-4) edge [densely dotted] (m-3-6)
            edge [densely dotted] (m-4-3)
            edge [densely dotted] (m-4-5)
    (m-4-3) edge (m-4-5)
    (m-3-6) edge (m-4-5)
    (m-2-5) edge [-,line width=6pt,draw=white] (m-4-5)
            edge (m-4-5)
    (m-1-2) edge (m-1-4)
    (m-2-1) edge (m-2-3) 
    (m-3-2) edge [densely dotted](m-3-4) 
    (m-4-1) edge (m-4-3) 
    (m-1-6) edge (m-1-8) 
    (m-2-5) edge [-,line width=6pt,draw=white] (m-2-7) edge (m-2-7)
    (m-3-6) edge (m-3-8) 
    (m-4-5) edge (m-4-7);             
\end{tikzpicture}

The $2$-cube slices of this $3$-ray look like:
\begin{align}
\xymatrix{ 
CH(H_n^{X\cup Y})\ar[r]\ar[d]\ar[dr]& CH(H_n^{X})\ar[d]\\ CH(H_n^{Y})\ar[r] &CH(H_n^{X\cap Y})}.
\end{align}

\begin{definition}\label{c5dcompatible3ray} Let us say that such a $3$-ray is \textbf{compatible} with $H_s^A$.
\end{definition}

We are going to show that, under certain strong geometric assumptions, we can set-up this $3$-ray in such a way that all of these slices are acyclic $2$-cubes. This implies the desired Mayer-Vietoris sequence by Lemma \ref{c2lacycliccube} and Lemma \ref{c2lmv}.

\begin{remark}\label{c5rmultiple}
In terms of Section \ref{c3smultiple}, setting up such a $3$-ray implies that $SH_M(X,Y)=0$. Yet, the fact that $SH_M(X,Y)$ is well-defined does not play a role in the existence a Mayer-Vietoris sequence (it does play a role in its canonicity, but we do not discuss that here). The results of Section \ref{c3smultiple} will be used only in Section \ref{c5sinvolutive}.
\end{remark}

\subsection{Zero energy solutions}\label{c5szero}

In this section we analyze the zero energy solutions of Floer equations, whose importance to us stems from Corollary \ref{c2ccomplete}, part (1). More precisely, our goal is to prove the following proposition. 

\begin{proposition}\label{c5pmax}
Let $f$ and $g$ be two non-degenerate Hamiltonians $M\times S^1\to \mathbb{R}$. We define $U=\{f<g\}\subset M\times S^1$ and $V=\{f>g\}\subset M\times S^1$. 

Consider the following conditions:

\begin{enumerate}
\item $\overline{U}$ and $\overline{V}$ are disjoint.
\item $M\times S^1-U$ and $M\times S^1-V$ are both closures of open sets.
\item No one-periodic orbit of $X_f$, $X_g$, $X_{min(f,g)}$ or $X_{max(f,g)}$ has a graph\footnote{the graph of $\gamma: S^1\to M$ is the image of the map $\gamma\times id:S^1\to M\times S^1.$} that intersects both $U$ and $V$ (see Figure \ref{c5fsep}).
\end{enumerate}

If we assume (1), then $max(f,g)$ and $min(f,g)$ are smooth functions.

If we assume (1), (2), and (3), then $max(f,g)$ and $min(f,g)$ are non-degenerate, and the $2$-cube
\begin{align}
\xymatrix{ 
CF(min(f,g))\ar[r]\ar[d]\ar[dr]& CF(f)\ar[d]\\ CF(g)\ar[r] &CF(max(f,g))}.
\end{align}
is acyclic, for any choice of monotone $2$-cube family of Hamiltonians (extending $f,g,max(f,g)$ and $min(f,g)$ at the corners $(1,0),(0,1),(1,1)$ and $(0,0)$, respectively) and extra data necessary to define the maps (as in Remark \ref{remarkextrachoices}). 
\end{proposition}

\begin{figure}[!h]
\centering
\includegraphics[scale=0.4]{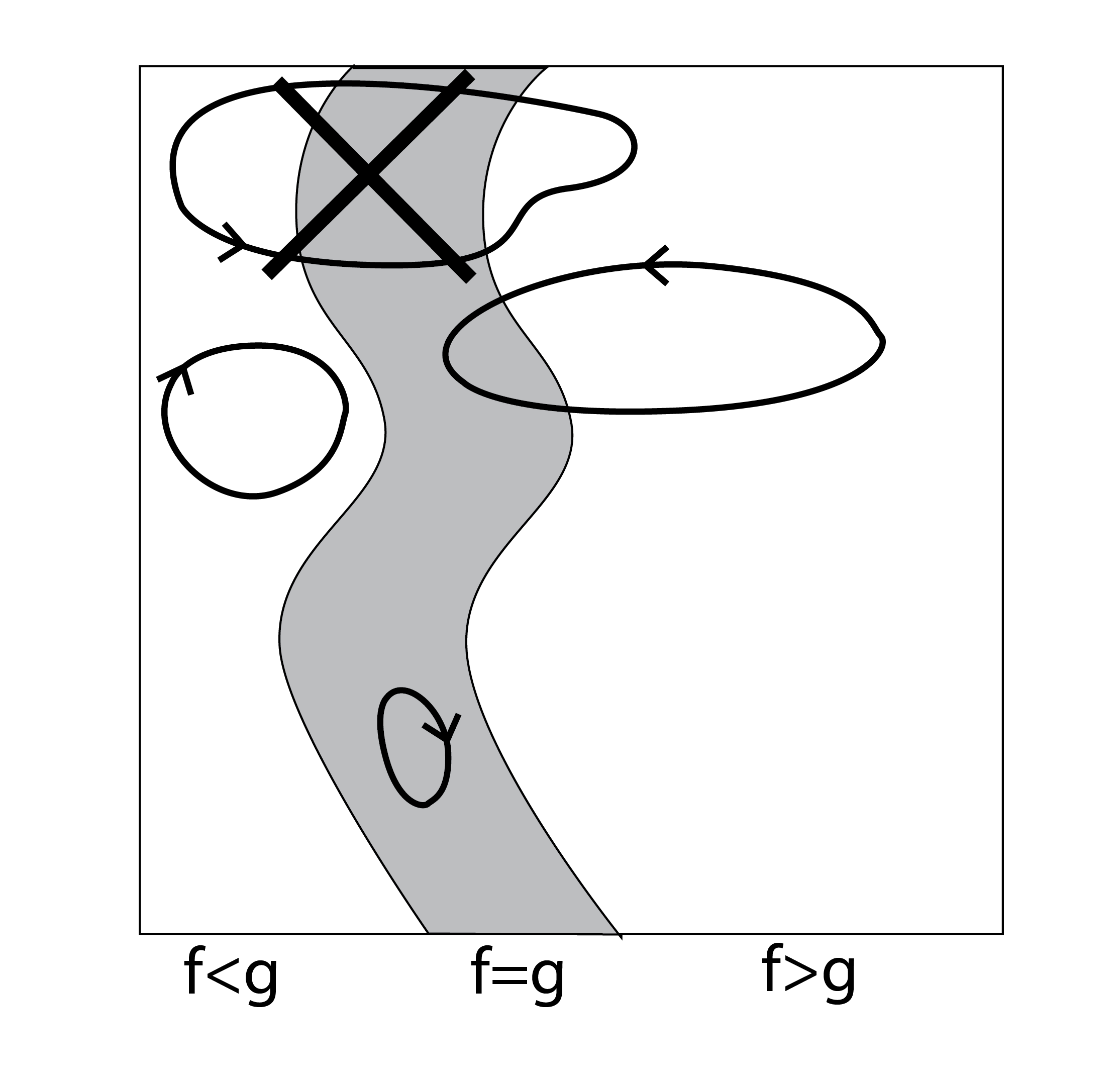}
\caption{A depiction of the orbit condition in Proposition \ref{c5pmax}. The crossed orbit is not allowed. The other $3$ orbits are allowed.}
\label{c5fsep}
\end{figure}

\begin{remark}\label{c5rconstant}
In the applications below $U$ and $V$ will be of the form $\tilde{U}\times S^1$ and $\tilde{V}\times S^1$.This case is what is realistically depicted in Figure \ref{c5fsep}. Note also that in this case, the condition of not intersecting both $U$ and $V$ is empty for constant orbits. 
\end{remark}

We need a preliminary lemma.

\begin{lemma}\label{c5lzero}
Let $h_0\leq h_1$ be non-degenerate Hamiltonians $M\times S^1\to \mathbb{R}$ with a monotone homotopy $H:\mathbb{R}_s\times M\times S^1\to \mathbb{R}$, \begin{align}
H(s,\cdot,\cdot) =
\begin{cases}
h_0(\cdot,\cdot) , & s<<0  \\
h_1(\cdot,\cdot) , & s>>0,
\end{cases}
\end{align}between them. 

Let $\gamma_0$ and $\gamma_1$ be maps $S^1\to M$ which are one-periodic orbits of $h_0$ and $h_1$, respectively. We make the necessary choices (arbitrarily as in Remark \ref{remarkextrachoices}) and define the continuation map $con:CF(h_0)\to CF(h_1)$. We consider the matrix coefficient $\alpha:= <con(\gamma_0),\gamma_1>\in \Lambda_{\geq 0}$.

\begin{itemize}
\item If $val(\alpha)=0$, then $\gamma_0=\gamma_1$, and $H(s,\gamma_0(t),t): S^1\to\mathbb{R}$ is independent of $s$.
\item Assume that for any point $(x,t)\in M\times S^1$ in the graph of $\gamma_0$ and $s\in \mathbb{R}$, the function $H(s,\cdot,\cdot)-h_0:M\times S^1\to \mathbb{R}$ vanishes to all orders at $(x,t)$\footnote{this means that at that point, in some arbitrary coordinate system, all iterated partial derivatives are zero.} . Then $\gamma_0$ is a non-degenerate one-periodic orbit for $h_1$ as well. Moreover, if $\gamma_1$ in question is equal to $\gamma_0$ as a map $S^1\to M$, then $val(\alpha)=0$.
\end{itemize}
\end{lemma}

\begin{proof}
The first statement immediately follows from the energy identity (Equation (\ref{c3etope}) in Section \ref{c3shamiltonian}). 

For the second statement, note that $u(s,t)=\gamma_0(t)$ satisfies the Floer equation. This solution is regular, because of the non-degeneracy of $\gamma_0$ as a one-periodic of $h_0$, and the vanishing up to all orders condition. By the energy identity, it is the only solution with zero topological energy. 

Now, we need to argue that the moduli space of stable Floer trajectories in the homotopy class of the constant solution also consists only of this solution (please see the Definition 10.2.2 for $n=1$, and Definition 10.2.3 from \cite{P}). Since topological energy is the same (and equal to $0$) for every member of this moduli space, the trajectory cannot contain any non-constant sphere bubbles. Because of stability, this means that it also cannot contain any constant sphere bubbles. Similarly, by the topological energy argument the trajectory cannot contain any $s$-dependent Floer solutions for $s$-independent data (equal to $h_0$ or $h_1$), and by stability, it cannot contain any $s$-independent ones. This finishes the argument.

It is clear that that no other homotopy class can contain Floer trajectories with $0$ topological energy. Hence, we have the desired statement using the second bullet point of Lemma \ref{c3lVFC}.
\end{proof}

In the notation of Lemma \ref{c5lzero}, if $\gamma_0$ and $\gamma_1$ are as in the second bullet point, we call $\gamma_0$ and $\gamma_1$ \textbf{common} orbits of $h_0$ and $h_1$. Note that $\gamma_0$ and $\gamma_1$ satisfying the conclusion of the first bullet point of Lemma \ref{c5lzero} does not imply that they are common orbits of $h_0$ and $h_1$.

\begin{proof}[Proof of Proposition \ref{c5pmax}]
Let us first assume (1), and prove that $min(f,g)$ is a smooth function. The proof for $max(f,g)$ is identical. Let us define $W:=M\times S^1 -(\overline{U}\cup \overline{V})$, which is an open set. We will check that $min(f,g)$ is smooth at all $y\in M\times S^1$. \begin{itemize}
\item If $y\in W$, then $y$ has an open neighborhood $N$ in $W$. On $N$, $min(f,g)=f=g$, by definition.
\item If $y\in \overline{U},$ then $y\notin \overline{V}$. Therefore, $y$ has an open neighborhood $N$ in $M\times S^1-\overline{V}$. On $N$, $min(f,g)=f$, by definition.
\item If $y\in \overline{V},$ the proof is the same as the previous item with the roles of $f$ and $g$ reversed.
\end{itemize}

Since $f$ and $g$ are smooth, this finishes the proof.

Let us now assume all of (1), (2), and (3). Note that if $h_0=h_1$ on an open set $S$, and $H$ is a monotone homotopy from $h_0$ to $h_1$ (as in Lemma \ref{c5lzero}), then for every $s\in\mathbb{R}$, $H(s,\cdot,\cdot)=h_0$ on $\overline{S}$ with all derivatives.  

If $(h_0,h_1)$ equals one of the two ordered pairs $(min(f,g), f)$ or $(g,max(f,g))$, then the previous paragraph and condition (2) implies that, for every $s\in\mathbb{R}$, $H(s,\cdot,\cdot)=h_0$ on $M\times S^1-V$ with all derivatives. Also note that $h_0<h_1$ on $V$. Similarly, if $(h_0,h_1)$ equals one of the two pairs $(min(f,g), g)$ or $(f,max(f,g))$, we have that $H(s,\cdot,\cdot)=h_0$ on $M\times S^1-U$ with all derivatives, and $h_0<h_1$ on $U$. We denote this paragraph by $\star$.

It follows from the previous paragraph (i.e. $\star$) and condition (3) that a $1$-periodic orbit $\gamma_0$ of $min(f,g)$ is common with an orbit of $f$ or $g$, or both. This is because the graph of $\gamma_0$ has to be contained in $M\times S^1-U$ or $M\times S^1-V$, or both. There is a similar statement for $max(f,g)$. The second bullet point of Lemma \ref{c5lzero} proves the non-degeneracy statement.

Let us call the $2$-cube in question $\mathcal{C}$. We want to show that the chain complex $cone^{2}(\mathcal{C})$ is acyclic. Now set the Novikov parameter $T=0$ (see Remark \ref{remarkT} for what this means) to obtain the chain complex $C:=cone^{2}(\mathcal{C})\otimes \Lambda_{\geq 0}/\Lambda_{> 0}$ over the rational numbers. 

First, notice that the valuations of the matrix entries of the homotopy map $$CF(min(f,g))\to CF(max(f,g))[1]$$ are all positive. The easiest way to see this is to use the mod $2$ grading. Assume the contrary, then by the first item of Lemma \ref{c5lzero}, we get that any two orbits corresponding to a $0$ valuation entry have to be the same map $S^1\to M$ whose graph is disjoint from $U\cup V$. By $\star$, we get that the two orbits are both common with an orbit of $f$ (and $g$). Therefore, the two generators have the same mod $2$ grading, as they have exactly the same linearization of the return map. This is a contradiction because the homotopy map has degree $1$. 

Now, let $CF(h_0)\otimes \Lambda_{\geq 0}/\Lambda_{> 0}\to CF(h_1)\otimes \Lambda_{\geq 0}/\Lambda_{> 0}$  be one of the four chain maps corresponding to the edges of $\mathcal{C}$ after setting $T=0$. Let $\gamma_0$ be a $1$-periodic orbit of $h_0$ with the corresponding basis element $b_0$ of $CF(h_0)\otimes \Lambda_{\geq 0}/\Lambda_{> 0}$. Let us similarly take $\gamma_1$ and $b_1$ for $h_1$. 

We claim that that the matrix entry corresponding to $b_0$ and $b_1$ is non-zero if and only if $\gamma_0$ and $\gamma_1$ are common orbits of $h_0$ and $h_1$ (we call this statement $\star\star$). The if part is the second item of Lemma \ref{c5lzero}. Conversely, if the matrix entry is non-zero, then we apply the first bullet point of Lemma \ref{c5lzero}, which in this case actually implies the stronger statement that $\gamma_0$ and $\gamma_1$ are common orbits of $h_0$ and $h_1$ using $\star$. For example, let $h_0=min(f,g)$ and $h_1=f$. Then $\gamma_0$ and $\gamma_1$ needs to have the same underlying $S^1\to M$, and the graph of this map cannot intersect $V$. This means that it is contained in $M-V$, which proves that they are common. 

Let us consider an orbit $\gamma$ of $min(f,g)$. Then $\gamma$ is common with an orbit of $f$ if and only if the graph of $\gamma$ is contained in $M\times S^1-V$ using $\star$. Similarly, $\gamma$ is common with an orbit of $g$ if and only if the graph of $\gamma$ is contained in $M\times S^1-U$. There are similar statements for all four Hamiltonians.

Also note that an orbit $\gamma$ of $min(f,g)$ is common with an orbit $\gamma_0$ of $f$ and with an orbit $\gamma_1$ of $g$ if and only if $\gamma$ is contained in $M\times S^1-U$ and $M\times S^1-V$. This then implies that there is an orbit $\gamma'$ of $max(f,g)$, which is common with both $\gamma_0$ of $f$ and $\gamma_1$ of $g$. There are similar statements for all four Hamiltonians.

Using the analysis of the last two paragraphs along with $\star\star$, noting that $C$ is a chain complex, and possibly rescaling generators, we obtain that $C$ is isomorphic to a direct sum of the following chain complexes  \begin{itemize}
\item $\mathbb{Q}<x,y>; dx=y, dy=0$
\item $\mathbb{Q}<x_1,x_2,y_1,y_2>; dx_1=y_1+y_2, dy_1=x_2, dy_2=-x_2, dx_2=0 $
\end{itemize}

Hence, we proved that $C$ is acyclic. This finishes the proof by Corollary \ref{c2ccomplete} Part (1).
\end{proof}

\begin{remark} Let us note that, in the notation of Proposition \ref{c5pmax}, the condition (3) only holds under very special circumstances that will be discussed in later sections. The condition (1) is also a serious one in the sense that it is not a generic condition, and is only satisfied for specially constructed functions. On the contrary, (2) is of a technical nature, and it is satisfied generically. For example, $0$ being a regular value of $f-g$ implies (2), but this is far from being a necessary condition.
\end{remark}

Let us note the following situations in which the conditions (1) and (2) of Proposition \ref{c5pmax} are satisfied that will be relevant in later sections.

\begin{lemma}\label{lemma12}
Let $X$ be a smooth manifold and $f,g$ be two smooth functions on $X$. Then, the conditions (1) and (2) of Proposition \ref{c5pmax} (with $M\times S^1$ replaced with $X$) are satisfied in the following cases.
\begin{itemize}
\item The locus where $f=g$ is a topological submanifold with boundary of codimension $0$.
\item There is a smooth and proper map $\pi: X\to Y$ such that $f$ and $g$ factors through $\pi$, i.e. there exists smooth functions $\tilde{f}, \tilde{g}$ on $Y$ such that $f=\tilde{f}\circ\pi$ and $g=\tilde{g}\circ\pi$, and $\tilde{f}, \tilde{g}$ satisfy the conditions (1) and (2) of Proposition \ref{c5pmax}.
\end{itemize}
\end{lemma}

\begin{proof}
Let $U$ and $V$ be as in Proposition \ref{c5pmax}. Let us denote the locus where $f=g$ by $D$. In the first case, by definition, every point of $\partial D$ admits an open neighborhood $N$ such that $N\cap \partial D$ is connected, and its complement in $N$ has two non-empty connected components, one in which $f=g$, and one in which either $f<g$ or $f>g$. The condition (1) automatically follows. For (2) note that it follows that $U\cup D$ and $V\cup D$ are topological submanifolds of codimension $0$ as well.

For the second item note that under these assumptions we have that for every subset $A\subset Y$, \begin{align}\label{eqclosure}\overline{\pi^{-1}(A)}=\pi^{-1}(\overline{A}).
\end{align} See for example the Corollary in \cite{Palais}. After this, the result follows from elementary set theoretic considerations and that $\pi$ is continuous.
\end{proof}

\begin{remark}
We do not need any of the smoothness requirements in Lemma \ref{lemma12}. For example, we could assume that $X$ and $Y$ are topological manifolds and change all the smoothness requirements to continuity. In the application of the second bullet point for this paper, $X$ is in fact compact, in which case the proof of Equation \ref{eqclosure} is simpler. We wrote it more generally with future applications in mind.
\end{remark}

\subsection{Boundary accelerators}\label{c5sboundary}

In this subsection we explain how to choose an acceleration data so that the interesting Hamiltonian dynamics concentrates near hypersurfaces that tightly envelop the compact subset in question.

\begin{definition}
Let $K$ be a subset of $M$, we say that a sequence of compact domains $D_K^1,D_K^2,\ldots $ \textbf{approximate} $K$, if: \begin{itemize}
\item $\bigcap D_K^i=K$.
\item $D_K^{i+1}\subset int(D_K^i)$ , for all $i\geq 1$.
\end{itemize}
\end{definition}

Note that every compact subset can be approximated by compact domains.

\begin{definition}\label{c5dboundaryacc}
A \textbf{boundary accelerator} for $K$ consists of three pieces of data:
\begin{enumerate}
\item A strictly increasing sequence of positive numbers $\Delta_i$ which converge to infinity as $i\to \infty$.
\item A sequence of triples of compact domains $\{(fill(\partial N_K^{i-}),N_K^i, fill(\partial N_K^{i+}))\}_{i\in\mathbb{Z}_{>0}}$ such that\begin{itemize}
\item $fill(\partial N_K^{i-})\cup N_K^i\cup fill(\partial N_K^{i+})=M$.
\item The interiors of $fill(\partial N_K^{i-}),N_K^i, fill(\partial N_K^{i+})$ are pairwise disjoint.
\item $\partial N_K^i=\partial N_K^{i-}\sqcup \partial N_K^{i
+}$, $\partial fill(\partial N_K^{i-})=\partial N_K^{i-}$, $\partial fill(\partial N_K^{i+})=\partial N_K^{i+}$.
\item $fill(\partial N_K^{i-})$ approximate $K$.
\item $N_K^i$ is contained in the interior of $fill(\partial N_K^{(i+1)+})$.
\end{itemize}
We call $N_K^i$ the \textbf{mixing regions}, $fill(\partial N_K^{i-})$ the \textbf{inner fillers}, and $fill(\partial N_K^{i+})$ the \textbf{outer fillers}.
\item Smooth functions $f_i:N_K^i\to [0,\Delta_i]$ such that \begin{itemize}
\item $f_i$ has no critical points along $\partial N_K^i$.
\item $f_i^{-1}(0)= \partial N_K^{i-}$ and $f_i^{-1}(\Delta_i)= \partial N_K^{i+}$.
\end{itemize}
We call these the \textbf{excitation functions}.
\end{enumerate}
\end{definition} See Figure \ref{c5ffillers} for a depiction. We will generally drop the fillers from notation, but they are always there.

\begin{figure}
\centering
\includegraphics[scale=0.4]{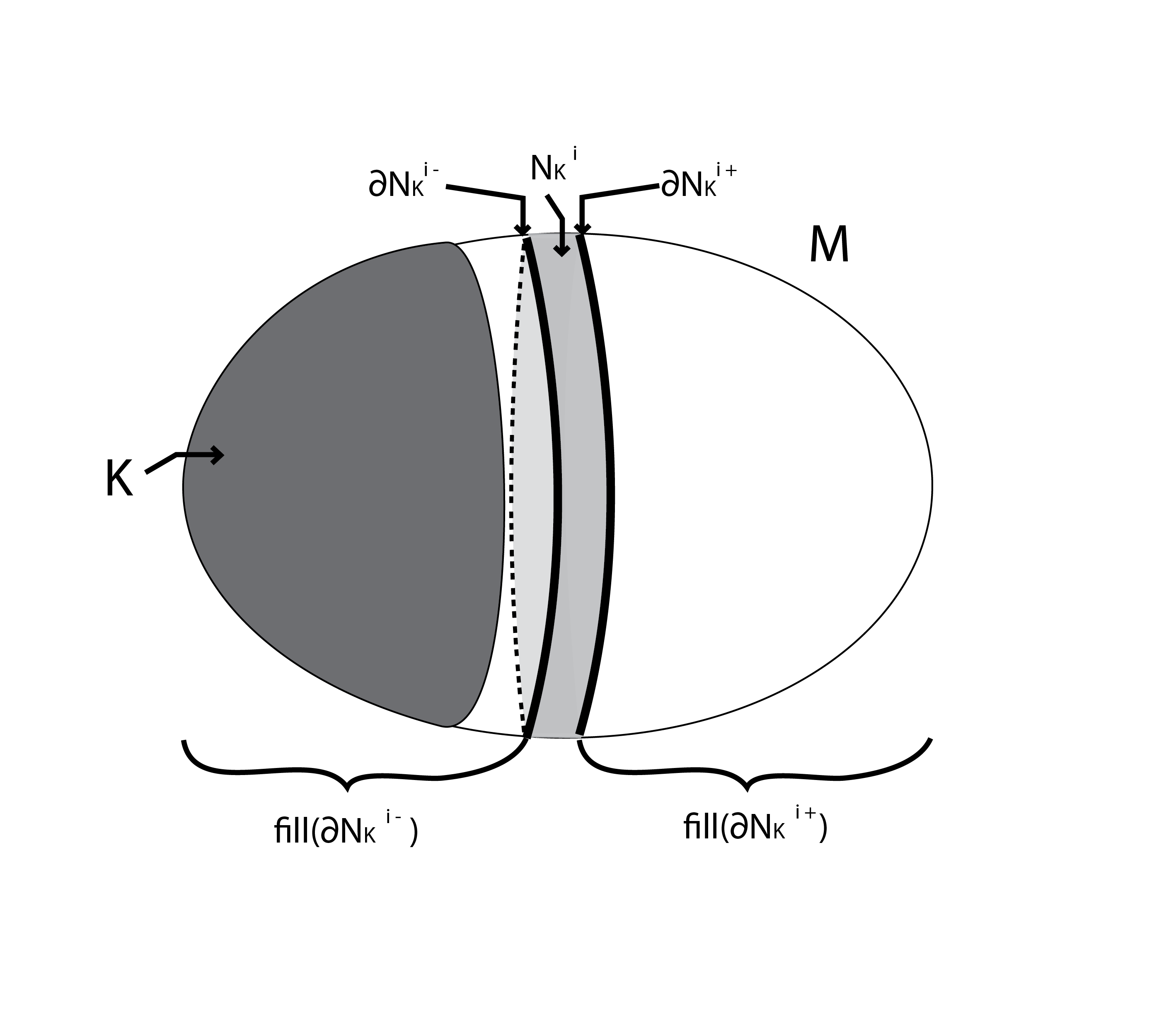}
\caption{One mixing region in a boundary accelerator, and the relevant notation. }
\label{c5ffillers}
\end{figure}

Now we explain how we get a valid acceleration data starting from a boundary accelerator. An extra property we want is to restrict the points that a non-constant periodic orbit can pass through to the interiors of the mixing regions. The following lemma is our main tool in that respect.

\begin{lemma}\label{c5lflat}
Let $H:X\to \mathbb{R}$ be a smooth function, where $X$ is a compact symplectic manifold with boundary and $H$ is constant along the boundary. For small enough $\epsilon>0$, all time-$1$ orbits of $X_{{\epsilon}H}$ are constant. Moreover, if $H$ is Morse, for a possibly smaller $\epsilon>0$, all of those orbits are non-degenerate as $1$-periodic orbits.
\end{lemma}

\begin{proof}
The first statement follows from the more general theorem of Yorke \cite{Y}. 

For the second statement, let $p$ be one of the finitely many critical points of $H$. Then, $X_H$ has a non-degenerate zero at $p$, and the linearization there is a well-defined map $SHess_p(f): T_pX\to T_pX$. Note that we have $$\omega(SHess_p\cdot,\cdot)=Hess(\cdot,\cdot),$$ and in particular $SHess_p$ is a linear isomorphism. It suffices to show that $exp(\epsilon SHess_p)$ does not have $1$ as an eigenvalue for sufficiently small $\epsilon>0$. This is in fact true for any linear isomorphism, which can be easily proved for diagonalizable (over $\mathbb{C}$) ones, and the general case follows from an approximation by diagonalizable matrices argument.
\end{proof}

Moreover, we will need to perturb the excitation functions to have non-degenerate orbits, but we will have to perturb in a very controlled fashion. We start with a preparatory lemma.

\begin{lemma}
Let $F:M\times [0,T]\to \mathbb{R}$ be a Hamiltonian, where $T$ is a positive real number, and $\gamma :[0,T]\to M$ a flow line of the time dependent vector field $X_{F_t}$. Then, for any $t_0\in [0,T]$, and neighborhood $V$ of $\gamma(t_0)$; we have that for every $v\in T_{\gamma(T)}M$, there is a smooth one parameter family of functions $f_s:M\times [0,T]\to \mathbb{R}$, $s\in [0,\tau),$ for some $\tau >0$, such that \begin{itemize}
\item $f_0=0$
\item for some $\epsilon >0$, $f_s(x,t)=0$, for $t<\epsilon$ and $t>T-\epsilon$ and all $s$
\item $f_s\geq 0$, for all $s$
\item $supp(f_s)\subset V$, for all $s$
\item The tangent vector to the curve $s\mapsto \phi^T_{F+f_s}(\gamma(0))$ at $s=0$ is equal to $v$.
\end{itemize}
\end{lemma}
\begin{proof}
We will show that we can reduce to the special case $F$ equals the $0$ function and $t_0=0$ in two steps:\begin{itemize}
\item We first reduce to the case $\gamma([0,T])\subset V$, and $t_0=0$. We can always find a small $\theta >0$, such that $\gamma([t_0, t_0+\theta])\subset V$, and define a new Hamiltonian $\tilde{F}:=F(\cdot , \cdot +t_0):M\times [0,\theta]\to \mathbb{R}$. Let $\psi: M\to M$ be the diffeomorphism obtained by flowing with the Hamiltonian flow of $F$ from time $t_0+\theta$ to $T$, in particular $\psi(\gamma(t_0+\theta))=\gamma(T)$. Now assuming that we solved the problem for $\tilde{F}, t_0=0, V$ and $(\psi^{-1})_*v$, as in the statement, to find $\tilde{f_s}$, we simply set 
   $$f_s(x,t) = 
\begin{cases}
\tilde{f_s}(x,t+t_0), & t_0\leq t\leq t_0+\theta  \\
0 , & \text{otherwise}
\end{cases}
$$
\item We now further reduce to the case $F=0$. Let $\tilde{V}$ be a neighborhood of $\gamma(0)$ which stays inside $V$ for the entire flow of $F$. Assume that we have solved the problem for $\tilde{F}=0, \tilde{V}$ and $((\psi_F^T)^{-1})_*v$ to find $\tilde{f_s}$. 

Now, we let $$f_s(x,t)=\tilde{f_s}\circ (\phi^t_F)^{-1}(x,t),$$for every $s\in [0,\tau)$. This does the job because $F+f_s:M\times [0,T]\to \mathbb{R}$, generates the Hamiltonian flow $\phi^t_F\circ \phi^t_{f_s}$.
\end{itemize}

By finding a Darboux chart that lies inside $V$, we can further reduce to $M=(\mathbb{R}^{2n},\omega_{st})$, $\gamma$ is the origin, and $V$ is an open ball around the origin. Our final trivial reduction is to $T=1$ for notational simplicity.

Now, consider the affine function $Jv\cdot x+c$ on $\mathbb{R}^{2n}$, where $J$ is the standard complex structure on $\mathbb{R}^{2n}$, $\cdot$ is the dot product, and $c$ sufficiently large so that $Jv\cdot x+c$ is positive on $V$. By reparametrizing the time domain we can make it supported away from $0$ and $1$, while keeping it positive, and not changing the time-$1$ map. Call the resulting Hamiltionian $f$ and define $f_s=sf$. For any positive cutoff function $\beta$ supported on $V$ and equal to $1$ near the origin, $\beta f_s$ does the job.
\end{proof}

\begin{lemma}\label{c5lnond}
Let $H:M\times S^1\to \mathbb{R}$ be a Hamiltonian, $n$ be a positive integer, $\delta$ be a positive real number, $\tilde{g}$ be a Riemannian metric on $M$, and $U\subset M$ be an open subset. Then, there exists an $H':M\times S^1\to \mathbb{R}$ such that \begin{itemize}
\item $supp(H-H')\subset U\times S^1$
\item $\|H'-H\|_{C^n(M\times S^1)}<\delta$, here we used the norm $
\|f\|_{C^n(M\times S^1)} =\sum_{l=0}^n  \sup_{M\times S^1}|\nabla^{l}f(x)|_g,$ where $\nabla$ is the Levi-Civita covariant derivative relative to $g=\tilde{g}+dt^2$.
\item $H\geq H'$.
\item All one-periodic orbits of $X_{H'}$ which intersect $U$ are non-degenerate.
\end{itemize}
\end{lemma}
\begin{proof} Let us consider the space $F$ all non-negative smooth functions $M\times S^1\to \mathbb{R}$ that are supported in $U\times S^1$ with $C^n$ norm at most $\delta$. This is a convex subset of a Banach space.

Consider the open subset $\mathcal{V}$ of $F\times M$ given by $(h,x)$, where the Hamiltonian flow $\phi_{h+H}^t$ of $h+H$, starting at $x$, intersects $U$. We have the map $\mathcal{V}\to M\times M$ given by $(h,x)\mapsto (x,\phi_{h+H}^1(x))$. 

The Sard-Smale theory of transversality extends to this setting \cite{ACC}. The previous lemma shows that this map is transverse to the diagonal in $M\times M$, and finishes the proof.
\end{proof}

Let us also note the following elementary lemma about the abundance of Morse functions.

\begin{lemma}\label{c5lmorse}
Let $f:M\to \mathbb{R}$ be a smooth function, $n$ be a positive integer, $\delta$ be a positive real number, and $U\subset M$ be an open subset. Then, there exists a smooth function $f':M\to \mathbb{R}$ with the following properties:

\begin{itemize}
\item $supp(f-f')\subset U$.
\item $\|f-f'\|_{C^n(M)}<\delta$ everywhere on $M$.
\item $f'\geq f$.
\item All critical points of $f'$ inside $U$ are non-degenerate.
\end{itemize}
\end{lemma}

\begin{remark}
In the third bullet point of both Lemma \ref{c5lmorse} and Lemma \ref{c5lnond} we could change the direction of the inequality sign.
\end{remark}

\begin{proposition}\label{c5pextension}
Let $K$ be a compact subset of a closed symplectic manifold $M$, then we can find functions $h_i:M\to\mathbb{R}$, $i\in\mathbb{Z}_{> 0}$ such that\begin{itemize}

\item There exists mixing regions $N_K^i$ (with fillers) and a sequence of numbers $\Delta_i$ so that $\{(h_i\mid_{N_K^i},N_K^i,\Delta_i)\}$ is a boundary accelerator.
\item The critical points of $h_i$ inside the fillers are non-degenerate as time-1 orbits of $X_{h_i}$. In particular, they are non-degenerate critical points.
\item All non-constant one-periodic orbits of $X_{h_i}$ are contained in the interior of $N_K^i$.
\item There exists a sequence of positive numbers $\delta_i\to 0$ such that $-\delta_i<h_i\mid_{fill(\partial N_K^{i-})}\leq 0$, with equality only on the boundary, and for $x\in fill(\partial N_K^{i-})$, $h_{i-1}(x)\leq h_{i}(x)$.
\item $\Delta_i\leq h_i\mid_{fill(\partial N_K^{i+})}<\Delta_{i+1}$, with equality only on the boundary.
\end{itemize}
\end{proposition}

\begin{proof}
Using compactness, it is elementary to find approximating domains for $K$. Then, using tubular neighborhood theorem, we construct boundary accelerators $\{(h_i\mid_{N_K^i},N_K^i,\Delta_i)\}$. Lastly, we extend the excitation functions to the fillers step by step. 
\begin{enumerate}
\item We extend the excitation functions to the fillers so that the extension is negative in the interior of the inner filler, and it is bigger than $\Delta_i$ in the interior of the outer filler.
\item By making perturbations supported inside the interior of the fillers we can make the functions Morse on the fillers, i.e. Lemma \ref{c5lmorse}.
\item Momentarily denote the function restricted to a small neighborhood of the inner filler by $f$. Let $\tilde{w}:[0,\Delta_i]\to [\epsilon,1]$ be a non-decreasing function which is equal to $\epsilon$ in a neighborhood of $0$, and to $1$ in a neighborhood of $\Delta_i$. We then extend $\tilde{w}\circ h_i$ to a function $w$ on $M$ by constants. If we multiply the function we had constructed in (2) by $w$, it still satisfies all the previous properties, but now its restriction to a (possibly smaller neighborhood) of the inner filler is $\epsilon f$. Now, recall Lemma \ref{c5lflat}. By choosing $\epsilon$ small enough we can make sure that there are no non-constant orbits contained in a neighborhood of the inner filler. We do the same for the outer filler, but this time we have to think of $\Delta_i$ as the zero level, and hence the rescaling results in $f$ becoming $\Delta_i+\epsilon(f-\Delta_i)$. Finally, notice that choosing $\epsilon$ small enough also achieves the extra non-degeneracy condition on the Morse critical points inside the fillers, as well as the last two conditions from the statement of the proposition.
\end{enumerate}
\end{proof}

\begin{proposition}\label{c5pfinal}
Let $i\geq 1$ and $h_i$ be as in Proposition \ref{c5pextension}. We also fix $n>0$ an integer, and $\tau>0$ a real number. 

We can find $H_i:M\times S^1\to\mathbb{R}$ such that
\begin{itemize}
\item $H_i=h_i$ on the fillers.
\item $H_i(int(N_K^i))\subset (0,\Delta_i)$.
\item $\|H_i-h_i\|_{C^n}<\tau$.
\item All one-periodic orbits of $X_{H_i}$ are non-degenerate.
\end{itemize}

\end{proposition}
\begin{proof}
We apply Lemma \ref{c5lnond} with $U$ being an open subset of interior of $N_K^i$ that contains all the $1$-periodic orbits of $h_i$, and whose closure does not intersect the boundary of $N_K^i$.
\end{proof}

In particular, a sequence of $H_i$'s, for $i\geq 1$, as constructed in Proposition \ref{c5pfinal} form a non-degenerate cofinal sequence for $K$. See Figure \ref{c5fsummary} for a summary of this procedure that constructs a cofinal sequence from boundary accelerators. Note that we can easily turn this into an acceleration data using a partitions of unity of $[0,1]$ as in the proof of Proposition \ref{Floertheoreticprop}.

\begin{center}
\begin{figure}
\includegraphics[scale=0.2]{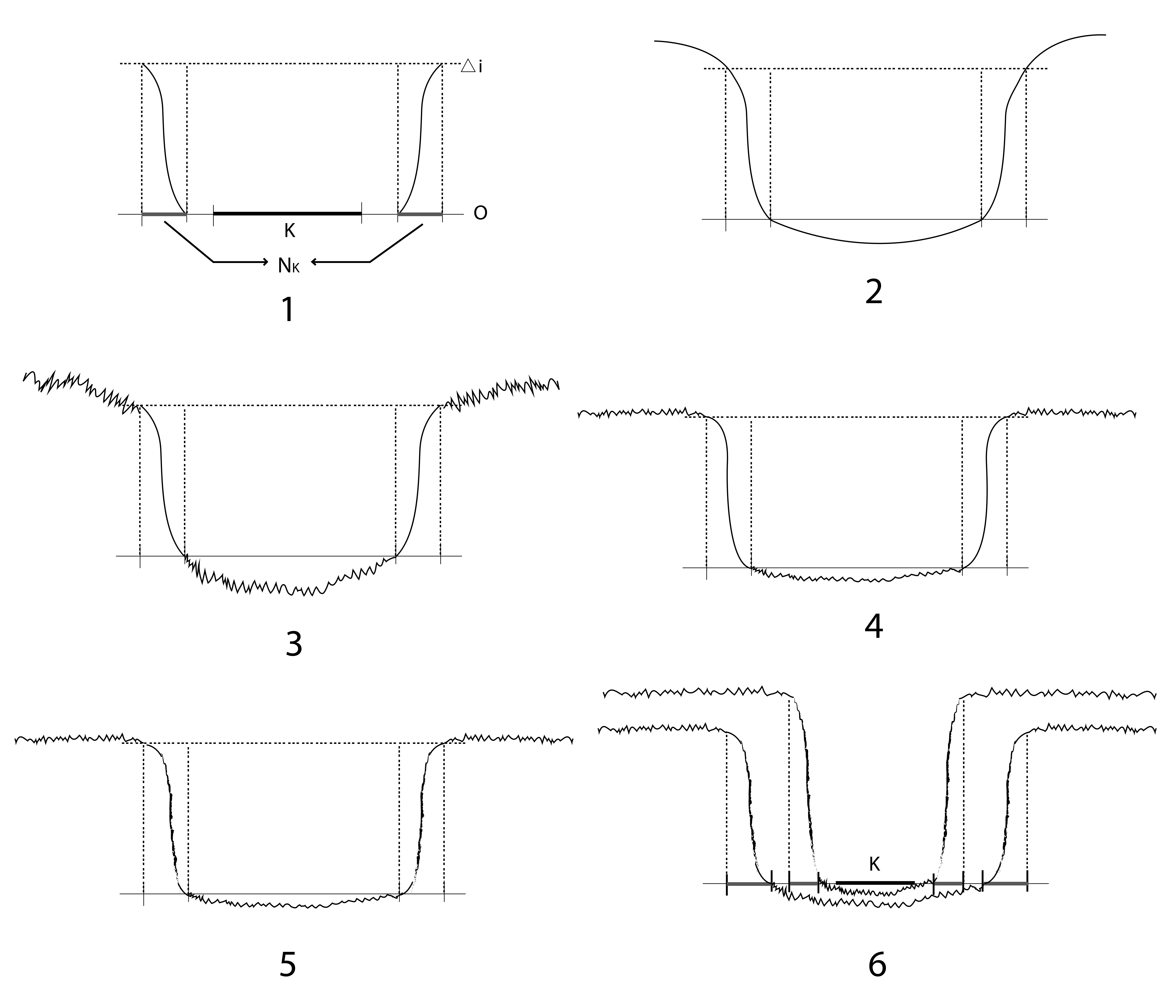}
\caption{This is a summary of the construction of a cofinal sequence for $K$ via boundary accelerators. 1) Boundary accelerators, 2) Extending excitation functions to smooth functions on the entire manifold, 3) Morsifying inside the fillers without changing the function along the mixing regions, 4) Scaling the functions in a neighborhood of the fillers, so that the non-constant one-periodic orbits are forced to lie inside the mixing region, 5) Making the non-constant orbits non-degenerate (note that in reality we start using time dependent Hamiltonians at this step), 6) Two Hamiltonians constructed in this way for $K$ to illustrate how the cofinal family looks. }
\label{c5fsummary}
\end{figure}
\end{center}

\begin{remark}
The main gain from this construction is that we obtained an acceleration data with no non-constant orbits outside of the mixing regions while inside the mixing regions changing the excitation functions only in very controlled ways from what they were originally. A notable challenge in the construction was to achieve non-degeneracy.


In the remaining sections, we will have to go through this construction again, trying to do it for two subsets simultaneously, while satisfying certain extra conditions related to Proposition \ref{c5pmax}. Roughly speaking, the excitation functions will satisfy these extra conditions by the assumptions, and our goal will be to not ruin it, while also achieving the necessary non-degeneracy conditions. 
\end{remark}

\subsection{Non-intersecting boundaries}\label{c5snon}
In this subsection we investigate the case when $X$ and $Y$ are two compact domains with disjoint boundaries.
\begin{definition}
We say that boundary accelerators $(f_i^X,N_X^i,\Delta_i^X)$ and $(f_i^Y,N_Y^i,\Delta^Y_i)$ are \textbf{compatible} if, for all $i\geq 1$, \begin{itemize}
\item $N_X^i$ and $N_Y^i$ are disjoint
\item $\Delta^X_i=\Delta^Y_i$
\end{itemize}
\end{definition}


Let us start with a slight strengthening of a fundamental result of Whitney.

\begin{lemma}\label{c5lwhitney}
Let $C$ be a closed subset of $\mathbb{R}^n$. Then, we can find a smooth function $f:\mathbb{R}^n\to \mathbb{R}$ such that: \begin{itemize}
\item $f\geq 0$
\item $f^{-1}(0)=C$
\item $f$ and all of its iterated partial derivatives vanish at all points of $C$
\end{itemize}
We can also replace  $f\geq 0$ with $f\leq 0$ in this statement. Note also that these properties are still satisfied if we multiply $f$ by a positive real number.
\end{lemma}
\begin{proof}
It is a classical result of Whitney that we can find a $\tilde{f}$ satisfying the second bullet point. Then, we can take $\hat{f}=(\tilde{f})^2$ and satisfy the first two bullet points. Finally, let us fix a bijective smooth function $\rho: [0,\infty)\to [0,\infty)$ such that $\rho^{(n)}(0)=0$ for all $n\geq 0$, and choose $f=\rho\circ\hat{f}$.
\end{proof}

Here is a two function version of Lemma \ref{c5lmorse}.

\begin{lemma}\label{c5lmorsesimul}
Let $D$ be a compact manifold with boundary, and  $D', D''$ be codimension 0 compact submanifolds with boundary such that $D=D'\cup D''$ and $\partial D'\cap \partial D''$ is empty. Assume that we have smooth functions $f:D'\to \mathbb{R}$ and $g:D''\to \mathbb{R}$ such that:
\begin{itemize}
\item $f\leq 0$ with $f^{-1}(0)=\partial D'.$
\item $g\leq 0$ with $g^{-1}(0)=\partial D''.$
\item $f$ and $g$ have no critical points along the boundaries of their domain of definition
\end{itemize}

Let $V\subset int(\{x\in D'\cap D''| f(x)=g(x)\}) $ be a compact set such that $D'-V$ and $D''-V$ both has no connected components where $f-g$ (in its domain of definition inside this connected component) takes strictly positive and strictly negative values or is identically zero.

Let us also choose a positive integer $n$, and positive real number $\delta$. Then, there exists smooth functions $\tilde{f}:D'\to \mathbb{R}$ and $\tilde{g}:D''\to \mathbb{R}$ with the following properties:

\begin{itemize}
\item $\{x\in D'\cap D''| \tilde{f}(x)=\tilde{g}(x)\}=V$.
\item $supp(f-\tilde{f})\subset int(D')$ and $supp(g-\tilde{g})\subset int(D'')$.
\item $\|f-\tilde{f}\|_{C^n(D')}<\delta$ everywhere on $D'$ and $\|g-\tilde{g}\|_{C^n(D'')}<\delta$ everywhere on $D''$, where the $C^n$-norms are defined using a Riemannian metric on $D$.
\item $\tilde{f}< 0$ on $int(D')$ and $\tilde{g}< 0$ on $int(D'')$.
\item $\tilde{f}$ and $\tilde{g}$ are both Morse functions.
\end{itemize}
\end{lemma}

\begin{proof}
Let $N'$, and $N''$ be closed, non-intersecting, and disjoint from $\{x\in D'\cap D''| f(x)=g(x)\}$ collar neighborhoods of $\partial D'$, and $\partial D''$ such that $f$, and $g$ has no critical points on $N'$, and $N''$ respectively. Using Lemma \ref{c5lmorse} inside $int(\{x\in D'\cap D''| f(x)=g(x)\})$, we can simultaneously turn $f$ and $g$ into functions $\hat{f}$ and $\hat{g}$ with only non-degenerate critical points inside $int(\{x\in D'\cap D''| f(x)=g(x)\})=int(\{x\in D'\cap D''| \hat{f}(x)=\hat{g}(x)\})$. 

Now consider the union $U'$ of the connected components of $D'-(N'\cup V)$, where $\hat{g}$ is defined somewhere and $\hat{f}\leq \hat{g}$.  Let $C':=D'-U'$. We choose a function $h'$ as in Lemma \ref{c5lwhitney} for $C'$ with non-positive values. We then apply Lemma \ref{c5lmorse} inside $U'$ to $\hat{f}+h'$, so that the new function is smaller, and inside the connected components of $D'-(N'\cup V)$, where $\hat{g}\leq \hat{f}$, so that the new function is larger. In the components where $\hat{g}$ is not defined we also apply apply Lemma \ref{c5lmorse} but we do not have to stipulate that the new functions is smaller or larger. This gives the function $\tilde{f}$. Note that $\tilde{f}$ is indeed a Morse function as by construction all the critical points are non-degenerate on $D'-(N'\cup V)$ and $V$, and there are no critical points on $N'$.

By exactly the same procedure inside $D''$ with $\tilde{f}$ and $\hat{g}$ we obtain $\tilde{g}$. Note that this is possible because at any point $\tilde{f}\leq \hat{g}$ iff $\hat{f}\leq \hat{g}$ (and similarly for $\geq$). 

It is clear that this procedure can be done so that the third and fourth bullet points are satisfied.
\end{proof}

For two functions defined on the same topological space, we call the set of points at which the functions take the same value the \textbf{region of equality}.

\begin{proposition}
We can find $h_i^X$ and $h_i^Y$ as in Proposition \ref{c5pextension} such that
\begin{itemize}
\item The corresponding boundary accelerators are compatible
\item The locus where $h_i^X=h_i^Y$ is satisfied is a compact domain, for every $i\geq 1$.
\item $min(h_i^X, h_i^Y)\leq min(h_{i+1}^X, h_{i+1}^Y)$ on $fill(\partial N_X^{(i+1)-})\cap fill(\partial N_Y^{(i+1)-})$, for every $i\geq 1$.
\end{itemize}
\end{proposition}

\begin{proof}
We start with any pair of compatible boundary accelerators. We extend the excitation functions to smooth functions as in Step (1) of the proof of Proposition \ref{c5pextension} so that the extensions are the same along a compact domain $P$. Note that $P$ is necessarily a subset of the union of the interiors of $fill(\partial N_X^{i-})\cap fill(\partial N_Y^{i-})$ and $fill(\partial N_X^{i+})\cap fill(\partial N_Y^{i+})$, and we can take it to be the complement of collar neighborhoods of the boundaries of these domains. 

We now want to perturb the excitation functions to achieve the Morse property (Step (2) of Proposition \ref{c5pextension}).  We take a compact domain $P'\subset P$ such that $P-P'$ is a collar neighborhood of $\partial P$ and apply Lemma \ref{c5lmorsesimul} inside $fill(\partial N_X^{i-})\cup fill(\partial N_Y^{i-})$ and $fill(\partial N_X^{i+})\cup fill(\partial N_Y^{i+})$. The lemma as it is written only directly applies to $fill(\partial N_X^{i-})\cup fill(\partial N_Y^{i-})$ but the procedure for $fill(\partial N_X^{i+})\cup fill(\partial N_Y^{i+})$ involves exactly same steps in the spirit of the proof of Proposition \ref{c5pextension}. 

Final step is to make the functions very flat as in Step (3) of the proof of Proposition \ref{c5pextension} compatibly so that the region of equality stays the same. The last bullet point of the statement can be easily satisfied in the process.
\end{proof}

As the last step, we independently apply Lemma \ref{c5lnond} to obtain $H_i^X$ and $H_i^Y$ as in Proposition \ref{c5pfinal}, using that the mixing regions are disjoint. 

\begin{proposition}\label{c5pblow}
\begin{enumerate}
\item The sequence of functions $min(H_i^X, H_i^Y)$, $i\geq 1$, form a non-degenerate cofinal family for $X\cup Y$. Similarly with $max$ for the intersection.
\item For every $i\geq 1$, $H_i^X$ and $H_i^Y$ satisfy all the conditions in Proposition \ref{c5pmax}.
\end{enumerate}
\end{proposition}

\begin{proof}
For every $i\geq 1$, we arranged $H_i^X$ and $H_i^Y$ so that the region of equality is of the form $D\times S^1$ for some compact domain $D$. In particular, we can use the first bullet point of Lemma \ref{lemma12}.

We have that $min(H_i^X, H_i^Y)$ is smooth and non-degenerate by Proposition \ref{c5pmax}, and $min(H_i^X, H_i^Y)\leq min(H_{i+1}^X, H_{i+1}^Y)$ by construction. Cofinality follows from Section \ref{c3sscofinality}. The same statements can be made verbatim for the maximum. 

Notice that the mixing regions, which contain all the non-constant orbits, are disjoint from $D$. Therefore, it follows that condition (3) of Proposition \ref{c5pmax} is also satisfied for $H_i^X$ and $H_i^Y$ (see Remark \ref{c5rconstant}).
\end{proof}

Therefore, we proved: 

\begin{theorem}\label{c5tbasic}
Let $X$ and $Y$ be two compact domains such that $\partial X\cap \partial Y$ is empty. Then, we have an exact sequence:
\begin{align}
\xymatrix{
SH_M(X\cup Y)\ar[r]&SH_M(X)\oplus SH_M(Y)\ar[dl]\\ SH_M(X\cap Y)\ar[u]^{[1]},}
\end{align}where the degree preserving maps are the restriction maps (up to sign).
\end{theorem}

\begin{proof}
The slices of any $3$-ray that is compatible (as in Definition \ref{c5dcompatible3ray}) with any acceleration datum for $X\cup Y, X, Y,$ and $X\cap Y$ that extends the cofinal families $min(H_i^X, H_i^Y), H_i^X, H_i^Y,$ and $max(H_i^X, H_i^Y)$ is acyclic by Proposition \ref{c5pmax}. We remind the reader that such a $3$-ray can be constructed by Proposition \ref{Floertheoreticprop}. Therefore, we have the desired exact sequence by Lemma \ref{c2lacycliccube} and Lemma \ref{c2lmv}.
\end{proof}

\subsection{Barriers}\label{c5sbarriers}

We start with an informal  discussion. Let us consider the simplest example with the boundaries of two domains intersecting to explain what goes wrong for our strategy in general. Take two small disks inside a surface intersecting in the minimal way in an eye-shaped region. Now the Hamiltonians in the acceleration data coming from boundary accelerators will have periodic orbits that make circles around the boundary for all 4 subsets in question. It is clear that in this case no continuation map equation can have topological energy 0 solutions. 

Continuing the informal discussion, we now motivate the definition to come in a slightly simplified setup. Let $N= Y\times [0,1]$ be a symplectic manifold with boundary, and $f:N\to [0,1]$ be any Hamiltonian such that $f^{-1}(0)=Y\times\{0\}$ and $f^{-1}(1)=Y\times\{1\}$. Let $D\subset Y$ be a compact domain, and consider the subset $S:=D\times [0,1]\subset N$. The boundary of $S$ has two portions: the horizontal one that overlaps with the boundary of $N$, and the vertical one coming from the boundary of $D$. We want to come up with a way to guarantee that if an orbit of $X_f$ intersects $S$ then it is contained in it. It appears as though the only feasible way to guarantee this is to assume that $X_f$ has some directionality along the vertical boundary of $S$, more precisely, that $X_f$ cannot be (strictly) inward pointing and (strictly) outward pointing at different points along the vertical boundary of $S$. Let us assume that it is never strictly outward pointing. See Figure \ref{c5fheuristic} for a depiction of the situation. Using energy conversation at the horizontal boundary, this shows that the flow of $X_f$ moves $S$ into itself. But, since Hamiltonian flows preserve volume, this can only happen if $X_f$ is everywhere tangent to the vertical boundary as well.

\begin{figure}
\centering
\includegraphics[scale=0.3]{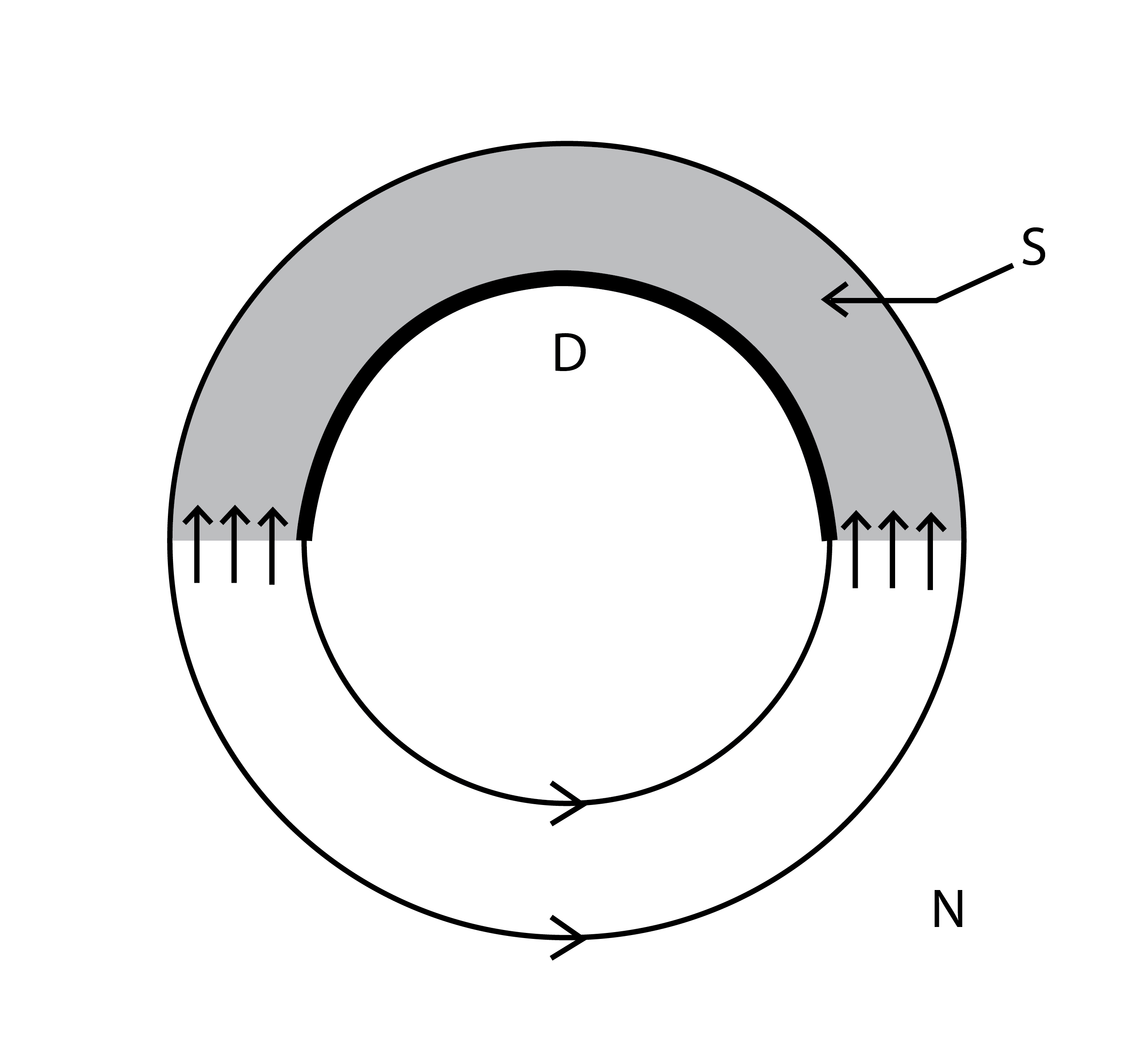}
\caption{The arrows here point in the direction of the Hamiltonian vector field $X_f$}
\label{c5fheuristic}
\end{figure}

Note that this is a very non-generic situation. Energy levels of $f$ will generically be transverse to the vertical boundary. Elementary symplectic geometry shows that intersections of these level sets with the vertical boundary then have to be cosiotropic manifolds of rank 2 (set $X=$ level set, and $Z=X\cap\partial S$ in the following lemma).

\begin{lemma}
Let $X\subset M$ be a hypersurface. Take another hypersurface $Z\subset X$. Then, the characteristic line field of $X$ is tangent to $Z$ if and only if $Z$ is a coisotropic (rank 2). 
\end{lemma}

\begin{proof}
If the characteristic line field is tangent to $Z$, then the kernel of $\omega\mid_Z$ is at least one dimensional. By the classification of skew-symmetric bilinear forms this means that the kernel in question is actually at least two dimensional. By the non-degeneracy of the symplectic form on $M$, we get that $Z$ is a coisotropic.

Conversely, if $Z$ is a coisotropic, then its symplectic orthogonal distribution needs to contain the characteristic line field of $X$. This is because a linear map from a two dimensional vector space to a one dimensional one has at least one dimensional kernel.
\end{proof}

We repeat the definition of a barrier from the introduction in light of this discussion.

\begin{definition}
Let $Z^{2n-2}$ be a closed manifold. We define a \textbf{barrier} to be an embedding $Z\times [-\epsilon,\epsilon]\to M^{2n}$, for some $\epsilon >0$, where $Z\times \{a\}\to M$ is a coisotropic for all $a\in [-\epsilon,\epsilon]$. We call the vector field obtained by pushing forward $\partial_{\epsilon}\in\Gamma(Z\times \{0\},T(Z\times (-\epsilon,\epsilon))\mid_{Z\times \{0\}}) $  to $M$ the \textbf{direction} of the barrier.
\end{definition}

Now, we go back to the formal discussion.

\begin{definition}\label{c5dcompatible}
We say that a Hamiltonian $f:M\to \mathbb{R}$ is \textbf{compatible} with a hypersurface $Y$ (possibly with boundary) if the Hamiltonian vector field $X_f$ is tangent to $Y$ and $\partial Y$.
\end{definition}

Let $B$ be the image of a barrier $Z\times [-\epsilon,\epsilon]\to M$. Sometimes we will abuse notation and denote the barrier by $B$ as well.

\begin{lemma}\label{c5lcompatible}
Let $h:M\to\mathbb{R}$ be a Hamiltonian. If $h$ is constant along $Z\times \{a\}$ for all $a\in [-\epsilon,\epsilon]$, then $h$ is compatible with $B$.
\end{lemma}
\begin{proof}
Let $z\in Z\times \{a\}$. We know that for any vector $v$ at $z$ tangent to $Z\times \{a\}$, the directional derivative of $h$ along $v$ is zero. This is equivalent to $\omega(v,X_h(z))=0$. By coisotropicity, $X_h(z)$ is tangent to $Z\times \{a\}$, finishing the proof.
\end{proof}

\begin{lemma}\label{c5lalmostcompatible}
Take an embedding $B\times [-\delta,\delta]\to M$ extending $B\subset M$. Let $h:M\times [0,1]\to\mathbb{R}$ be a function, and let $\phi_h$ denote the flow of its time-dependent Hamiltonian vector field. Let $\tilde{g}$ be a Riemannian metric on $M$. There exists a $\tau>0$ such that, if for some $B$-compatible $\tilde{h}$, $$
\|h-\tilde{h}\|_{C^2(M\times S^1)} =\sum_{l=0}^2  \sup_{M\times S^1}|\nabla^{l}(h-\tilde{h})(x)|_g<\tau,$$ where $\nabla$ is the Levi-Civita covariant derivative relative to $g=\tilde{g}+dt^2$ and $\tilde{h}$ is considered as a function on $M\times S^1$, then no trajectory of $\phi_h$ starting at a point on $B\times\{\pm\delta\}$ can intersect $B$ within time $1$.
\end{lemma}

\begin{proof}
This is a standard application of Arzela-Ascoli theorem and Gronwall estimates. Assuming the contrary, we reach a contradiction to the fact the Hamiltonian flow of a $B$-compatible function is tangent to the barrier and its boundary. See the proof of Lemma \ref{c5lalmostgeneral} for details.
\end{proof}

\subsection{The proof of the main theorem}\label{c5sproof}

\begin{definition}
We say that a sequence of approximating domains $D_X^i$ and $D_Y^i$ \textbf{have barriers} if there are barriers $Z^i\times [-\epsilon_i,\epsilon_i]\to M$ such that \begin{itemize}
\item $\partial D_X^i\pitchfork \partial D_Y^i=Z^i\times\{0\}$
\item The direction of the barrier points strictly outside of $D_X^i$ and $D_Y^i$
\end{itemize}
\end{definition}

\begin{remark}
We note that $Z^i$'s are not required to be connected, and there might also be some components of $\partial D_X^i$ which do not intersect $\partial D_Y^i$ (and vice versa). 
\end{remark}

\begin{theorem}\label{c5tmv}
Assume that $X$ and $Y$ admit a sequence of approximating domains with barriers. Then, we have an exact sequence:
\begin{align}
\xymatrix{
SH_M(X\cup Y)\ar[r]&SH_M(X)\oplus SH_M(Y)\ar[dl]\\ SH_M(X\cap Y)\ar[u]^{[1]}},
\end{align}
where the degree preserving maps are the restriction maps (up to sign).
\end{theorem}

Our strategy is exactly the same as in the proof of Theorem \ref{c5tbasic}. We will construct a cofinal sequence $H_i^X$ and $H_i^Y$ satisfying the conditions of Proposition \ref{c5pmax}. Of course now we have to deal with the intersection of the mixing regions using the barriers.

\begin{remark}\label{c5rmv}
We stress that, in proving the Mayer-Vietoris property, what matters is that we can construct a cofinal sequence of functions $H_i^X$ and $H_i^Y$ for $X$ and $Y$ that satisfy the conditions of Proposition \ref{c5pmax} for every $i\geq 1$. In Theorem \ref{c5tmv}, we merely described a geometric condition in which such cofinal functions can be constructed. This is not the most general theorem we can prove with our methods.

See Theorem \ref{c5tmv2} in Section \ref{c5sinvolutive} for another statement we prove in this paper. Note that the proof of Theorem \ref{c5tmv2} is simpler (even though its context is more general). This is because the special data that is assumed to exist for the subsets $X$ and $Y$ are already in the form of functions. We do not need to construct functions from the existence of certain geometric objects as in the assumptions of Theorem \ref{c5tmv}.
\end{remark}

The rest of this section is devoted to the proof of Theorem \ref{c5tmv}. All the preparations are put together at the very end, right before Section \ref{c5snond} begins.

\subsubsection{Neighborhoods of intersections of the boundary}\label{c5ssneighborhoods}

Let $K_1$ and $K_2$ be two compact domains in $M$ such that $\partial K_1\pitchfork \partial K_2=Z$. Let $F:D\times Z\to M$ to be an embedding, where $D\subset \mathbb{R}^2$ is an open disk centered at the origin, and the map is identity at $\{0\}\times Z$. In what follows we use that $Z$ is compact without mentioning it.

We can assume that $D\times\{z\}$ is transverse to both $\partial K_1$ and $\partial K_2$ for all $z\in Z$, by restricting the domain of $F$ to a smaller radius disk. This implies that for every $z\in Z$ and $i\in\{1,2\}$, $$F^{-1}(\partial K_i)\cap (D\times\{z\})$$ is a properly embedded $1$-dimensional submanifold of $D\times\{z\}$ passing through the origin. For every $z\in Z$ and $i\in\{1,2\}$, let $l_i(z)\subset D\times\{z\}$ be the oriented straight line passing through the origin that is tangent to $F^{-1}(\partial K_i)\cap (D\times\{z\})$. The orientation is given by requiring that $l_i(z)$ points out of $K_{i+1}$ (subscripts are modulo $2$).

Making \textbf{compactly supported modifications} to a domain $K$ inside $F$ means that we find another domain $K'$ such that outside of a compact subset of $im(F)$ we have $K=K'$.

\begin{proposition}\label{c3pbarrierneighborhood}
Let $l(z)\subset D\times\{z\}$, $z\in Z$, be a smooth (as a map $Z\to S^1$) family of oriented straight lines passing through the origin such that for every $z\in Z$, $l(z)$ is transverse to $l_i(z)$, for $i\in\{1,2\}$. Moreover, we require that $l(z)$ points out of $K_1$ and $K_2$ for every $z\in Z$. Also let $l_i'(z)\subset D\times\{z\}$, $z\in Z$ and $i\in\{1,2\}$, be smooth families of oriented straight lines passing through the origin such that for every $z\in Z$, and $i\in\{1,2\}$, $l_i'(z)$ can be rotated to $l_i(z)$ without crossing $l(z)$. 

Then, we can apply compactly supported modifications to $K_1$ and $K_2$ and restrict the domain of $F$ to smaller radius disk such that in the modified versions, for every $z\in Z$ and $i\in\{1,2\}$, $$F^{-1}(\partial K_i)\cap (D\times\{z\})=l_i'(z).$$
\end{proposition}
\begin{proof}

Let us first deal with the case $l_i'(z)=l_i(z)$ (we will do $i=1$ and $i=2$ seperately in what follows, so fix one of them). For every $z\in Z$, introduce coordinates $(x_z, y_z)$ at $D\times\{z\}$, which are just rotated versions of the standard coordinates on $\mathbb{R}^2$, such that the $x_z$-axis is $l_i(z)$, where $x_z$ is increasing in the positive direction. First, we restrict the domain of $F$ such that the submanifold $F^{-1}(\partial K_i)\cap (D\times\{z\})$ is the graph of a function of $x_z$, for every $z\in Z$. This is possible for example because of a quantitative version of the inverse function theorem, see Supplement 2.5A of \cite{Marsden}. Let us denote the domain of $F$ again by $D\times Z$, with $D$ having radius $r$.

Let $D'$ be a smaller disk with radius $r'<r$. Now, by compactly supported modifications, we will change $K_i$ to $K_i'$ so that $F^{-1}(\partial K_1')\cap D\times\{z\}$ agrees with $l_i(z)$ inside $D'$, for every $z\in Z$. We fix a non-decreasing smooth function $\rho:[0,r]\to[0,1]$ which is equal to $0$ on $[0,r']$, and to $1$ in a neighborhood of $r$. Let each $F^{-1}(\partial K_i)\cap D\times\{z\}$ be the graph of function $f_z$, i.e. $y_z=f_z(x_z)$. We modify these functions as follows: $$f'_z(x_z):=\rho\left(\sqrt{f_z(x_z)^2+x_z^2}\right)f_z(x_z).$$ We define our new domain $K_i$ as the union of the undersets of the graphs of $f'_z(x_z)$. This finishes the argument for $l_i'(z)=l_i(z)$ case.

The general case follows easily from here as there is a canonical way to rotate $l_i(z)$ to $l'_i(z)$ without crossing $l(z)$. We again take a smaller disk with radius $r'<r$, and consider the $\rho$ above (abusing notation). We then make a compactly supported modification to the $K_i$ we obtained in the last paragraph in the following way. From $l_i(z)$ we can define a new curve in $D\times\{z\}$ by rotating each point $w$ on $l_i(z)$ by $\rho(\abs{w})\cdot \theta_z$, where $ \theta_z$ is the angle between $l_i(z)$ and $l'_i(z)$. We continuously vary the half-plane $F^{-1}(K_i)\cap (D\times\{z\})$ along this isotopy, and define the new domain as the union over $z\in Z$. Restricting the domain to the disk with radius $r'$ finishes the proof.

\end{proof}

\begin{remark}Note that we have not changed $F$ in this process, only restricted its domain. We achieved what we wanted by changing the domains. This is allowed because of the flexibility in choosing the approximating domains for our given compact set.
\end{remark}

\subsubsection{Tangentialization}\label{c5sstangentialization}

The last ingredient in the proof is a procedure we call tangentialization. Assume that $X$ and $Y$ admit a sequence of approximating domains with barriers. We want to construct mixing regions for $X$ and $Y$ which can be rearranged to mixing regions for $X\cap Y$ and $X\cup Y$. See Figure \ref{c5ftangent} for a simplified depiction - we will have to be a lot more careful. 

\begin{figure}
\includegraphics[scale=0.12]{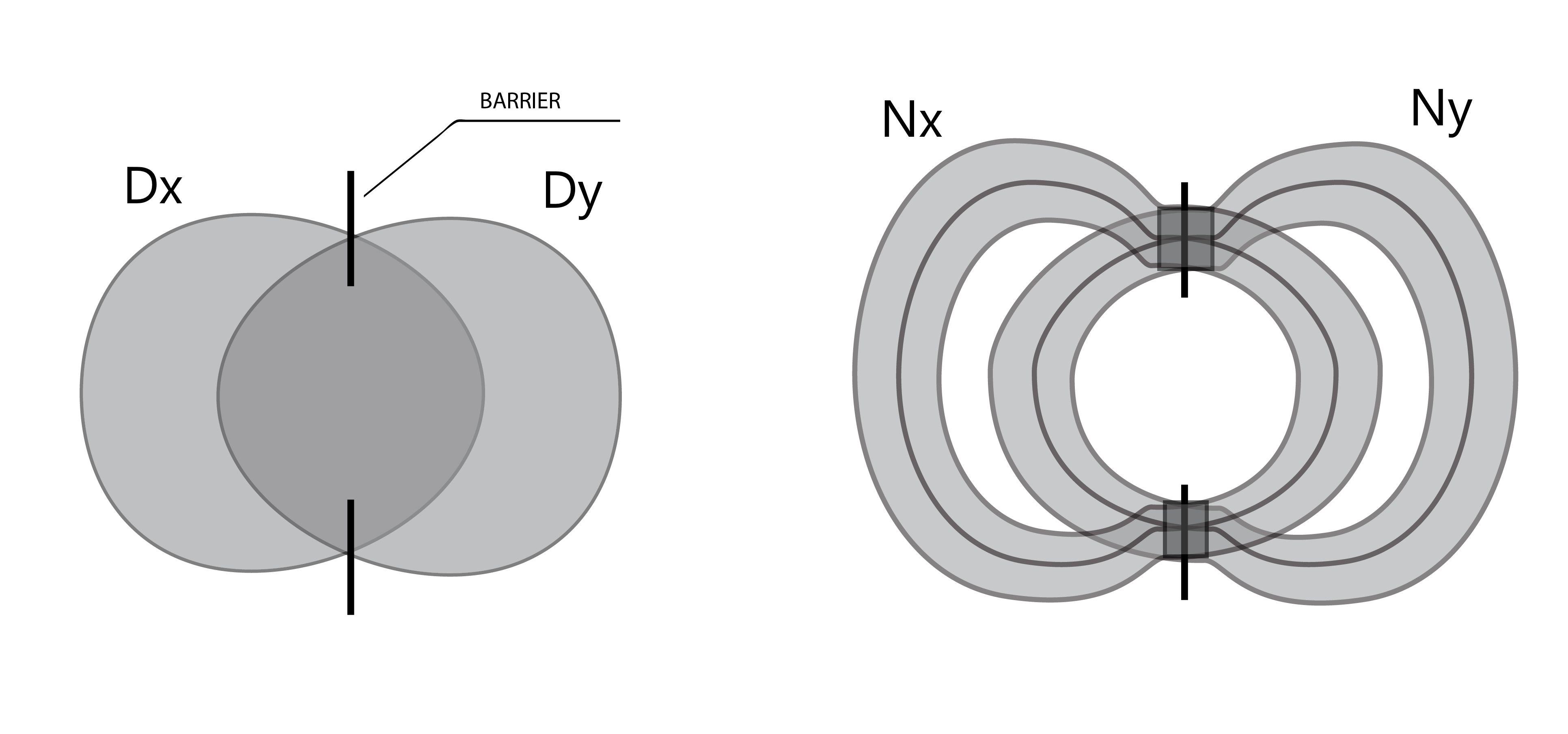}
\caption{A simplified depiction of the tangentialization process. On the left we see a member of the approximating domains with their barrier, and on the right the mixing regions that are compatible with the barrier. Note that all the labels have an $i$ superscript which we dropped from the picture.}
\label{c5ftangent}
\end{figure}

\begin{definition}\label{c5dfriendly}
Let $Z\times [-\epsilon,\epsilon]\to M$ be a barrier. We call an embedding $Z\times [-\epsilon',\epsilon']\times [-\delta,\delta]\to M$, with $\delta>0$,  a \textbf{thickening} of the barrier if the map $Z\times [-\min{(\epsilon, \epsilon')},\min{(\epsilon, \epsilon')}]\times \{0\}\to M$ is the restriction of the barrier. 

Let us call the image of the thickening $P$. A subset $A$ of $M$ is called \textbf{barrier-friendly}  for some thickening if the preimage of $A\cap P$ in $Z\times [-\epsilon',\epsilon']\times [-\delta,\delta]$ is of the form $Z\times S$, where $S$ is a subset of $[-\epsilon',\epsilon']\times [-\delta,\delta]$.
\end{definition}

Let $D_X^i$ and $D_Y^i$ be a sequence of approximating domains with barriers $Z_i\times [-\epsilon_i,\epsilon_i]\to M$. Let us call the barriers $B^i$. We will now use Proposition \ref{c3pbarrierneighborhood}. Take a sufficiently small thickening $Z_i\times [-\epsilon',\epsilon']\times [-\delta,\delta]\to M$, so that $[-\epsilon',\epsilon']\times [-\delta,\delta]\times\{z\}$ is transverse to both $D_X^i$ and $D_Y^i$ for all $z\in Z$. The barrier gives us the  $l(z)$ in the statement, and we choose $l_j'(z)$, $j\in\{1,2\}$, as independent of $z$, which are graphs of non-zero linear functions $[-\delta,\delta]\to [-\epsilon'/2,\epsilon'/2]$ that are horizontal reflections of each other. Hence, using compactly supported modifications and restricting the domain of the thickening, we can assume that $D_X^i$ and $D_Y^i$ are barrier friendly for some thickening and the subset of the square look as in the left picture of Figure \ref{c5fzoom}, because of the outward pointing condition. From now on the thickenings of $B^i$ are fixed, and when we say a subset is barrier friendly, we refer to these thickenings. 

We use the same rectangle $[-\epsilon',\epsilon']\times [-\delta,\delta]$ in our thickenings for every $i\geq 1$. We do this only to not clutter up the exposition with even more notation, there is nothing essential about this uniformity. 

We will abuse notation a little bit and denote the axis $[-\epsilon',\epsilon']\times \{0\}$ the $\epsilon'$-axis, and similarly $\{0\}\times [-\delta,\delta]$ the $\delta$-axis. In the pictures, the former is the vertical axis, whereas the latter is the horizontal one. We stress that the subsets of $[-\epsilon',\epsilon']\times [-\delta,\delta]$ directly give subsets of $M$ (by taking their product with $Z^i$ and using the embeddings), the pictures drawn inside $[-\epsilon',\epsilon']\times [-\delta,\delta]$ are real pictures, not impressions of a complicated reality.

\begin{definition}\label{c5dcompp}
Let $B_i$ as above, for $i\geq 1$. We say that the boundary accelerators $(f_i^X,N_X^i,\Delta_i^X)$ and $(f_i^Y,N_Y^i,\Delta^Y_i)$ are \textbf{compatible with barriers} if, for every $i\geq 0$:\begin{itemize}
\item $\Delta^X_i=\Delta^Y_i$
\item $N_X^i$ and $N_Y^i$ are barrier-friendly for $B_i$ (for the thickening $Z_i\times [-\epsilon',\epsilon']\times [-\delta,\delta]\to M$ that was fixed earlier), with the subsets of the square as described in the right picture of Figure \ref{c5fzoom}. More precisely, we take a curve in $[-\epsilon',\epsilon']\times [-\delta,\delta]$ that is the graph of a non-decreasing smooth function $[-\delta,\delta]\to [-\epsilon'/2,\epsilon'/2]$ that is equal to $0$ exactly for $[-\delta/10,\delta/10]$, and agrees with the linear functions from the left picture of Figure \ref{c5fzoom} near the boundary of $[-\delta,\delta]$. We take one of the subsets as the tubular $\kappa$-neighborhood (for $\kappa$ a sufficiently small positive real number and using the standard metric) of this curve, and the other subset is obtained by reflecting along the $\epsilon'$-axis. 
\item $N_X^i$ and $N_Y^i$ do not intersect outside of the image of the thickening.
\item $f_i^X=f_i^Y$ along the barrier friendly subset obtained from the rectangle that is the product of $[-\delta/10,\delta/10]$ on the $\delta$-axis and $[-\kappa,\kappa]\subset [-\epsilon',\epsilon']$. Moreover, $f_i^X\neq f_i^Y$ everywhere else on $N_X^i\cap N_Y^i$. 
\item $f_i^X$ is compatible with the barrier $Z_i\times [-\kappa,\kappa]\to M$ (obtained by restricting $B_i$) as in Definition \ref{c5dcompatible}, for every $i\geq 1$. The rest of $B_i$ will not play a role from now on.
\end{itemize}
\end{definition}

Let us call the barrier friendly subset obtained from the rectangle that is the product of $[-\delta/10,\delta/10]$ on the $\delta$-axis and $[-\kappa,\kappa]\subset [-\epsilon',\epsilon']$ the \textbf{plaster}. 

\begin{figure}
\includegraphics[scale=0.3]{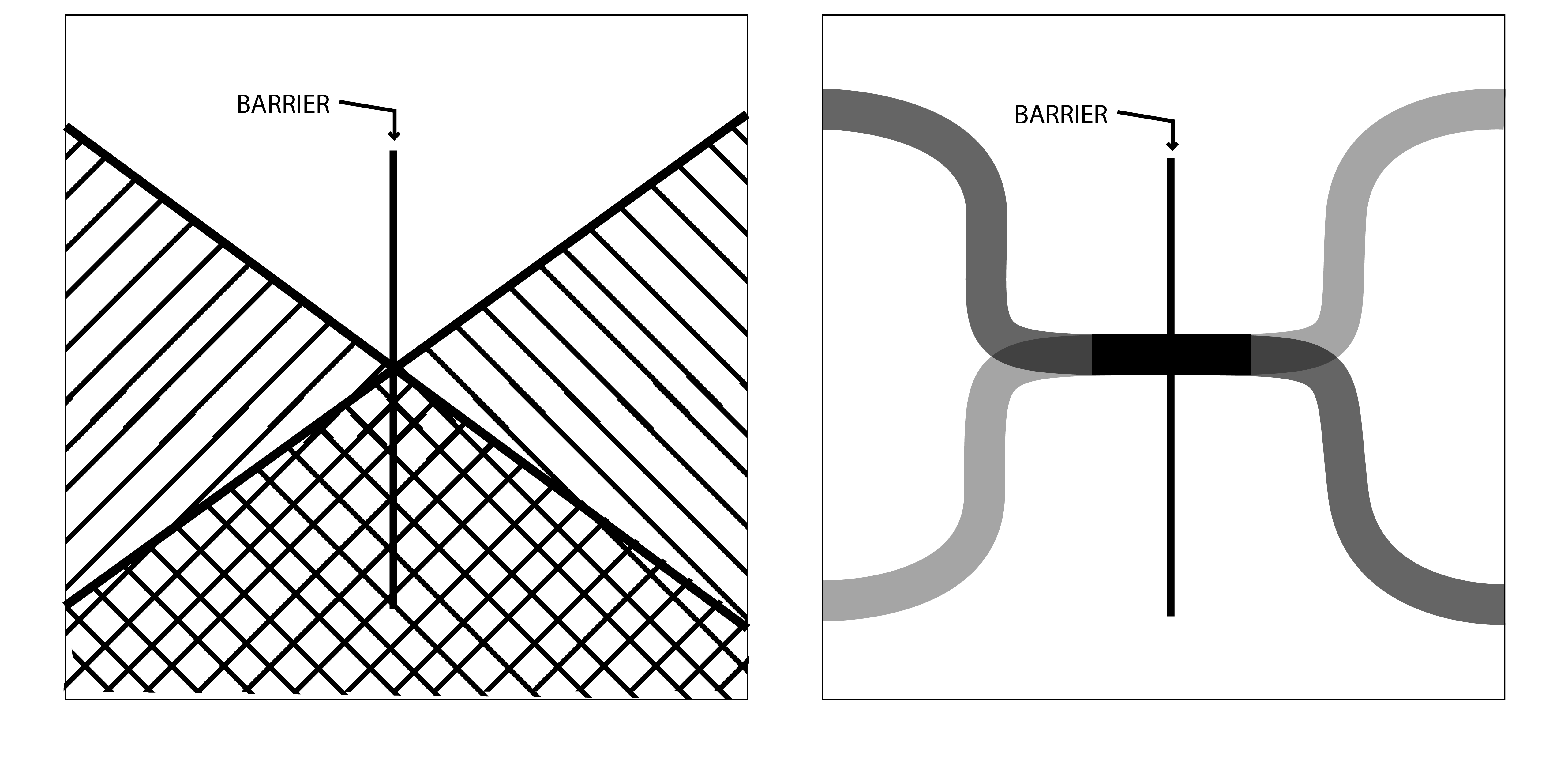}
\caption{The left picture shows the barrier compatible approximating domains after making the original ones barrier friendly by compactly supported perturbations. The right one shows the barrier compatible mixing regions. The plaster is seen as the rectangle in black contained in the intersection of the slices of the mixing regions.}
\label{c5fzoom}
\end{figure}

We now define a standard region of equality (\textbf{SRoE}) for $N_X^i$ and $N_Y^i$ as in Definition \ref{c5dcompp} as follows. This will take a couple of steps. Note that $fill(\partial N_X^{i-})\cap fill(\partial N_Y^{i-})$ and $fill(\partial N_X^{i+})\cap fill(\partial N_Y^{i+})$ are both compact domains. A \textbf{minimum SRoE} is a barrier friendly compact domain contained inside $fill(\partial N_X^{i-})\cap fill(\partial N_Y^{i-})$ whose complement is a collar neighborhood of the boundary of $fill(\partial N_X^{i-})\cap fill(\partial N_Y^{i-})$. A \textbf{maximum SRoE} is similarly defined using $fill(\partial N_X^{i+})\cap fill(\partial N_Y^{i+})$. Finally, let us call the barrier friendly subset of $M$ given by the subset $[-\delta/20,\delta/20]\times [-\epsilon',\epsilon']$ the \textbf{bridge}. An SRoE is defined as a subset that can be represented as the union of a minimum SRoE, a maximum SRoE, the bridge, and the plaster such that the subset of the rectangle $[-\epsilon',\epsilon']\times [-\delta,\delta]$ corresponding to the union of the minimum SRoE, the maximum SRoE, the bridge (which is barrier friendly) is the closure of an open subset of the square that deformation retracts to $\epsilon'$-axis, see the left picture in Figure \ref{c5fblack} for an example. For getting an idea of the overall shape of a SRoE the reader should look at the right picture of Figure \ref{c5fblack}.

\begin{proposition}\label{c5pfinalfiller}
We can find $h_i^X$ and $h_i^Y$ as in Proposition \ref{c5pextension} such that
\begin{itemize}
\item The corresponding boundary accelerators are compatible with barriers as Definition \ref{c5dcompp}.
\item The locus where $h_i^X=h_i^Y$ is satisfied is a topological submanifold with boundary of codimension $0$.
\item $min(h_i^X, h_i^Y)\leq min(h_{i+1}^X, h_{i+1}^Y)$ on $fill(\partial N_X^{(i+1)-})\cap fill(\partial N_Y^{(i+1)-})$, for every $i\geq 1$.
\end{itemize}
\end{proposition}

\begin{proof}We first construct the boundary accelerators that are compatible with the barriers. As explained before we can assume that we have approximating domains $D_X^i$ and $D_Y^i$ and thickenings $Z_i\times [-\epsilon',\epsilon']\times [-\delta,\delta]\to M$ so that the approximating domains are barrier friendly and, in the thickening, look like the left picture of  Figure \ref{c5fzoom}. Now we can do more compactly supported modifications to $D_X^i$ and $D_Y^i$ inside the thickening, and obtain mixing regions that satisfy the second bullet point of Definition \ref{c5dcompp} as $\kappa$-tubular neighborhoods of $\partial D_X^i$ and $\partial D_Y^i$ for a Riemannian metric on $M$ that is the product of the standard metric on $[-\epsilon',\epsilon']\times [-\delta,\delta]$ and some Riemannian metric on $Z_i$ inside the thickening. We construct the excitation functions with the desired properties so that, inside the thickening, they are lifts of functions on the square. 

Next, we extend the excitation functions to smooth functions $\tilde{h}_i^X$ and $\tilde{h}_i^Y$ on $M$ as in the first step of the proof of the Proposition \ref{c5pextension}, so that they are equal along a standard region of equality. We fix an arbitrary SRoE, and define the extensions on it first using the special form of our subsets in the thickening (e.g. we can choose the extensions to be constant along the maximum and minimum SRoE). Then we use the Whitney extension theorem (as in the first paragraph of the proof of Lemma \ref{c3lcontractible}) and Lemma \ref{c5lwhitney} to construct the rest of the desired extensions. Note that the connected components of the complement of our SRoE contain exactly one connected component of $fill(\partial N_Y^{i-}) - fill(\partial N_X^{i-})$ ($X$-dominated) or $fill(\partial N_X^{i-}) - fill(\partial N_Y^{i-})$ ($Y$-dominated). We require that $\tilde{h}_i^X\geq \tilde{h}_i^Y$ on $X$ dominated components, and $\tilde{h}_i^Y\geq \tilde{h}_i^X$ on the $Y$-dominated ones (where the inequalities are strict in the complement of the SRoE of course).

Now, we want to apply a simultaneous Morsification procedure so that the second bullet point of the statement we are proving is still satisfied (the SRoE property is not stipulated anymore) using the argument in the proof of Lemma \ref{c5lmorsesimul}. We first modify $\tilde{h}_i^X$, and $\tilde{h}_i^Y$ in $fill(\partial N_X^{i-})$, and $fill(\partial N_Y^{i-})$, respectively, and then do the outer fillers. Note that in a neighborhood of the boundaries of the mixing regions there are no critical points, and we do not change our functions at this step in such neighborhoods (which we choose). Even though the boundaries of the domains intersect, the proof of Lemma \ref{c5lmorsesimul} extends to this case (also see the proof of Proposition \ref{c5pmixnond}). 

The final step is to flatten the functions as in the Step (3) of \ref{c5pextension} in a compatible way so that the region of equality does not change. The third bullet can easily be satisfied during the procedure, and hence we have our construction.
\end{proof}

\begin{figure}
\centering
\includegraphics[scale=0.5]{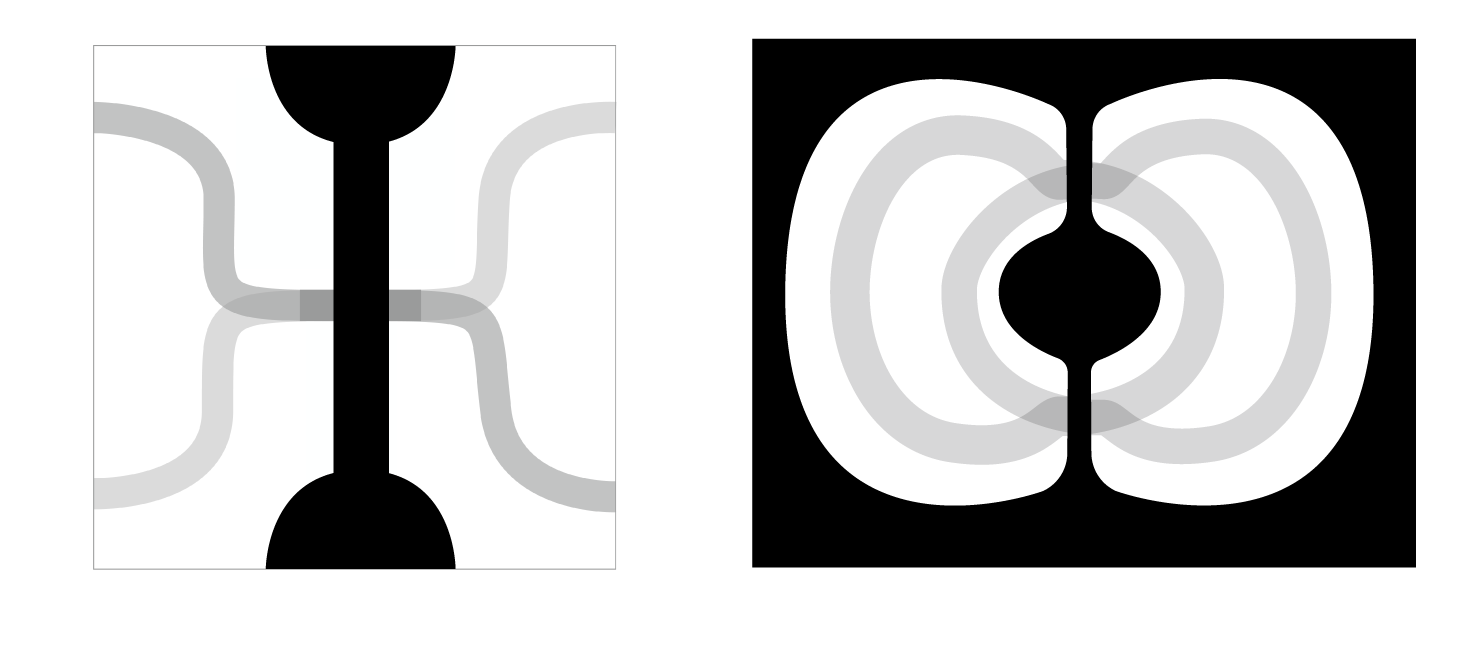}
\caption{On the left we see a SRoE near the barrier drawn in the square of the thickening. This is a picture that represents an actual possibility. On the right we see a simplified depiction of an SRoE as a whole.}
\label{c5fblack}
\end{figure}

Final step is to make the Hamiltonians non-degenerate as in Proposition \ref{c5pfinal}. Note that in Section \ref{c5snon} we did not need a simultaneous perturbation statement in the mixing regions as they were disjoint, but here we do. This is the content of the next proposition. Also note that, in this step, we need to give up the barrier compatibility of the functions (see Section \ref{c5snond} for an explanation).

\begin{proposition}\label{c5pmixnond}
Let $i\geq 1$ and $h_i^X$ and $h_i^Y$ be as in Proposition \ref{c5pfinalfiller}, and assume that we are given a positive real number $c<\delta/10$ (see the definition of the plaster right after Definition \ref{c5dcompatible}). We can find $H_i^X:M\times S^1\to\mathbb{R}$ and $H_i^Y:M\times S^1\to\mathbb{R}$ such that
\begin{itemize}
\item They both satisfy the conditions in Proposition \ref{c5pfinal} (with $n=2$ and for any given $\tau>0$ and Riemannian metric).
\item The regions where $H_i^X(\cdot, t)<H_i^Y(\cdot, t)$ and $H_i^Y(\cdot, t)<H_i^X(\cdot, t)$ are independent of $t\in S^1$. 
\item The locus where $H_i^X=H_i^Y$ is a topological submanifold with boundary of codimension $0$.
\item $H_i^X=H_i^Y$ inside the portion of the thickening given by the subset $[-\kappa,\kappa]\times [-c,c]$. This is a narrowing (in the horizontal direction of the figures) of the plaster.
\item Inside $N_X^i$ any path that goes from a point with $H_i^X<H_i^Y$ to a point with $H_i^X>H_i^Y$ needs to cross the barrier $B_i$. The same statement holds for $N_Y^i$ as well.
\end{itemize}
\end{proposition}

\begin{proof}
We drop the $i$ from the notation. Also note that we are doing nothing outside of $N_X$ to $h^X$, and similarly for $N_Y$ and $h^Y$. Hence, we will think of $h^X$, and $h^Y$ as functions defined on $N_X$, and $N_Y$, respectively. Let $P$ denote the plaster and let $T$ be the subset of $P$ given by the subset $[-\kappa,\kappa]\times [-c,c]$ as in the statement. Let $Q_X$ be a collar neighborhood of the boundary of $N_X$ (inside $N_X$) which does not intersect any $1$-periodic orbit of $h_X$. Similarly, we take $Q_Y$. Moreover, assume that $Q_X$ and $Q_Y$ are both given by $([-\kappa,-\kappa+a]\cup[\kappa-a,\kappa])\times [-\delta/10,\delta/10]$ inside the plaster for some small $a>0$.

Now we use the strategy of the proof of Lemma \ref{c5lmorsesimul}. We start by applying Lemma \ref{c5lnond} in the interior of $[-\kappa+a,\kappa-a]\times [-\delta/10,\delta/10]$ (i.e. the complement of $Q_X$ inside the interior of the plaster) and changing $h^X$ and $h^Y$ simultaneously to $\tilde{h}^X$ and $\tilde{h}^Y$. Note that the new functions are $S^1$-dependent.

Then, we go to $N_X-(T\cup Q_X)$ to modify $\tilde{h}^X$ there. We again consider its three kinds of connected components: \begin{enumerate}
\item $\tilde{h}_Y$ is defined somewhere and $\tilde{h}_X\leq \tilde{h}_Y$
\item $\tilde{h}_Y$ is defined somewhere and $\tilde{h}_Y\leq \tilde{h}_X$ 
\item $\tilde{h}_Y$ is not defined
\end{enumerate}

We first make the inequality strict everywhere in the union of type (1) components, using Lemma \ref{c5lwhitney}. Note that after this step, $1$-periodic orbits of the modified $\tilde{h}^X$ that lie entirely inside $T$ are still non-degenerate. Then, we apply Lemma \ref{c5lnond} in type (1) components so that the new functions are smaller, and in type (2) components so that the new functions are larger. In type (3) components we apply Lemma \ref{c5lnond} without having to be careful about whether the new function is smaller or larger. The resuting function is called $H^X$ and has non-degenerate $1$-periodic orbits. This is because an orbit either intersects $N_X-(T\cup Q_X)$, or is contained inside $T$.

Now, we consider $N_Y-(T\cup Q_Y)$ and modify $\tilde{h}^Y$. Note that $\tilde{h}^X$ is already modified to $H^X$, but this was done in such a way that $\tilde{h}_X\leq \tilde{h}_Y$ at a point if and only if $H_X(\cdot, t) \leq \tilde{h}_Y$ (similarly, with the $\geq$ inequality), where whether the latter inequality holds is independent of $t$. Hence, we can simply repeat the procedure to obtain $H^Y$.

The region of equality for $H^X$ and $H^Y$ is now easily seen to be the union of $T$ and $P\cap Q_X$. It is straightforward to see that we can satisfy all of the requirements in the statement by this procedure.
\end{proof}

\begin{proof}[Proof of Theorem \ref{c5tmv}]
We first construct a sequence of approximating domains and thickenings of the barriers as in the paragraph after Definition \ref{c5dfriendly}. We note that the main operation here is compactly supported modifications to domains as in Section \ref{c5ssneighborhoods}. 

Then, we choose $h_i^X:M\to \mathbb{R}$ and $h_i^Y: M\to\mathbb{R}$ as in Proposition \ref{c5pfinalfiller}. 

Finally, we apply Proposition \ref{c5pmixnond} separately for every $i\geq 1$ to obtain $H_i^X:M\times S^1\to\mathbb{R}$ and $H_i^Y:M\times S^1\to\mathbb{R}$. We choose the $\tau$ in the first bullet point so that the Lemma \ref{c5lalmostcompatible} applies with the thickening there being the one given in the fourth bullet point of Proposition \ref{c5pmixnond} (in particular, we have different $\tau$'s for each $i\geq 1$). 

This finishes the proof of Theorem \ref{c5tmv} by the same statement with Proposition \ref{c5pblow} as used in proving Theorem \ref{c5tbasic}. The main difference here is that the restrictions on the $1$-periodic orbits condition of Proposition \ref{c5pmax} is more serious since the mixing regions are not disjoint. To see why this condition is true, first of all, note that all $1$-periodic orbits are contained in the mixing regions. Moreover, the last two bullet points of Proposition \ref{c5pmixnond} are satisfied, which implies the desired property using Lemma \ref{c5lalmostcompatible} (recall how $\tau$ was chosen in the previous paragraph).
\end{proof}

\subsection{Non-degeneracy}\label{c5snond}

This subsection is a long remark on why we could not restrict ourselves to barrier compatible Hamiltonians in constructing our cofinal sequences in the previous section as a result of the non-degeneracy requirement of the one periodic orbits. We start with an elementary lemma.

\begin{lemma}
The Hamiltonian vector field $X_h$ of a function $h:M\to \mathbb{R}$ is tangent to a submanifold of codimension one $X\subset M$ if and only if it is constant along the characteristic leaves of $X$
\end{lemma}
\begin{proof}
Let $X\subset\{f=0\}$ for some function $f$ with $df\neq 0$ along $X$. Note that $X_f$ is tangent to the characteristic foliation of $X$. We know that $h$ is compatible with $X$ iff $X_h\cdot(f)=0$ along $X$. Moreover, $df(X_h)=\{f,h\}=-dh(X_f)=-X_f(h)$. The claim follows.
\end{proof}

Using the embedding of the barrier, we can construct compatible Hamiltonians with the stronger property that they are constant along the rank 2 coisotropics making up the barrier (as in Lemma \ref{c5lcompatible}). The definition of compatibility does not impose this a priori, but the lemma above shows that we may be forced to it nevertheless as there might be characteristic lines of $B$, which are dense inside $Z\times \{a\}$ for almost all $a$'s

The upshot for us is that we may not have a single compatible $M\times S^1\to\mathbb{R}$ with non-degenerate periodic orbits. The problematic orbits are the ones that lie inside the barrier. In fact if $dim(M)\geq 8$, one can show by a Jacobian computation that in the scenerio described above with the dense characteristic lines we can never make those orbits non-degenerate for barrier compatible Hamiltonians. Fortunately, we can be a little more flexible as in Lemma \ref{c5lalmostcompatible}.

\subsection{Instances of barriers}\label{c5sinstances}

First, note that the outward pointing condition of the barriers can be relaxed to a more cohomological condition. Namely:\begin{proposition}
Assume that we have a sequence of approximating domains $D_X^i$ and $D_Y^i$, and barriers $Z^i\times [-\epsilon_i,\epsilon_i]\to M$ such that \begin{itemize}
\item $\partial D_X^i\pitchfork \partial D_Y^i=Z^i\times\{0\}$
\item The vector field $\partial_{\epsilon_i}$ has winding number 0 with respect to the homotopy class of of trivializations of the normal bundle of $Z_i$ induced by $D_X^i$ and $D_Y^i$.
\end{itemize}
Then there exists a sequence of approximating domains with barriers for $X$ and $Y$.
\end{proposition}
\begin{proof}
This follows from the same argument as in the proof of Proposition \ref{c3pbarrierneighborhood} using compactly supported modifications to the domains. The details are omitted.
\end{proof}

When $dim(M)=2$ the barrier condition can be satisfied only when the boundaries of the approximating domains do not intersect. For $dim(M)=4$, the cohomological condition becomes of importance. 

\begin{lemma}\label{c5ltorus}
Consider the standard neighborhood of a Lagrangian torus $T^2\times\mathbb{R}_p^2,\omega=dq_1\wedge dp_1+dq_2\wedge dp_2$, where we think of $T^2=\mathbb{R}^2_q/\sim$. If $T^2\to T^2\times \mathbb{R}^2$ is a Lagrangian section, which is nowhere zero, then the map $T^2\to\mathbb{R}^2-\{0\}\to S^1$ is nullhomotopic.
\end{lemma}

\begin{proof}
Such Lagrangian sections correspond to closed $1$-forms on $T^2$. Any nowhere vanishing $1$-form $\alpha$ on $T^2$ defines a map $T^2\to\mathbb{R}^2-\{0\}\to S^1$, and we can talk about its homotopy class $h_{\alpha}$. Notice that $h_{\alpha}$ only depends on the cooriented foliaton given by $\alpha$. More precisely, we fix an orientation of $T^2$ and hence a coorientation of the foliation induces an orientation. We also fix a trivialization of $TT^2$ given by the coordinates we used in the statement of the Lemma. Then to any embedded loop $S^1\to T^2$ we can assign a number that is the winding of the oriented line field given by the foliation with respect to the trivialization of the tangent bundle. This number is the same for homotopic loops, and the assigment determines the homotopy class $h_{\alpha}$ in question. In particular, if we can show that, for $d\alpha=0$, the number associated to two non-homotopic non-contractible embedded loops are $0$, we will be done.

By an elementary result of Tischler (\cite{Law} Theorem 29, which follows from the proof of \cite{Ti} Theorem 1), we can find a submersion $\theta:T^2\to S^1$ such that the foliation given by the fibers is arbitrarily close to the foliation defined by $\alpha$. Hence, we are reduced to showing the statement for $\alpha=d\theta$. Notice that we can find an embedded loop that is transverse to all the fibers of $\theta$. If we can show that the winding number of the fiber loops and the transverse loop are both zero, we will be done.

First note that any homotopically non-trivial embedded loop on our torus can be isotoped through embedded loops into a linear loop. This can be shown by unfolding the given loop to $\mathbb{R}^2$. We draw the straight line between its endpoints, and by a small isotopy make our curve transverse to the straight line. Then we cancel intersections between the two curves by isotoping our (curvy) curve along ribbons, using the Schoenflies theorem. We finish by Schoenflies theorem again. This shows that the winding number of the tangent lines of any homotopically non-trivial embedded loop is zero. Applying this statement to the fiber and transverse loops finishes the proof.
\end{proof}

\begin{remark}
Note that if the section is not required to be Lagrangian, meaning that $\alpha$ is not necessarily closed, we can realize all homotopy classes of maps $T^2\to S^1$ by inserting Reeb components. Our proof above is basically showing that when $\alpha$ is closed there can be no Reeb components in the foliation. 
\end{remark}

\begin{remark}
Right before the revision after the referee report is submitted, we realized that Lemma \ref{c5ltorus} (in fact for any $T^n$, $n>0$) was already proved by Laudenbach in \cite{Lau}.
\end{remark}

\begin{corollary}\label{c5ctorus}
Let $D$ and $D'$ be two domains with transversely intersecting boundaries along a disjoint union of Lagrangian tori $L$. Then, $L$ can be extended to an outward pointing barrier if and only if its Lagrangian framing (see \cite{Gompf} for the simple definition) agrees with the framing obtained by the outward pointing normal vectors of $D$ and $D'$. 
\end{corollary}

\subsection{Involutive systems}\label{c5sinvolutive}

In this section, we use Section \ref{c3smultiple}. Recall the following definition from the introduction.

\begin{definition}  We say that compact subsets $K_1,K_2,\ldots K_n$ of $M$ satisfy \textbf{descent}, if $SC_M(K_1,\ldots K_n)$ is acyclic.
\end{definition} 

\subsubsection{A slight generalization of the main theorem}\label{c5ssslight}

\begin{theorem}\label{c5tmv2}
Let $f_i^X:M\to\mathbb{R}$ and $f_i^Y:M\to\mathbb{R}$ be smooth functions such that \begin{enumerate}
\item $(f_i^X)^{-1}((-\infty,0])$ and $(f_i^Y)^{-1}((-\infty,0])$ approximate $X$ and $Y$ respectively
\item Let $f:=(f_i^X,f_i^Y):M\to\mathbb{R}^2$. There exists a smooth curve $C$ passing through the origin once and intersecting only the first and third quadrants (see the left picture of Figure \ref{c5finv}) such that \begin{align}\{f_i^X,f_i^Y\}\mid_{f^{-1}(C)}=0.\end{align} Note that this condition is automatically satisfied if $f^{-1}(C)=\varnothing$.
\end{enumerate}
Then $X$ and $Y$ satisfy descent.
\end{theorem}

\begin{remark}
By Sard's theorem, it does not make a difference whether we assume that $(f_i^X)^{-1}((-\infty,0])$ and $(f_i^Y)^{-1}((-\infty,0])$ are domains or not.
\end{remark}

The proof of this version requires very little modification compared to the proof of Theorem \ref{c5tmv}, see also Remark \ref{c5rmultiple} and Remark \ref{c5rmv}. Of course, $f^{-1}(C)$ plays the role of a barrier. It is a little more general in that it admits a map $f^{-1}(C)\to \mathbb{R}$ with coisotropic fibres, but the fibres are possibly singular. We draw the pictures that we were drawing in the $[-\epsilon',\epsilon']\times[-\delta,\delta]$ rectangle before, for the manipulations near the barrier, in the target plane of the map $f$ near the origin (Figure \ref{c5finvolutive}). In this framework, we can see the entirety of $M$ and the subsets in our pictures, which is convenient. We make the subsets tangential by making the subsets inside $\mathbb{R}^2$ tangential near the origin, tangent direction being transverse to $C$. We construct the excitation functions as functions of $f_i^X$ and $f_i^Y$ (we are still using Definition \ref{c5dboundaryacc} as is). Such functions are all compatible with $f^{-1}(C)$ (in the sense that their Hamiltonian flow leaves $f^{-1}(C)$, and in fact $f^{-1}(c)$, for every $c\in C$, invariant), because of the following lemma (we are using $k=2$ only here).

\begin{lemma}
Let $f_1,\ldots f_k:M\to\mathbb{R}$, and $g_1,g_2:\mathbb{R}^k\to \mathbb{R}$ be smooth functions. Assume that $\{f_i,f_j\}=0$ at $x\in M$, for all $i,j$. Then the functions $G_i:M\to\mathbb{R}$, $i=1,2$, defined by $x\mapsto g_i(f_1(x),\ldots f_k(x))$ also satisfy $\{G_1,G_2\}=0$ at $x\in M$.
\end{lemma}
\begin{proof}
We have that $\{h,h'\}=\omega(X_{h},X_{h'})$. Moreover, $dG_i$ is a $C^{\infty}$ linear combination of $df_1,\ldots df_k$, and hence $X_{G_i}(x)$ is an $\mathbb{R}$-linear combination of $X_{f_1}(x),\ldots X_{f_k}(x)$. This finishes the proof.
\end{proof}

The condition that the excitation functions should have no critical points at the boundary of the mixing regions can be satisfied using Sard's lemma (this implies that the mixing regions are compact domains, which is not very important). Propositions \ref{c5pfinalfiller} and \ref{c5pmixnond} are proved using the same techniques. Note that we can draw an actual picture of an SRoE in the target plane of the map $f$ (see the right picture in Figure \ref{c5finv}). There are two places in which we need a new input. The first one is that we need a generalization of what we had before in Lemma \ref{c5lalmostcompatible}. Of course there was nothing essential about the various smoothness assumptions there. A replacement Lemma \ref{c5lalmostgeneral} is provided below. The second one is that the regions of equality (as in the second bullet point of \ref{c5pfinalfiller} and the third bullet point of \ref{c5pmixnond}) are not necessarily topological submanifolds with boundary of codimension $0$ anymore. Yet, they are preimages of subsets of the plane which are of this class under $f$. Therefore, we can now use the second bullet point of Lemma \ref{lemma12}.

\begin{lemma}\label{c5lalmostgeneral}
Let $S$ and $T$ be closed subsets of $M$ that are disjoint. Let $V$ be a smooth vector field which leaves $S$ invariant under its flow. Let us also fix a Riemannian metric $g$ on $M$. Then, there exists a $\delta>0$ with the property that for every smooth $[0,1]_t$ dependent vector field $V'$ such that $$\sup_{M\times [0,1]}(|V(x)-V'_t(x)|_g+|\nabla(V(x)-V'_t(x))|_g+|\frac{d}{dt}V'_t(x))|_g)<\delta,$$ no point in $T$ can intersect $S$ within time $1$ under the flow of $V'$.
\end{lemma}
\begin{proof}
Assume otherwise: that for a decreasing sequence of positive real numbers $\delta_i\to 0$, there exist smooth $[0,1]$-dependent vector fields $V^i$, such that, for every $i\geq 1$, \begin{itemize}
\item $\sup_{M\times [0,1]}(|V(x)-V^i_t(x)|_g+|\nabla(V(x)-V^i_t(x))|_g+|\frac{d}{dt}V^i_t(x))|_g)<\delta_i.$
\item There exists points $x_i\in T$ and times $0< t_i\leq 1$ such that $\phi_i^{t_i}(x_i)\in S$, where $\phi_i^t$ is the flow of $V^i$.
\end{itemize}
Let us note the following inequalities that follow: \begin{align}\label{c5eczero}\sup_{M\times [0,1]}(|V(x)-V^i_t(x)|_g)<\delta_i,
\end{align} 
\begin{align} \sup_{M}(|\nabla V(x)|_{g})<C,\end{align} and 
\begin{align}\label{c5econe} \sup_{M\times [0,1]}(|\tilde{\nabla} V^i_t(x)|_{\tilde{g}})<C,\end{align} where $\tilde{g}$ is the Riemannian metric $g+dt^2$ on $M\times [0,1]$, $\tilde{\nabla}$ its covariant derivative operator, and $C>0$ is some constant (independent of $i$).

By Arzela-Ascoli theorem, we can assume that the paths $\gamma_i:=\phi_i^{t}(x_i):[0,1]\to M$ converge $C^0$-wise. Let $\gamma:[0,1]\to M$ be the limit path. Let us also define $\phi^t$ be the flow of $V$, and $\psi:=\phi^t(\gamma(0)):[0,1]\to M$. Note that $\gamma(0)$ is in $T$, because $T$ is closed. We claim that $\gamma$ is an integral curve of $V$, i.e. we want to show that $\gamma=\psi$.

%

Let us define $\gamma_{i,j}:=\phi_i^{t}(x_j):[0,1]\to M$ for every $i,j\geq 1$. Let us also define $\psi_i:=\phi_i^{t}(\gamma(0)):[0,1]\to M$. It follows from the Gronwall estimates (\cite{Kun}, Theorem 2 and 4) and (\ref{c5econe}) that, for every $\epsilon>0$, there exists a $j_0\geq 1$ such that for every $i\geq 1$ and $j\geq j_0$, $$\sup_{t\in[0,1]}d(\gamma_{i,j}(t),\psi_i(t))<\epsilon.$$ Note that in order to apply the estimates of \cite{Kun}, we are employing the standard trick of turning a (possibly) time dependent vector field $W_t$ on $M$ to a time independent vector field $W_t+\partial_t$ on $M\times [0,1]_t$, and use the product Riemannian metric $\tilde{g}$.

With a similar strategy we show now that for every $\epsilon>0$, there exists a $i_0\geq 1$ such that for every $i\geq i_0$, $$\sup_{t\in[0,1]}d(\psi_{i}(t),\psi(t))<\epsilon.$$ Consider the vector field $$\hat{V}^i(x,t,s):=\left(1-\frac{s}{\delta_i}\right) V(x)+\frac{s}{\delta_i}V_t^i(x)+\frac{\partial}{\partial t}$$ on $M\times [0,1]_t\times [0,\delta_i]_s$. We use the product Riemannian metric on $M\times [0,1]_t\times [0,\delta_i]_s$, called $\hat{g}$. Note that
\begin{align}\sup_{M\times [0,1]}(|\hat{\nabla} \hat{V}^i_t(x)|_{\hat{g}}<2C+1,\end{align} using all three of (\ref{c5eczero})-(\ref{c5econe}). By applying the Gronwall estimate for the integral curves of $\hat{V}_i$ starting at $(\gamma(0),0,0)$ and $(\gamma(0),0,\delta_i)$, we obtain the desired result.

Now, an application of triangle inequality shows that $$\gamma_{i}=\gamma_{i,i}\to \psi,$$  $C^0$-wise. This shows that $\gamma=\psi$ by uniqueness of limits.

Let us show that this is a contradiction. Consider the sequence of times $t_i\in [0,1]$ as in the second bullet point above. By passing to a subsequence, we can assume that $t_i\to t^*$, for some $0<t^*\leq 1$. We see that $\gamma_i(t_i)$ then has to converge to $\gamma(t^*)$ by the triangle inequality and the equicontinuity of $\gamma_i$. 

Finally, because $S$ is closed, $\gamma(t^*)$ is in $S$. It gives the desired contradiction using that $S$ is invariant under $\phi^t$, as $\gamma$ starts at $T$ but intersects $S$.
\end{proof}

\begin{figure}\label{c5finv}
\includegraphics[scale=0.3]{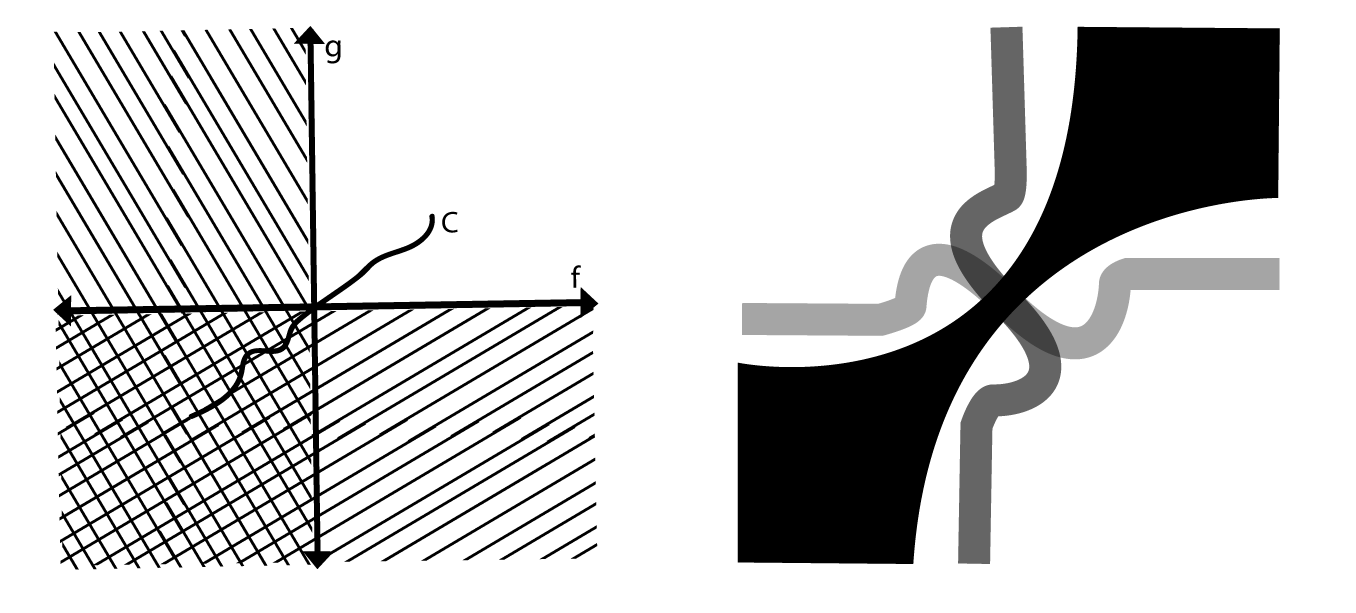}
\caption{On the left we see the approximating sets and the barrier in the in the target plane of the map $f$. On the right we see an SRoE for the functions before the non-degeneracy requirements are satisfied.}
\label{c5finvolutive}
\end{figure}

\subsubsection{Descent for symplectic manifolds with involutive structure}\label{c5ssdescent}

\begin{definition} An \textbf{involutive map} is a smooth map $\pi: M\to B$ to a smooth manifold $B$, such that for any $f,g\in C^{\infty}(B)$, we have $\{f\circ\pi,g\circ\pi\}=0$
\end{definition}

\begin{theorem}
Let $X_1,\ldots X_n$ be closed subsets of $B$. Then $\pi^{-1}(X_1),\ldots \pi^{-1}(X_n)$ satisfy descent.
\end{theorem}
\begin{proof}
It suffices to show this for $n=2$ (see Appendix B for the easy inductive argument). In that case, we have already proved a stronger version in Theorem \ref{c5tmv2}, as we can use functions on $B$ to get the sequences of functions in Theorem \ref{c5tmv2}. 
\end{proof}


\appendix
\section{Cubical diagrams from simplicial ones}\label{appa}

We show that $n$-cube families of Hamiltonians give $n$-cubes using Pardon's results on simplex families. The main challenge here is to show that the signs work out correctly.

Let $Cube=[0,1]^n$, with an ordering of its coordinates. We can cover it by $n!$ simplices, one for each permutation $(i_1,\ldots,i_n)$ of $(1,\ldots, n)$. We can think of such a permutation as a path that starts at $(0,\ldots ,0)$ and takes a unit step in the positive $i_k$ direction at time $k=1,\ldots ,n$, and ends up at $(1,\ldots ,1)$. The corresponding simplex $\Delta^n\to Cube$ is the linear map that sends the $ith$ vertex of the simplex to the $i$th vertex we encounter on this path. Note also that the function $f$ (as in \ref{c3emorse}) is tangent to all the faces of all these simplices.

Now let $H$ be a $Cube$ family of Hamiltonians. Let $F$ be a $k$-dimensional face of the $Cube$. $F$ itself is a cube with an induced ordering of its coordinates. By the above procedure it can be covered by $k!$ simplices. We get a map $C_{\nu_{in}}\to C_{\nu_{ter}}$ for each of these simplices by restricting the family of Hamiltonians. We define   \begin{align}
f_F=\sum_{k!\text{ simplices}} (-1)^{sign} f_{(j_1,\ldots ,j_{k})},
\end{align}where $(j_1,\ldots ,j_{k})$ is the permutation corresponding to the simplex, and sign is given by its signature. We claim that these define a cubical diagram. We need to show that the Equation (\ref{c2ecube}) from  subsubsection \ref{c2sscubes}:
\begin{align}
\sum_{F'>F'' \text{is a bdry of }F}(-1)^{*_{F',F}} f_{F''}f_{F'}=0,
\end{align} is satisfied for each face $F$. Recall that $*_{F',F}=\#_v1+\#_v01$ for $v=\nu_{ter}F'-\nu_{in}F'$ considered as a vector inside $F$. Without loss of generality, we will show the one for the top dimensional face.   

By Pardon (Equation (7.6.5)), we get $n!$ equations of the form below for the top dimensional face of each of the simplices in the cover. 
\begin{align}
\sum_k(-1)^{k+1}g_{(1,\ldots , k)}g_{(k,\ldots n)}+\sum_k(-1)^kg_{(1,\ldots,\hat{k} , \ldots, n)}=0
\end{align}
We add all of these equations up after multiplying them with $(-1)^{sign}$, where sign is again the signature of the permutation corresponding to the simplex. The second group of terms cancel out because the signature of a permutation changes after one transposition. Using the description of the signature of a permutation via the number of inversions, we see that we get exactly the equation we wanted from the first group of terms.

We also need that $Slit^n$ and $Triangle^n$ families of Hamiltonians give rise to $n$-slits and $n$-triangles. A similar argument to the one for $Cube^n$ families (as above) works for $Triangle^n$ as well, as it also admits a nice cover by simplices. For $Slit^n$ the equation that corresponds to the top dimensional face requires a separate treatment but this is easy, and we omit it.

\section{Descent for multiple subsets}\label{appb}

Let $K_1,\ldots ,K_n$ be compact subsets of $M$. Let $\mathcal{K}$ be smallest set of subsets of $M$, which is closed under intersection and union, and contains $K_1,\ldots ,K_n$.

Assume that for any $X,Y\in \mathcal{K}$, $SH_M(X,Y)=0$. Then, we want to show that for any $X_1,\ldots , X_l\in \mathcal{K}$, $SH_M(X_1,\ldots , X_l)=0$. We do this by induction. Assume that it holds for $l-1\geq 2$.

By the descent for two subsets we have that the natural map \begin{align}\label{abe} \xymatrix{SC_M(X_1\cup\ldots\cup X_l)\ar[d] \\ *+[l]{\makebox[.1\textwidth]{$cone(SC_M(X_1)\oplus SC_M(X_2\cup\ldots\cup X_l)\to SC_M((X_2\cap X_1)\cup\ldots\cup (X_l\cap X_1)))$}}}\end{align} is a quasi-isomorphism.

We also have homotopy commutative diagrams:
\begin{align}
\xymatrix{
SC_M(X_2\cup\ldots\cup X_l) \ar[d]\ar[r]  &SC_M((X_2\cap X_1)\cup\ldots\cup (X_l\cap X_1)))\ar[d] \\ \bigoplus_{0\neq I\subset \{2,\ldots ,l\}} SC_M\left(\bigcap_{i\in I} X_i\right) \ar[r] &\bigoplus_{0\neq I\subset \{2,\ldots ,l\}} SC_M\left(\bigcap_{i\in I} (X_i\cap X_1\right),}
\end{align} and
\begin{align}
\xymatrix{
SC_M(X_1) \ar[dr]\ar[r]  &SC_M((X_2\cap X_1)\cup\ldots\cup (X_l\cap X_1)))\ar[d] \\ &\bigoplus_{0\neq I\subset \{2,\ldots ,l\}} SC_M\left(\bigcap_{i\in I} (X_i\cap X_1\right),}
\end{align}

In these two diagrams, by the direct sum we mean the homotopy limit of the corresponding homotopy coherent diagram (it is a chain complex with that underlying module, and differential can explicitly written using the same techniques as in Section \ref{c2}). By the induction hypothesis all the vertical arrows are quasi-isomorphisms.

By piecing together these diagrams, we see that the cone in (\ref{abe}) is quasi-isomorphic to $\bigoplus_{0\neq I\subset \{1,\ldots ,l\}} SC_M\left(\bigcap_{i\in I} X_i\right)$ in a way that is compatible with the maps that they receive from $SC_M(X_1\cup\ldots\cup X_l)$. This finishes the proof.

\section{Analysis of the local model of the Morse pair near the corners}\label{appc}

Consider $\mathbb{R}^n_{\geq 0}$ as a manifold with corners with coordinates $(x_1,\ldots x_n)$, and the function $$h(x_1,\ldots ,x_n):=x_1^2+\ldots +x_k^2-x_{k+1}^2-\ldots -x_n^2.$$ Let $V$ be the gradient vector field of $h$ with respect to the flat metric.

Let us first prove that $h^{-1}(w)\subset \mathbb{R}^n_{\geq 0}$ is a $p$-submanifold, for $w>0$.  Let $(a_1,\ldots ,a_n)\in h^{-1}(w)$. At least one of $a_1,\ldots a_k$ is not zero. Without loss of generality, let $a_1\neq 0$. Then, $(h-w,x_2,\ldots x_n)$ is a local manifold with corners chart at $(a_1,\ldots ,a_n)$. Since in this chart $h^{-1}(w)$ is given by one of these coordinates being equal $0$, we have the desired statement. The case $w<0$ is similarly handled.

Let us define the radial functions $r_1:=\sqrt{x_1^2+\ldots +x_k^2}$ and $r_2:=\sqrt{x_{k+1}^2+\ldots +x_n^2}$; and also the radial vector fields $\widehat{r_1}=x_1\partial_{x_1}+\ldots +x_k\partial_{x_k}$ and $\widehat{r_2}=x_{k+1}\partial_{x_{k+1}}+\ldots +x_n\partial_{x_n}$. Then, one can easily compute that $$h=r_1^2-r_2^2,$$ and on $\mathbb{R}^n_{\geq 0}-\{0\}$, $$-\frac{V(x_1,\ldots ,x_n))}{\abs{V(x_1,\ldots ,x_n)}^2}=\frac{\widehat{r_2}-\widehat{r_1}}{2(r_1^2+r_2^2)}.$$

Let us call this vector field $W$. Let us also consider the map $$r: \mathbb{R}^n_{\geq 0}\to \mathbb{R}^2_{\geq 0},$$ given by $(x_1,\ldots ,x_n)\mapsto (r_1,r_2)$. Polar coordinates on $\mathbb{R}^k_{\geq 0}$ and $\mathbb{R}^{n-k}_{\geq 0}$ provides us with a map
\begin{align}\label{appce}
\xymatrix{
\pi: \mathbb{R}^2_{\geq 0}\times \mathbb{S}^{k-1}\times \mathbb{S}^{n-k-1} \ar[dr]\ar[r]   &\mathbb{R}^n_{\geq 0}\ar[d] \\ &\mathbb{R}^2_{\geq 0}.}
\end{align} We use the notation $\mathbb{S}^m\in \mathbb{R}^m_{\geq 0}$ to denote the intersection of the unit sphere in $\mathbb{R}^m$ with $\mathbb{R}^m_{\geq 0}$, for any $m\geq 1$. Note that $ \mathbb{R}^2_{\geq 0}\times \mathbb{S}^{k-1}\times \mathbb{S}^{n-k-1}$ is the iterated blow-up $$[\mathbb{R}^n_{\geq 0}; \{x_{1}=\ldots =x_k=0\}, \{x_{k+1}=\ldots =x_n=0\}],$$and $\pi$ is the blow-down map.

The vector field $W$ has a canonical lift $\tilde{W}$ to $(\mathbb{R}^2_{\geq 0}-\{0\})\times \mathbb{S}^{k-1}\times \mathbb{S}^{n-k-1}$. The flow of $\tilde{W}$ preserves the $\mathbb{S}^{k-1}\times \mathbb{S}^{n-k-1}$ coordinates, and on $\mathbb{R}^2_{\geq 0}-\{0\}$, it is given by the vector field $$\frac{r_2\partial_{r_2}-r_1\partial_{r_1}}{2(r_1^2+r_2^2)}.$$ It is easy to check that if $t\mapsto (r_1(t),r_2(t))$ is an integral curve, then $r_1(t)r_2(t)$ is constant, and expectedly $\frac{d}{dt}(r_1(t)^2-r_2(t)^2)=-1$. This determines the flow completely. Also note that the flow (resp. reverse flow) can be extended to continuous maps $\{r_1^2-r_2^2=w\}\to \{r_1^2-r_2^2=0\}$ for every $w>0$ (resp. $w<0$). Finally, notice that the maps \begin{align}\label{appccc}\{r_1^2-r_2^2=w\}\to \{r_1^2-r_2^2=w'\},\end{align} for every $w>0$ and $w'<0$, obtained by composition of the maps in the previous sentence, are diffeomorphisms.

From this discussion it follows that for every $w>0$, we have continous maps $h^{-1}(w)\to h^{-1}(0)$ extending the map defined by the time-$w$ flow of $W$ defined on $h^{-1}(w)-h^{-1}(w)\cap\{x_{k+1}=\ldots =x_n=0\}$. Similarly, for $w<0$, we have $h^{-1}(w)\to h^{-1}(0)$ extending the map defined via the inverse flow.

Going back to (\ref{appce}), first we claim that for $w>0$, $\pi^{-1}(h^{-1}(w))$ is the blow-up of $h^{-1}(w)$ at $h^{-1}(w)\cap \{x_{k+1}=\ldots =x_n=0\},$ and the restriction of $\pi$: $$\pi^{-1}(h^{-1}(w))\to h^{-1}(w)$$ is the blow-down map. This immediately follows because $h^{-1}(w)$ is a $p$-submanifold that transversely intersects the blow-up locus $\{x_{k+1}=\ldots =x_n=0\}$, and it does not intersect the other blow-up locus $\{x_{1}=\ldots =x_k=0\}$.

Similarly, for $w<0$, $\pi^{-1}(h^{-1}(w))$ is the blow-up of $h^{-1}(w)$ at $h^{-1}(w)\cap \{x_{1}=\ldots =x_k=0\},$ and the restriction of $\pi$ is the blow-down map. 

The upshot of the discussion is that, for $w'<0<w$, we have the following commutative diagram

\begin{align}\label{appcee}
\xymatrix{
\pi^{-1}(h^{-1}(w))\ar[rr]\ar[d] &&\pi^{-1}(h^{-1}(w'))\ar[d] \\ h^{-1}(w)\ar[dr] && h^{-1}(w')\ar[dl] \\&h^{-1}(0)&,}
\end{align}where the horizontal arrow is the diffeomorphism from \ref{appccc} (hence we can identify its source and target with each other and think of them as one manifold with corners), the vertical arrows are the blow-down maps that we just discussed, and the diagonal arrows are the continuous maps from the two paragraphs above.

\bibliographystyle{plain}
\bibliography{mv-bib,thesisbib}

\begin{thebibliography}{10}

\bibitem{AA}
Mohammed Abouzaid.
\newblock {\em Symplectic cohomology and Viterbo's theorem}.
\newblock IRMA Lect. Math, 2015.

\bibitem{AS}
Mohammed Abouzaid and Paul Seidel.
\newblock An open string analogue of {V}iterbo functoriality.
\newblock {\em Geometry \& Topology}, 14(2):627--718, 2010.

\bibitem{Marsden}
R.~Abraham, J.~E. Marsden, and T.~Ratiu.
\newblock {\em Manifolds, Tensor Analysis, and Applications}.
\newblock Springer-Verlag, 1988.

\bibitem{ACC}
E.~Accinelli and E.~Covarrubias.
\newblock An extension of the {S}ard-{S}male theorem to convex domains with an
  empty interior.
\newblock {\em Journal of Mathematical Economics}, 55:123 -- 128, 2014.

\bibitem{Audin}
Mich{\`e}le Audin and Mihai Damian.
\newblock {\em {Morse theory and Floer homology}}.
\newblock Universitext. Springer, London, 2014.

\bibitem{Bosch_2014}
Siegfried Bosch.
\newblock Lectures on formal and rigid geometry.
\newblock {\em Lecture Notes in Mathematics}, 2014.

\bibitem{Bur}
Dan {Burghelea} and S.~Haller.
\newblock On the topology and analysis of a closed one-form i.
\newblock {\em Monogr. Enseign. Math.}, 38:133--175, 2001.

\bibitem{CFH}
K.~Cieliebak, A.~Floer, and H.~Hofer.
\newblock Symplectic homology. {II}. {A} general construction.
\newblock {\em Mathematische Zeitschrift}, 218(1):103--122, 1995.

\bibitem{CO}
K.~Cieliebak and A.~Oancea.
\newblock Symplectic homology and the {E}ilenberg-{S}teenrod axioms.
\newblock {\em Algebraic \& Geometric Topology 18}, pages 1953--2130, 2018.

\bibitem{CJS}
Ralph~L. Cohen, Ralph~L. Cohen, John D.~S. Jones, and Graeme~B." Segal.
\newblock Morse theory and classifying spaces.
\newblock 1995.

\bibitem{Datta}
Rankeya Datta and Karen~E. Smith.
\newblock Frobenius and valuation rings.
\newblock {\em Algebra Number Theory}, 10(5):1057--1090, 2016.

\bibitem{EP}
Michael Entov and Leonid Polterovich.
\newblock Quasi-states and symplectic intersections.
\newblock {\em Commentarii Mathematici Helvetici. A Journal of the Swiss
  Mathematical Society}, 81(1):75--99, 2006.

\bibitem{Gompf}
Ronald Fintushel and Ronald~J. Stern.
\newblock Invariants for {L}agrangian tori.
\newblock {\em Geometry and Topology}, 8:947--968, 2004.

\bibitem{FHcoh}
A.~Floer and H.~Hofer.
\newblock Coherent orientations for periodic orbit problems in symplectic
  geometry.
\newblock {\em Mathematische Zeitschrift}, 212(1):13--38, Jan 1993.

\bibitem{FH}
A.~Floer and H.~Hofer.
\newblock Symplectic homology. {I}. {O}pen sets in {${\bf C}^n$}.
\newblock {\em Mathematische Zeitschrift}, 215(1):37--88, 1994.

\bibitem{GPS}
Sheel {Ganatra}, John {Pardon}, and Vivek {Shende}.
\newblock {Covariantly functorial wrapped Floer theory on Liouville sectors}.
\newblock {\em arXiv e-prints}, page arXiv:1706.03152, Jun 2019.

\bibitem{G}
Yoel {Groman}.
\newblock {Floer theory and reduced cohomology on open manifolds}.
\newblock {\em ArXiv e-prints 1510.04265}, October 2015.

\bibitem{HS}
Helmut Hofer and Dietmar Salamon.
\newblock Floer homology and novikov rings.
\newblock {\em Floer Memorial Volume, Progress in math. 133, Birkhause}, 1992.

\bibitem{KM}
P.~B. Kronheimer and T.~S. Mrowka.
\newblock Khovanov homology is an unknot-detector.
\newblock {\em Publications Math\'ematiques. Institut de Hautes \'Etudes
  Scientifiques}, (113):97--208, 2011.

\bibitem{Kun}
M~Kunzinger, H~Schichl, R~Steinbauer, and J~A Vickers.
\newblock {Global Gronwall Estimates for Integral Curves on Riemannian
  Manifolds}.
\newblock {\em Rev. Mat. Complut.}, 19(1):133--137, Dec 2006.

\bibitem{Lau}
F.~Laudenbach.
\newblock Formes diff{\'e}rentielles de degr{\'e} 1 ferm{\'e}es non
  singuli{\`e}res: Classes d'homotopie de leurs noyaux.
\newblock {\em Commentarii Mathematici Helvetici}, 51(1):447--464, Dec 1976.

\bibitem{Law}
H.~Blaine Lawson, Jr.
\newblock Foliations.
\newblock {\em Bulletin of the American Mathematical Society}, 80:369--418,
  1974.

\bibitem{Lee}
John~M. Lee.
\newblock {\em Introduction to smooth manifolds}, volume 218 of {\em Graduate
  Texts in Mathematics}.
\newblock Springer, New York, second edition, 2013.

\bibitem{Sar}
Robert Lipshitz and Sucharit Sarkar.
\newblock A khovanov stable homotopy type.
\newblock {\em Journal of the American Mathematical Society}, 27(4):983--1042,
  Apr 2014.

\bibitem{Lunts}
Valery~A. Lunts and Olaf~M. Schn\"{u}rer.
\newblock New enhancements of derived categories of coherent sheaves and
  applications.
\newblock {\em J. Algebra}, 446:203--274, 2016.

\bibitem{Mc}
Mark Mclean.
\newblock Birational {C}alabi-{Y}au manifolds have the same small quantum
  product.
\newblock {\em preprint}, 2018.

\bibitem{Mel}
Richard~B Melrose.
\newblock Differential analysis on manifolds with corners, 1996.

\bibitem{Palais}
Richard~S Palais.
\newblock When proper maps are closed.
\newblock {\em Proceedings of the American Mathematical Society},
  24(4):835--836, 1970.

\bibitem{P1}
John {Pardon}.
\newblock {Contact homology and virtual fundamental cycles}.
\newblock {\em ArXiv e-prints 1508.03873}, August 2015.

\bibitem{P}
John Pardon.
\newblock An algebraic approach to virtual fundamental cycles on moduli spaces
  of pseudo-holomorphic curves.
\newblock {\em Geometry \& Topology}, 20(2):779--1034, 2016.

\bibitem{Qin}
Lizhen Qin.
\newblock On moduli spaces and cw structures arising from morse theory on
  hilbert manifolds.
\newblock {\em Journal of Topology and Analysis}, 02(04):469--526, Dec 2010.

\bibitem{Sa}
Dietmar Salamon.
\newblock Lectures on {F}loer homology.
\newblock {\em Symplectic geometry and topology ({P}ark {C}ity, {UT}, 1997)},
  7:143--229, 1999.

\bibitem{S1}
Paul Seidel.
\newblock Some speculations on pairs-of-pants decompositions and {F}ukaya
  categories.
\newblock {\em Surveys in differential geometry. {V}ol. {XVII}}, 17:411--425,
  2012.

\bibitem{St}
Elias~M. Stein.
\newblock {\em Singular integrals and differentiability properties of
  functions}.
\newblock Princeton Mathematical Series, No. 30. Princeton University Press,
  Princeton, N.J., 1970.

\bibitem{Ti}
D.~Tischler.
\newblock On fibering certain foliated manifolds over {$S\sp{1}$}.
\newblock {\em Topology. An International Journal of Mathematics}, 9:153--154,
  1970.

\bibitem{U}
Umut Varolgunes.
\newblock {\em Mayer-Vietoris property for relative symplectic cohomology}.
\newblock PhD thesis, MIT, 2018.

\bibitem{Ve}
Sara {Venkatesh}.
\newblock {Rabinowitz Floer homology and mirror symmetry}.
\newblock {\em Journal of Topology 11}, pages 144--179, 2018.

\bibitem{V}
C.~Viterbo.
\newblock Functors and computations in {F}loer homology with applications. {I}.
\newblock {\em Geometric and Functional Analysis}, 9(5):985--1033, 1999.

\bibitem{W}
Charles~A. Weibel.
\newblock {\em An introduction to homological algebra}, volume~38 of {\em
  Cambridge Studies in Advanced Mathematics}.
\newblock Cambridge University Press, Cambridge, 1994.

\bibitem{Wh}
Hassler Whitney.
\newblock On the extension of differentiable functions.
\newblock {\em Bulletin of the American Mathematical Society}, 50:76--81, 1944.

\bibitem{Y}
James~A. Yorke.
\newblock Periods of periodic solutions and the {L}ipschitz constant.
\newblock {\em Proceedings of the American Mathematical Society}, 22:509--512,
  1969.

\end{thebibliography}

\end{document}